\documentclass[11pt,reqno]{amsart}

\usepackage{enumerate,tikz-cd,mathtools,amssymb}

\usepackage[a4paper,top=3cm,bottom=2.5cm,left=3cm,right=3cm,marginparwidth=1.75cm]{geometry}

\usepackage{verbatim}
\usepackage{mathrsfs}
\usepackage{array, makecell}
\usepackage{float}
\usepackage{longtable}
\usepackage{graphicx}
\usepackage{caption}
\usepackage{libertine}
\usepackage{xcolor}
\usepackage{tikz}
\usepackage{color}
\usepackage[all]{xy}
\usetikzlibrary{matrix,arrows,decorations.pathmorphing}
\usepackage{courier}
\usepackage{listings}
\allowdisplaybreaks
\usepackage{microtype}
\usepackage{longtable}
\usepackage{tabularx}
\usepackage{subcaption}
\usepackage{multicol}

\definecolor{commentgreen}{RGB}{2,112,10}
\definecolor{eminence}{RGB}{108,48,130}
\definecolor{frenchplum}{RGB}{149,20,83}

\definecolor{ffqqqq}{RGB}{215,25,28}
\definecolor{cccccc}{RGB}{171,217,233}
\definecolor{cite}{RGB}{44,123,182}
\definecolor{ref}{RGB}{215,25,28}

\lstdefinelanguage{Sage}
{
keywords={load, matrix, from, import},
emph={},
}

\usepackage[pdftex,
              pdfauthor={Riccardo Moschetti, Franco Rota, Luca Schaffler},
              pdftitle={},
              pdfsubject={non-degeneracy invariant, Enriques surfaces, computational view},
              pdfkeywords={Enriques surface, elliptic fibration, rational curve, non-degeneracy invariant, Fano polarization, Moschetti,  Rota, Schaffler},
              colorlinks=true,
              citecolor=cite,
              linkcolor=ref 
              ]{hyperref}

\makeatletter
\newtheorem*{rep@theorem}{\rep@title}
\newcommand{\newreptheorem}[2]{
\newenvironment{rep#1}[1]{
 \def\rep@title{#2 \ref{##1}}
 \begin{rep@theorem}}
 {\end{rep@theorem}}}
\makeatother

\theoremstyle{plain}
\newtheorem{theorem}{Theorem}[section]
\newtheorem{corollary}[theorem]{Corollary}
\newtheorem{lemma}[theorem]{Lemma}
\newtheorem{proposition}[theorem]{Proposition}
\newtheorem{conjecture}[theorem]{Conjecture}

\theoremstyle{definition}
\newtheorem{definition}[theorem]{Definition}
\newtheorem{example}[theorem]{Example}

\theoremstyle{remark}
\newtheorem{remark}[theorem]{Remark}

\newcommand{\bdelta}{\boldsymbol\delta}
\newcommand{\bepsilon}{\boldsymbol\epsilon}
\newcommand{\bgamma}{\boldsymbol\gamma}
\newcommand{\bsigma}{\boldsymbol\sigma}

\begin{document}
\title[The non-degeneracy invariant of Brandhorst and Shimada's families of Enriques surfaces]{The non-degeneracy invariant of Brandhorst and Shimada's families of Enriques surfaces}

\author[R. Moschetti]{Riccardo Moschetti}
\address{RM: Department of Mathematics G. Peano, University of Turin, via Carlo Alberto 10, 10123 Torino, Italy} 
\email{riccardo.moschetti@unito.it}

\author[F. Rota]{Franco Rota}
\address{FR: School of Mathematics and Statistics, University of Glasgow, Glasgow G12 8QQ, United Kingdom} 
\email{franco.rota@glasgow.ac.uk}

\author[L. Schaffler]{Luca Schaffler}
\address{LS: Dipartimento di Matematica e Fisica, Universit\`a degli Studi Roma Tre, Largo San Leonardo Murialdo 1, 00146, Roma, Italy}
\email{luca.schaffler@uniroma3.it}


\begin{abstract}
Brandhorst and Shimada described a large class of Enriques surfaces, called $(\tau,\overline{\tau})$-generic, for which they gave generators for the automorphism groups and calculated the elliptic fibrations and the smooth rational curves up to automorphisms. In the present paper, we give lower bounds for the non-degeneracy invariant of such Enriques surfaces, we show that in most cases the invariant has generic value $10$, and we present the first known example of complex Enriques surface with infinite automorphism group and non-degeneracy invariant not equal to $10$.
\end{abstract}

\maketitle

\section{Introduction}

A fundamental feature of an Enriques surface $Y$ is that it always has an elliptic pencil, and the understanding of these gives information about the geometry of $Y$. More specifically, in the current paper we are interested in studying the so-called \emph{non-degeneracy invariant} $\mathrm{nd}(Y)$. This was introduced in \cite{CD89} and it is defined as follows. Every elliptic pencil can be written as $|2F|$, where $F\in \mathrm{Pic}(Y)$ is called a \emph{half-fiber}. Then, the non-degeneracy invariant $\mathrm{nd}(Y)$ is the maximum $m$ such that there exist half-fibers $F_1,\ldots,F_m$ such that $F_i\cdot F_j=1-\delta_{ij}$. Back to geometry, if $\mathrm{nd}(Y)=10$, then $Y$ can be realized as a degree $10$ surface in $\mathbb{P}^5$ given by the intersection of $10$ cubics (see the discussion in \cite[\S\,2.3]{DM19}). In a way, $\mathrm{nd}(Y)$ can be thought of as a way of measuring how far we are from such a projective realization.

In characteristic different from $2$, it is known that $4\leq\mathrm{nd}(Y)\leq10$. While the upper bound simply follows from the fact that $\mathrm{Num}(Y)$, the group of divisors on $Y$ modulo numerical equivalence, has rank $10$, the lower bound is a recent result \cite{MMV22b}. The non-degeneracy invariant is known to be $10$ for Enriques surfaces without smooth rational curves \cite[Theorem~3.2]{Cos85} and for generic Enriques surfaces containing smooth rational curves (see \cite[\S\,4.2]{DM19} and \cite[Lemma~3.2.1]{Cos83}). The non-degeneracy invariant was also computed for Enriques surfaces with finite automorphism group \cite{DK24}, and for specific examples of special Enriques surfaces with smooth rational curves and infinite automorphism group (see \cite[\S\,4.1--\S\,4.3]{Dol18} and \cite[\S\,5]{MRS22Paper}). Otherwise, computing $\mathrm{nd}(Y)$ is in general a hard problem as it is difficult to understand all the elliptic fibrations on $Y$, and the knowledge of their orbits under automorphisms is not sufficient to determine it.

In the present paper, we work over $\mathbb{C}$ (see Remark~\ref{rmk:realizable-ADE-sublattices}) and turn to a specific class of Enriques surfaces which were introduced by Brandhorst and Shimada in \cite{BS22}. These are called \emph{$(\tau,\overline{\tau})$-generic Enriques surfaces}, where $\tau$ is the ADE-type of a lattice spanned by a set of smooth rational curves on $Y$, and $\overline{\tau}$ is the ADE-type of its primitive closure in $\mathrm{Num}(Y)$. The $184$ lattice-theoretic possibilities for all pairs $(\tau,\overline{\tau})$ were classified in \cite{Shi21}. Of these, $155$ are obtained by \emph{realizable} families of $(\tau,\overline{\tau})$-generic Enriques surfaces (see Definition~\ref{def:realizable-Enriques}). For these families, Brandhorst and Shimada study $\mathrm{Aut}(Y)$ and the sets of smooth rational curves and elliptic fibrations on $Y$ up to automorphisms. Combining this information with \cite{MRS22Paper,MRS22Code}, we prove the following result. In the statement we do not include the values of the non-degeneracy invariants for the families $1,172,184$, since they are already known, see Remark~\ref{rmk:nd-for1-172-184}.

\begin{theorem}[Theorem~\ref{thm:nds-of-tautaubar-generic-Enriques-surfaces} and Theorem~\ref{thm:nd-for-Enriques-145}]
\label{thm:main}
For an integer $i\in \{1,\ldots,184\}$, let $Y_i$ be the $i$-th realizable $(\tau,\overline{\tau})$-generic Enriques surface in \cite[Table~1]{BS22} with $i\neq1,172,184$. Then,
\begin{enumerate}

\item $\mathrm{nd}(Y_{145})=4$.

\item We have the lower bounds

\begin{table}[H]
\begin{tabular}{cccc}
$\mathrm{nd}(Y_{84})\geq9$, & $\mathrm{nd}(Y_{85})\geq7$, & $\mathrm{nd}(Y_{121})\geq9$, & $\mathrm{nd}(Y_{122})\geq7$,\\
$\mathrm{nd}(Y_{123})\geq7$, & $\mathrm{nd}(Y_{143})\geq8$, & $\mathrm{nd}(Y_{144})\geq8$, & $\mathrm{nd}(Y_{158})\geq9$,\\
$\mathrm{nd}(Y_{159})\geq7$, & $\mathrm{nd}(Y_{171})\geq8$, & $\mathrm{nd}(Y_{176})\geq7$. &
\end{tabular}
\end{table}

\item Finally, in the remaining $140$ cases, $\mathrm{nd}(Y_i)=10$.

\end{enumerate}
\end{theorem}

For each surface $Y_i$ in the statement of Theorem~\ref{thm:main}, we provide an explicit sequence of half-fibers on $Y_i$ realizing the claimed lower bound for $\mathrm{nd}(Y_i)$, see Section~\ref{App:TableNonDeg} for a reference. For the Enriques surface $Y_{145}$, using the work of Brandhorst and Shimada, the automorphism group of $Y_{145}$ acts on cohomology as the infinite dihedral group. We give a full combinatorial description of its action on the set of smooth rational curves. We then assemble this information to compute the value of the non-degeneracy invariant. The details of the proof are spelled out in \S\,\ref{sec:145}. To our knowledge, $Y_{145}$ is the first example of Enriques surface with infinite automorphism group and $\mathrm{nd}(Y)<10$. Different aspects of the Enriques surfaces $Y_{145}$ have been studied in the literature: they appear in \cite{BP83} as examples with ``small'' infinite automorphism group, their K3 covers have zero entropy \cite[Remark~5.5]{BM22}, and some non-extendable isotropic sequences on them were computed in \cite[\S\,4]{MMV22a}.

The strategy we used to show that $\mathrm{nd}(Y_{145})=4$ applies in principle to the remaining $11$ cases in Theorem~\ref{thm:main}~(2). However, the argument is harder to replicate: both the automorphism groups and the sets of orbit representatives of smooth rational curves are more complex. This will be object of future study.

\subsection*{Acknowledgements}
We would like to thank Simon Brandhorst, Igor Dolgachev, Shigeyuki Kond\=o, Gebhard Martin, Giacomo Mezzedimi, Ichiro Shimada, and Davide Cesare Veniani for helpful conversations. We also thank the anonymous referee for useful insights about the Enriques surface $Y_{145}$. This research was partially funded by the project GNSAGA - INdAM ``Numerical Invariants of Nodal Enriques Surfaces''. During the preparation of the paper, the first author was partially supported by PRIN 2017 Moduli and Lie theory, and by MIUR: Dipartimenti di Eccellenza Program (2018-2022) - Dept. of Math. Univ. of Pavia. The second author is supported by EPSRC grant EP/R034826/1 and acknowledges support from the Deutsche Forschungsgemeinschaft (DFG, German Research Foundation) under Germany's Excellence Strategy -- EXC-2047/1 -- 390685813. The third author was supported by ``Programma per Giovani Ricercatori Rita Levi Montalcini'' and partially supported by the projects PRIN2017 ``Advances in Moduli Theory and Birational Classification'' and PRIN2020 ``Curves, Ricci flat varieties and their Interactions''. R. Moschetti and L. Schaffler are members of the GNSAGA - INdAM and were both partially supported by the project PRIN 2022 ``Moduli Spaces and Birational Geometry'' -- CUP E53D23005790006.

\section{Preliminaries on Enriques surfaces and the non-degeneracy invariant}

\subsection{Enriques surfaces and the \texorpdfstring{$E_{10}$}{Lg} lattice}

An \emph{Enriques surface} $Y$ is a smooth minimal projective connected algebraic surface of Kodaira dimension zero such that $h^0(Y,\omega_Y)=h^1(Y,\mathcal{O}_Y)=0$. For an Enriques surface, the canonical class $K_Y$ is the unique torsion element of $\mathrm{Pic}(Y)$ (more specifically, $2K_Y\sim0$), and after quotienting by it we obtain $S_Y:=\mathrm{Num}(Y)$, i.e. the group of divisors on $Y$ modulo numerical equivalence. The intersection product among curves endows $S_Y$ with the structure of a \emph{lattice}: a finitely generated free abelian group $L$ of finite rank together with a non-degenerate symmetric bilinear form $b\colon L\times L\rightarrow\mathbb{Z}$. As a lattice, $S_Y$ is isometric to $U\oplus E_8$, where $U$ is the hyperbolic lattice $\left(\mathbb{Z}^2,\left(\begin{smallmatrix}0&1\\1&0\end{smallmatrix}\right)\right)$ and $E_8$ is the negative definite root lattice associated to the corresponding Dynkin diagram. We review root lattices more in detail in \S\,\ref{sec:root-lattices-and-tau-taubar-invariant}, but first we recall an alternative realization of the lattice $U\oplus E_8$.

\begin{definition}
The \emph{$E_{10}$ lattice} is defined to be $\mathbb{Z}^{10}$ together with the intersection form associated to the canonical basis $e_1,\ldots,e_{10}$ as represented in Figure~\ref{fig:E10}: $e_i^2=-2$ and $e_i\cdot e_j=1$ if the corresponding vertices are joined by an edge, and zero otherwise. A direct check shows that $E_{10}$ is even, unimodular, and of signature $(1,9)$ (see \cite[Chapter~V]{Ser73} for this standard terminology). It follows by \cite[Theorem~1]{Mil58} that $E_{10}$ is isometric to $U\oplus E_8$. Therefore, for any Enriques surface $Y$, $S_Y$ is isometric to $E_{10}$.
\end{definition}

In the current paper, it will be crucial to work with $\mathbb{Z}$-bases of $S_Y$ in the following form.

\begin{definition}
Let $Y$ be an Enriques surface and let $\mathcal{B}=\{B_1,\ldots,B_{10}\}$ be a basis of $S_Y$. We call $\mathcal{B}$ an \emph{$E_{10}$-basis} if the map $B_i\mapsto e_i$ extends to an isometry between $S_Y$ and $E_{10}$.
\end{definition}

\begin{figure}
\begin{tikzpicture}

\draw  (0,0)-- (1,0);
\draw  (1,0)-- (2,0);
\draw  (2,0)-- (3,0);
\draw  (3,0)-- (4,0);
\draw  (4,0)-- (5,0);
\draw  (5,0)-- (6,0);
\draw  (6,0)-- (7,0);
\draw  (7,0)-- (8,0);
\draw  (2,0)-- (2,1);

\fill [color=black] (0,0) circle (2.4pt);
\fill [color=black] (1,0) circle (2.4pt);
\fill [color=black] (2,0) circle (2.4pt);
\fill [color=black] (3,0) circle (2.4pt);
\fill [color=black] (4,0) circle (2.4pt);
\fill [color=black] (5,0) circle (2.4pt);
\fill [color=black] (6,0) circle (2.4pt);
\fill [color=black] (7,0) circle (2.4pt);
\fill [color=black] (8,0) circle (2.4pt);
\fill [color=black] (2,1) circle (2.4pt);

\draw[color=black] (2.4,1) node {$e_1$};
\draw[color=black] (0,-0.4) node {$e_2$};
\draw[color=black] (1,-0.4) node {$e_3$};
\draw[color=black] (2,-0.4) node {$e_4$};
\draw[color=black] (3,-0.4) node {$e_5$};
\draw[color=black] (4,-0.4) node {$e_6$};
\draw[color=black] (5,-0.4) node {$e_7$};
\draw[color=black] (6,-0.4) node {$e_8$};
\draw[color=black] (7,-0.4) node {$e_9$};
\draw[color=black] (8,-0.4) node {$e_{10}$};

\end{tikzpicture}
\caption{The $E_{10}$ lattice.}
\label{fig:E10}
\end{figure}
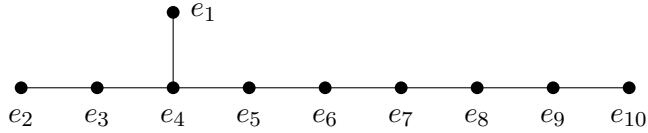

\subsection{The non-degeneracy invariant of an Enriques surface}

Given an elliptic fibration $f\colon Y\rightarrow\mathbb{P}^1$ on an Enriques surface $Y$, then $f$ has exactly two multiple fibers $2F$ and $2F'$ \cite{BHPV04}. The curves $F,F'$ are called the \emph{half-fibers} of the elliptic fibration. The half-fibers on an Enriques surface can be used to define the following invariant of $Y$.

\begin{definition}
Let $Y$ be an Enriques surface and let $m$ be the maximum for which there exist half-fibers $F_1,\ldots,F_m$ on $Y$ such that $F_i\cdot F_j=1-\delta_{ij}$. Then $m$ is called the \emph{non-degeneracy invariant} of $Y$ and it is denoted by $\mathrm{nd}(Y)$.
\end{definition}

As we briefly reviewed in the Introduction, if the Enriques surface $Y$ is unnodal, or general nodal, then $\mathrm{nd}(Y)=10$. Otherwise, it is in general a hard problem to compute $\mathrm{nd}(Y)$, which is known only in few cases (for instance, the Enriques surfaces with finite automorphism group). Motivated by this, in \cite{MRS22Paper} we introduced a combinatorial version of the non-degeneracy invariant which we now review. First we recall that an \emph{elliptic configuration} on $Y$ is a curve $C$ appearing in Kodaira's classification of singular fibers in elliptic fibrations \cite[Chapter~V, Table~3]{BHPV04}, but not a multiple of such a fiber (note that $C$ does not have to be primitive in $S_Y$). In this case, we have that either $|C|$ is an elliptic pencil or $|2C|$ is an elliptic pencil of which $C$ is a half-fiber \cite[Chapter~VIII, Lemma~17.3]{BHPV04}.

\begin{definition}
\label{def:relative-cnd-invariant}
Let $\mathcal{R}$ be a finite collection of smooth rational curves on an Enriques surface $Y$. Let $\mathsf{HF}(Y,\mathcal{R})$ be the set of numerical equivalence classes $h\in S_Y$ in the following form:
\begin{itemize}

\item $h=[C]$, where $C$ is an elliptic configurations with irreducible components in $\mathcal{R}$ and $C$ is a half-fiber of an elliptic pencil on $Y$;

\item $h=\frac{1}{2}[C]$, where $C$ is an elliptic configurations with irreducible components in $\mathcal{R}$ and $C$ is a fiber of an elliptic pencil on $Y$.

\end{itemize}
Then we define the \emph{combinatorial non-degeneracy invariant of $Y$ relative to $\mathcal{R}$}, $\mathrm{cnd}(Y,\mathcal{R})$, to be the maximum $m$ such that there exist $f_1,\ldots,f_m\in \mathsf{HF}(S,\mathcal{R})$ such that $f_i\cdot f_j=1-\delta_{ij}$.
\end{definition}

As $\mathrm{cnd}(Y,\mathcal{R})$ only considers classes of half-fibers supported in $\mathcal{R}$, it gives a lower bound for $\mathrm{nd}(Y)$. In \cite{MRS22Code} we implemented a code that computes $\mathrm{cnd}(Y,\mathcal{R})$ by listing all elliptic configurations supported in $\mathcal{R}$ and then, with a recursive procedure, finds all the maximal sequences $f_1,\ldots,f_m\in \mathsf{HF}(S,\mathcal{R})$ satisfying $f_i\cdot f_j=1-\delta_{ij}$. The procedure terminates as $\mathcal{R}$ is a finite set of smooth rational curves.

\subsection{The combinatorial non-degeneracy invariant \texorpdfstring{$\mathrm{cnd}(Y)$}{Lg}}

We now introduce an additional version of non-degeneracy invariant which explores the behavior of $\mathrm{cnd}(Y,\mathcal{R})$ as $\mathcal{R}$ varies.

\begin{definition}
Let $Y$ be an Enriques surface. We define the \emph{combinatorial non-degeneracy invariant of $Y$} as follows:
\[
\mathrm{cnd}(Y)=\max\{\mathrm{cnd}(Y,\mathcal{R})\mid\mathcal{R}\subseteq\mathcal{R}(Y)~\textrm{is finite}\}.
\]
\end{definition}

\begin{remark}
It is an immediate consequence of the definition that we have the following order relations among the different invariants we introduced above:
\[
\mathrm{cnd}(Y,\mathcal{R})\leq\mathrm{cnd}(Y)\leq\mathrm{nd}(Y).
\]
\end{remark}

The next proposition pinpoints a (nonempty) class of Enriques surfaces where the invariants $\mathrm{cnd}(Y)$ and $\mathrm{nd}(Y)$ are equal.

\begin{proposition}
\label{prop:CndEqualNd}
Let $Y$ be an Enriques surface such that every elliptic fibration admits a fiber or a half-fiber supported on the union of smooth rational curves. Then we have that $\mathrm{nd}(Y)=\mathrm{cnd}(Y)$.
\end{proposition}

\begin{proof}
Let $m=\mathrm{nd}(Y)$ and let $f_1=[F_1],\ldots,f_m=[F_m]$ be numerical equivalence classes of half-fibers on $Y$ such that $f_i\cdot f_j=1-\delta_{ij}$. By hypothesis, for all $i\in\{1,\ldots,m\}$, $|2F_i|=|nC_i|$ where $\mathrm{Supp}(C_i)=\cup_jR_j^{(i)}$ for some smooth rational curves $R_j^{(i)}$ and $n=1$ or $2$ if $C_i$ is a fiber or a half-fiber respectively. Define $\mathcal{R}=\{R_j^{(i)}\mid i,j\}$. Then
\[
\mathrm{nd}(Y)=m\leq\mathrm{cnd}(Y,\mathcal{R})\leq\mathrm{cnd}(Y)\leq\mathrm{nd}(Y)\implies\mathrm{nd}(Y)=\mathrm{cnd}(Y).\qedhere
\]
\end{proof}

We now give an example where $\mathrm{nd}(Y)$ does not coincide with $\mathrm{cnd}(Y)$.

\begin{proposition}
\label{prop:example-cnd<nd}
    Let $Y$ be a general nodal Enriques surface. Then  $\mathrm{cnd}(Y)<10$. 
\end{proposition}

\begin{proof}
Suppose $(f_1,\ldots,f_{10})$ is an isotropic sequence realizing $\mathrm{cnd}(Y)=10$. By the definition of $\mathrm{cnd}(Y)$ and \cite[Theorem~6.5.5~(ii)]{DK24}, it is necessary that every pencil $|2F_i|$ has a reducible fiber of type $\widetilde{A}_1$. So, let us write $2f_i= R_i+S_i$ with $R_i$ and $S_i$ smooth rational curves.

The first step is to show that for all $i,j\in\{1,\ldots,10\}$, $i\neq j$, the fibers $R_i+S_i$ and $R_j+S_j$ have a component in common. By contradiction, assume otherwise. Then, we have that
\[
f_i\cdot f_j=1\implies 4=(R_i+S_i)\cdot(R_j+S_j)= R_i\cdot R_j+R_i\cdot S_j+S_i\cdot R_j+S_i\cdot S_j.
\]
Since no two $(-2)$-curves on $Y$ are disjoint by \cite[Corollary~6.5.2]{DK24}, then the intersection products $R_i\cdot R_j,R_i\cdot S_j,S_i\cdot R_j,S_i\cdot S_j$ equal $1$, which contradicts \cite[Lemma~6.5.1]{DK24}.

In the second step of the proof, we show that there exists a smooth rational curve $R_0$ which is an irreducible component of $R_i+S_i$ for all $i$. Let us start by considering $R_1+S_1$ and $R_2+S_2$. As they have a common irreducible component by the first step, we can assume up to relabeling that $R_0\coloneqq S_1=S_2$. Let us consider $R_i+S_i$, $i\in\{3,\ldots,10\}$, and assume by contradiction that $R_0\neq R_i,S_i$. Then, by the first step, we must have that $R_i+S_i=R_1+R_2$. Up to relabeling, we may assume that $R_i=R_1$ and $S_i=R_2$. Now, consider $R_j+S_j$, $j\in\{3,\ldots,10\}\setminus\{i\}$. As $R_j+S_j$ has a component in common with $R_i+S_i=R_1+R_2$, then we can assume up to relabeling that $R_j=R_1$. At the same time, $R_j+S_j$ has a component in common with $R_2+R_0$, which is impossible because $R_j+S_j$ would equal $R_1+R_2$ or $R_1+R_0$, contradicting the fact that the $10$ fibers of type $\widetilde{A}_1$ that we fixed are distinct. So, $R_0$ is a component of $R_i+S_i$. Summarizing, we can write for all $i$ that
\[
R_i+S_i = R_i + R_0.
\]

Finally, we study the intersection matrix $\mathfrak{R}=(R_i\cdot R_j)_{0\leq i,j\leq10}$. For $i=1,\ldots,10$, we have that $R_0\cdot R_i=2$ as $R_0+R_i$ is an elliptic configuration of type $\widetilde{A}_1$. Lastly, for $i,j\neq0$, $i\neq j$, we have that
\[
4=(R_0+R_i)\cdot(R_0+R_j)=2+R_i\cdot R_j\implies R_i\cdot R_j=2.
\]
It can be checked directly that the matrix $\mathfrak{R}$ has rank $11$, which implies that $R_0,\ldots,R_{10}$ generate a sublattice of $S_Y$ of rank $11$, which cannot be.
\end{proof}

\section{Smooth rational curves and automorphisms of Enriques surfaces}

From the discussion so far, it emerged that the more we know about smooth rational curves on an Enriques surface $Y$, the more we understand $\mathrm{nd}(Y)$. In \cite{BS22} Brandhorst and Shimada studied the distribution of smooth rational curves on $Y$ in relation to the automorphism group $\mathrm{Aut}(Y)$. Below we recall some aspects of their work.

\subsection{Root lattices and the \texorpdfstring{$(\tau,\overline{\tau})$}{Lg}-generic Enriques surfaces}
\label{sec:root-lattices-and-tau-taubar-invariant}

We follow the exposition in \cite[\S\,1.1]{BS22}. An \emph{ADE-lattice} is an even, negative definite lattice $R$ generated by \emph{roots}, i.e. vectors $v\in R$ such that $v^2=-2$. It is well known that an ADE-lattice $R$ has a basis consisting of roots whose associated dual graph is the disjoint union of some of the Dynkin diagrams $A_n$ ($n\geq1$), $D_n$ ($n\geq4$), and $E_6,E_7,E_8$. This \emph{ADE-type} for the lattice $R$ is denoted by $\tau(R)$.

In \cite{Shi21} Shimada classified the ADE-sublattices of $E_{10}$ up to the action of $O^{\mathcal{P}}(E_{10})$, which is the group of isometries of $E_{10}$ which preserve a positive half-cone $\mathcal{P}$, that is one of the two connected components of $\{v\in E_{10}\mid v\cdot v>0\}$.

\begin{theorem}[{\cite{Shi21}}]
The following hold:
\begin{enumerate}

\item Let $R_1,R_2\subseteq E_{10}$ be two ADE-lattices. Denote by $\overline{R}_1,\overline{R}_2$ their respective primitive closures in $E_{10}$. Then also $\overline{R}_1,\overline{R}_2$ are ADE-lattices and $(\tau(R_1),\tau(\overline{R}_1))=(\tau(R_2),\tau(\overline{R}_2))$ if and only if $R_1$ and $R_2$ are in the same $O^{\mathcal{P}}(E_{10})$-orbit.

\item Let $R\subseteq E_{10}$ be an ADE-sublattice. Then there are $184$ possibilities for the pairs $(\tau(R),\tau(\overline{R}))$. These are listed in \cite[Table~1]{Shi21} (see also \cite[Table~1]{BS22}).

\end{enumerate}
\end{theorem}

\begin{remark}
\label{rmk:realizable-ADE-sublattices}
Given an ADE-sublattice $R\subseteq E_{10}$, it is natural to ask whether there exists an Enriques surface $Y$ together with a configuration of smooth rational curves $C_1,\ldots,C_n\subseteq Y$ whose dual graph equals $\tau(R)$. By \cite[Corollary~1.8]{Shi21}, we have that $\tau(R)$ is realized in this way on a complex Enriques surface if and only if the fourth column of \cite[Table~1]{Shi21} does not have the symbol ``$-$'', which occurs in $175$ cases out of the $184$ possibilities for $(\tau(R),\tau(\overline{R}))$. As our goal is to study these Enriques surfaces, we work over $\mathbb{C}$.
\end{remark}

We will be interested in understanding the geometry of nodal Enriques surfaces $Y$ whose universal K3 cover satisfies specific conditions with respect to a configuration of smooth rational curves on $Y$ generating a sublattice of $S_Y$ of fixed ADE-type. Recall that for a lattice $(L,b)$ and a positive integer $m$, $L(m)$ denotes the lattice $(L,mb)$.

\begin{definition}
Let $Y$ be an Enriques surface with universal K3 cover $X\rightarrow Y$. Let $(\tau,\overline{\tau})$ one of the pairs in \cite[Table~1]{BS22}. Then $Y$ is called \emph{$(\tau,\overline{\tau})$-generic} provided the following hold:
\begin{enumerate}

\item Consider $H^{2,0}(X)\subseteq T_X\otimes\mathbb{C}$, where $T_X$ denotes the transcendental lattice of the K3 surface $X$. Then the group of isometries of $T_X$ preserving $H^{2,0}(X)$ is equal to $\{\pm\mathrm{id}_{T_X}\}$.

\item Let $R\subseteq E_{10}$ be an ADE-sublattice such that $(\tau(R),\tau(\overline{R}))=(\tau,\overline{\tau})$. Define $M_R$ to be the sublattice of $E_{10}(2)\oplus R(2)$ given by
\[
\langle(v,0),(w,\pm w)/2\mid v\in E_{10},~w\in R\rangle.
\]
Then there exist isometries $M_R\cong S_X$ and $E_{10}\cong S_Y$ such that the following diagram commutes:

\begin{center}
\begin{tikzpicture}
\matrix(a)[matrix of math nodes,
row sep=2em, column sep=2em,
text height=1.5ex, text depth=0.25ex]
{E_{10}(2)&M_R\\
S_Y(2)&S_X.\\};
\path[right hook->] (a-1-1) edge node[]{}(a-1-2);
\path[->] (a-1-1) edge node[left]{$\cong$}(a-2-1);
\path[right hook->] (a-2-1) edge node[]{}(a-2-2);
\path[->] (a-1-2) edge node[right]{$\cong$}(a-2-2);
\end{tikzpicture}
\end{center}

\end{enumerate}

\end{definition}

\begin{definition}
\label{def:realizable-Enriques}
Among the $175$ cases in Remark~\ref{rmk:realizable-ADE-sublattices} of $(\tau,\overline{\tau})$ which can be realized geometrically by smooth rational curves on an Enriques surface, $155$ of them are $(\tau,\overline{\tau})$-generic. These are the cases not marked with ``$\times$'' in the fifth column of \cite[Table~1]{BS22}. We will focus on these Enriques surfaces, which we will refer to as \emph{realizable} $(\tau,\overline{\tau})$-generic Enriques surfaces.
\end{definition}

\subsection{Automorphisms and smooth rational curves on \texorpdfstring{$(\tau,\overline{\tau})$}{Lg}-generic Enriques surfaces}

Let $Y$ be a $(\tau,\overline{\tau})$-generic Enriques surface. Consider the natural homomorphism
\[
\mathrm{Aut}(Y)\rightarrow O^{\mathcal{P}}(S_Y)
\]
and denote by $\mathrm{aut}(Y)$ its image. One of the main results in \cite{BS22} is the computation of a finite generating set for $\mathrm{aut}(Y)$ and the study of its action on $\mathcal{R}(Y)$, which denotes the set of smooth rational curves on $Y$. The construction is quite technical \cite[\S\,6.1]{BS22}, but for our purposes it suffices to focus on two of its main ingredients: the sets $\mathcal{R}_{\mathrm{temp}}$ and $\mathcal{H}$. The former set $\mathcal{R}_{\mathrm{temp}}$ is a specific subset of $\mathcal{R}(Y)$ with the property that the composition
\[
\mathcal{R}_{\mathrm{temp}}\hookrightarrow\mathcal{R}(Y)\twoheadrightarrow\mathcal{R}(Y)/\mathrm{aut}(Y)
\]
is surjective, while $\mathcal{H}$ is a subset of $\mathrm{aut}(Y)$ that acts in a specific way on a chamber decompositions of the nef cone of $Y$. We focus on $\mathcal{R}_{\mathrm{temp}}$ because it contains explicit examples of smooth rational curves on $Y$ that we can use to study $\mathrm{nd}(Y)$, and we can apply the automorphisms in $\mathcal{H}$ to $\mathcal{R}_{\mathrm{temp}}$ to obtain more smooth rational curves if needed. Moreover, the elements in $\mathcal{R}_{\mathrm{temp}}$ and $\mathcal{H}$ are conveniently described in \cite{SB20data} in terms of a fixed $E_{10}$-basis of $S_Y$. So, a curve in $\mathcal{R}_{\mathrm{temp}}$ is an integral vector of dimension $10$ and an automorphism in $\mathcal{H}$ is a $10\times10$ matrix with entries in $\mathbb{Z}$. In the next example we explain how to extract this information from \cite{SB20data} (see also \cite{SB20explanation}).

\begin{example}
\label{ex:how-to-extract-the-data}
Consider the $(\tau,\overline{\tau})$-generic Enriques surface number $145$. To find the data corresponding to $\mathcal{R}_{\mathrm{temp}}$ and $\mathcal{H}$ in the file \cite[Enrs.txt]{SB20data}, we can first search for \texttt{no := 145}. In the corresponding record named \texttt{Rats}, which is before the label $145$, we can find \texttt{Ratstemp}, and then different records named \texttt{rat}. For each one of these, we consider \texttt{ratY}, which give the following vectors:
\begin{align*}
R_0:=(4, 2, 4, 6, 5, 4, 3, 2, 1, 0),~R_1:=(2, 2, 3, 4, 3, 2, 1, 0, 0, 0),\\
R_2:=(0, 0, 0, 0, 0, 0, 0, 0, 0, 1),~R_3:=(0, 0, 0, 0, 0, 0, 0, 0, 1, 0),\\
R_4:=(0, 0, 0, 0, 0, 0, 0, 1, 0, 0),~R_5:=(0, 0, 0, 0, 0, 0, 1, 0, 0, 0),\\
R_6:=(0, 0, 0, 0, 0, 1, 0, 0, 0, 0),~R_7:=(0, 0, 0, 0, 1, 0, 0, 0, 0, 0),\\
R_8:=(0, 0, 1, 0, 0, 0, 0, 0, 0, 0),~R_9:=(0, 0, 0, 1, 0, 0, 0, 0, 0, 0).
\end{align*}
This means that, for a fixed $E_{10}$-basis $\{B_1,\ldots,B_{10}\}$ of $S_Y$, the vectors
\begin{gather*}
4B_1+2B_2+4B_3+6B_4+5B_5+4B_6+3B_7+2B_8+B_9,\\
2B_1+2B_2+3B_3+4B_4+3B_5+2B_6+B_7,~B_{10},~B_9,~B_8,~B_7,~B_6,~B_5,~B_3,~B_4.
\end{gather*}
are numerically equivalent to smooth rational curve on $Y$, giving the curves in $\mathcal{R}_{\mathrm{temp}}$. The record \texttt{HHH} encodes the automorphisms in $\mathcal{H}$, whose action on $S_Y$ is described by the different \texttt{gY}. For the Enriques surface $145$ there are two such \texttt{gY}, and these have to be thought of as $10\times10$ matrices in the $E_{10}$-basis $\{B_1,\ldots,B_{10}\}$.
\end{example}

\section{The non-degeneracy invariant for \texorpdfstring{$(\tau,\overline{\tau})$}{Lg}-generic Enriques surfaces}

\subsection{The main result}

Using \cite{MRS22Paper,MRS22Code} we compute a lower bound, and in the majority of the cases the exact value, for the non-degeneracy invariant of the realizable $(\tau,\overline{\tau})$-generic Enriques surfaces in \cite[Table~1]{BS22}.

\begin{remark}
\label{rmk:nd-for1-172-184}
Let $Y_i$ be the $i$-th realizable $(\tau,\overline{\tau})$-generic Enriques surface in \cite[Table~1]{BS22}. By \cite[Section~7.4]{BS22}, the Enriques surfaces $Y_{172}$ and $Y_{184}$ have finite automorphism group. These Enriques surfaces were classified in \cite{Kon86}, and more precisely we have that $Y_{172}$ and $Y_{184}$ are respectively of type I and II. For these Enriques surfaces, the non-degeneracy invariants were computed in \cite[Propositions~8.9.6 and 8.9.9]{DK24} (see also \cite[Table~3]{MRS22Paper}). The Enriques surface $Y_1$ is a general nodal Enriques surfaces by combining \cite[\S\,6.5]{BS22} and \cite[Theorem~6.5.5~(ii)]{DK24}, and thus we have $\mathrm{nd}(Y_1)=10$. This leaves us with $152$ cases of $\mathrm{nd}(Y_i)$ to compute.
\end{remark}

\begin{theorem}
\label{thm:nds-of-tautaubar-generic-Enriques-surfaces}
Let $Y_i$ be the $i$-th realizable $(\tau,\overline{\tau})$-generic Enriques surface in \cite[Table~1]{BS22} with $i\neq1,172,184$. Then,
\begin{enumerate}

\item $\mathrm{nd}(Y_{145})=4$.

\item We have that

\begin{table}[H]
\begin{tabular}{cccc}
$\mathrm{nd}(Y_{84})\geq9$, & $\mathrm{nd}(Y_{85})\geq7$, & $\mathrm{nd}(Y_{121})\geq9$, & $\mathrm{nd}(Y_{122})\geq7$,\\
$\mathrm{nd}(Y_{123})\geq7$, & $\mathrm{nd}(Y_{143})\geq8$, & $\mathrm{nd}(Y_{144})\geq8$, & $\mathrm{nd}(Y_{158})\geq9$,\\
$\mathrm{nd}(Y_{159})\geq7$, & $\mathrm{nd}(Y_{171})\geq8$, & $\mathrm{nd}(Y_{176})\geq7$. &
\end{tabular}
\end{table}
Moreover, in the above $11$ cases, $\mathrm{nd}(Y_i)=\mathrm{cnd}(Y_i)$.

\item Finally, in the remaining $140$ cases we have that $\mathrm{nd}(Y_i)=10$.

\end{enumerate}
\end{theorem}

\begin{proof}
The value $\mathrm{nd}(Y_{145})=4$ is computed in \S\,\ref{sec:145} (in particular, see Theorem~\ref{thm:nd-for-Enriques-145}). In the remaining cases where $\mathrm{nd}(Y_i)\geq m$, we provide an explicit non-degenerate isotropic sequence of length $m$ (if $m=10$, then we obtain the $141$ cases in (3)). To find such sequences we use the code \cite{MRS22Code} with input the smooth rational curves in $\mathcal{R}_{\mathrm{temp}}$ provided by \cite{SB20data}, and in some case $\mathcal{R}_{\mathrm{temp}}$ together with some curves in $\mathcal{R}_{\mathrm{temp}}\cdot\mathrm{aut}(Y_i)$ (see \S\,\ref{sec:produce-output} for more detail). The output, which can be tested directly without using again \cite{MRS22Code} (see \S\,\ref{sec:test-output} for more detail), is discussed in Section~\ref{App:TableNonDeg}. Finally, the claim in (2) that $\mathrm{nd}(Y_i)=\mathrm{cnd}(Y_i)$ for the eleven examples mentioned follows by Proposition~\ref{prop:CndEqualNd} combined with the fact that every elliptic fibration on $Y_i$ admits a fiber or a half-fiber supported on the union of smooth rational curves, as it can be argued by \cite{SB20ellipticfibrations}.
\end{proof}

\begin{remark}
It follows immediately from Theorem~\ref{thm:nds-of-tautaubar-generic-Enriques-surfaces} that there are no $(\tau,\overline{\tau})$-generic Enriques surfaces with non-degeneracy invariant equal to $5$ or $6$. To our knowledge, examples of Enriques surfaces with non-degeneracy invariant $5,6$, or $9$ are not known.
\end{remark}

\subsection{Computer-based construction of the isotropic sequences}
\label{sec:produce-output}

Fix a $(\tau,\overline{\tau})$-generic Enriques surface $Y_i$. To compute a lower bound for $\mathrm{nd}(Y_i)$, the program \cite{MRS22Code} needs as input a basis for $S_{Y_i}$ and a finite collection $\mathcal{R}$ of smooth rational curves on $Y_i$ (for the explanation of the algorithm see \cite[\S\,4]{MRS22Paper}). Following \cite{BS22}, we take an $E_{10}$-basis for $S_{Y_i}$ and $\mathcal{R}=\mathcal{R}_{\mathrm{temp}}$ if $i\neq43,78,84,121,158$ (recall the smooth rational curves in $\mathcal{R}_{\mathrm{temp}}$ are already expressed in the $E_{10}$ basis in \cite{SB20data}). For $i=43,78,84,121,158$, we obtain the claimed lower bound on $\mathrm{nd}(Y_i)$ by taking $\mathcal{R}$ equal to $\mathcal{R}_{\mathrm{temp}}$ union a finite subset of $\mathcal{R}_{\mathrm{temp}}\cdot\mathrm{aut}(Y_i)$. More precisely, we consider the action of the automorphisms in $\mathcal{H}$ (see Example~\ref{ex:how-to-extract-the-data}).

\begin{remark}
We collect some subtleties behind the above calculations.
\begin{enumerate}

\item In the cases where $\mathrm{nd}(Y_i)=10$, it is sometimes enough to consider a proper subset of curves in $\mathcal{R}_{\mathrm{temp}}$ to construct a non-degenerate isotropic sequence of length $10$.

\item In Theorem~\ref{thm:nds-of-tautaubar-generic-Enriques-surfaces}, for the cases where $\mathrm{nd}(Y_i)\geq m$ with $m\neq10$ and $i\neq 43,78,84$, we used all the curves in $\mathcal{R}_{\mathrm{temp}}$. We also used some of the automorphisms in $\mathcal{H}$ to produce new smooth rational curves not in $\mathcal{R}_{\mathrm{temp}}$, but this did not improve the lower bound $m$.

\item The code \cite{MRS22Code} needs that the chosen smooth rational curves in $\mathcal{R}$ generate $S_{Y_i}$ over $\mathbb{Q}$. We have that $\mathcal{R}_{\mathrm{temp}}$ has this property, except for $i=84$.

\item For $i=43,78,84, 121, 158$, by only using $\mathcal{R}=\mathcal{R}_{\mathrm{temp}}$ we obtain the lower bounds $\mathrm{nd}(Y_{43})\geq6,\mathrm{nd}(Y_{78})\geq7,\mathrm{nd}(Y_{84})\geq4, \mathrm{nd}(Y_{121})\geq6, \mathrm{nd}(Y_{158})\geq8$, instead of the better $10, 10, 9, 9, 9$, respectively.

\end{enumerate}
\end{remark}

\subsection{How to test the correctness of the output}
\label{sec:test-output}

One way of proving parts (2) and (3) of Theorem~\ref{thm:nds-of-tautaubar-generic-Enriques-surfaces} is by directly checking the output discussed in Section~\ref{App:TableNonDeg}. More precisely, let $\mathrm{nd}(Y_k)\geq m$ be one of the inequalities claimed in (2) or (3) of Theorem~\ref{thm:nds-of-tautaubar-generic-Enriques-surfaces} and let $F_1,\ldots,F_m$ be the corresponding sequence of half-fibers we want to check. Using the data \cite{SB20data}, we can find the coordinates of each smooth rational curve in the support of $F_1,\ldots,F_m$ with respect to the fixed $E_{10}$ basis, and consequently express each half-fiber $F_1,\ldots,F_m$ in this basis. (Sometimes, smooth rational curves are obtained by acting on some $R_i\in\mathcal{R}_{\mathrm{temp}}$ with automorphisms in $\mathcal{H}=\{H_0,H_1,\ldots\}$). We point out that the obtained vectors have integral entries, which guarantees that $F_1,\ldots,F_m\in S_{Y_k}$, and that the entries are not simultaneously divisible by $2$, which guarantees that there are no fibers among $F_1,\ldots,F_m$. By organizing the $10$-dimensional vectors as rows of a matrix $F$, we can check the following equality by simple matrix multiplication:
\[
FM_{E_{10}}F^\intercal=\mathbf{1}_m-\mathbf{I}_{m}.
\]
Here, $\mathbf{1}_m$ is the $m\times m$ matrix with entries equal to $1$ and $\mathbf{I}_m$ is the $m\times m$ identity matrix. This means that $F_i\cdot F_j=1-\delta_{ij}$ for all $i,j$. The last thing to verify is that $F_1,\ldots,F_m$ are actually half-fibers. For this, let $C_i$ be the elliptic configuration associated to $F_i$. That is, if $F_i$ has a coefficient $1/2$, then $C_i=2F_i$. Otherwise, $C_i=F_i$. Then we have to check that the dual graph of the curves $R_i$ in the support of $C_i$ form an extended Dynkin diagram and that the coefficients of the smooth rational curves $R_i$ in $C_i$ match the multiplicities of the irreducible components of the singular fibers in Kodaira's classification.

\begin{example}
For $Y_k=Y_{158}$, the claimed sequence of half-fibers is the following:
\begin{gather*}
F_1=\frac{1}{2}(R_{0}+R_{2}),~F_2=\frac{1}{2}(R_{2}+R_{16}),~F_3=\frac{1}{2}(R_{2}+R_{2}\cdot H_2),~F_4=\frac{1}{2}(R_{3}+R_{4}),\\
F_5=\frac{1}{2}(R_{3}+R_{12}),~F_6=\frac{1}{2}(R_{2}+R_{14}),~F_7=(R_{2}+R_{8}),\\
F_8=\frac{1}{2}(R_{1}+R_{6}+R_{8}+R_{15}+2R_{7}),~F_9=\frac{1}{2}(R_{5}+R_{6}+R_{8}+R_{9}+2R_{11}).
\end{gather*}

Using the data in \cite{SB20data} we can write $R_0,R_1,R_2,R_3\ldots,R_9$, $R_{11},R_{12},R_{14},R_{15},R_{16},R_2\cdot H_2$ in the $E_{10}$-basis as

\begin{align}
\label{eq:smooth-rat-curves-158}
\begin{split}
R_{0}&=(2, 1, 2, 3, 2, 1, 0, 0, 0, 0),~R_{1}=(2, 1, 2, 3, 2, 2, 2, 1, 0, 0),\\
R_{2}&=(8, 5, 10, 15, 14, 13, 10, 6, 4, 2),~R_{3}=(8, 5, 10, 15, 14, 13, 10, 7, 4, 1),\\
R_{4}&=(4, 1, 4, 7, 6, 5, 4, 3, 2, 1),~R_{5}=(0, 1, 0, 0, 0, 0, 0, 0, 0, 0),\\
R_{6}&=(0, 0, 0, 0, 0, 0, 0, 0, 0, 1),~R_{7}=(0, 0, 0, 0, 0, 0, 0, 0, 1, 0),\\
R_{8}&=(0, 0, 0, 0, 0, 0, 0, 1, 0, 0),~R_{9}=(2, 1, 2, 4, 4, 4, 4, 3, 2, 1),\\
R_{11}&=(4, 2, 5, 8, 7, 6, 4, 2, 1, 0),~R_{12}=(4, 3, 6, 9, 8, 7, 4, 3, 2, 1),\\
R_{14}&=(4, 3, 6, 9, 8, 7, 6, 4, 2, 0),~R_{15}=(4, 3, 6, 9, 8, 8, 6, 4, 2, 1),\\
R_{16}&=(6, 3, 8, 13, 12, 11, 8, 6, 4, 2),~R_{2}\cdot H_2=(12,7,14,23,20,19,14,10,6,2).
\end{split}
\end{align}
So, the coordinates of $F_1,\ldots,F_9$ in the same basis are
\begin{align*}
F_1&=(5,3,6,9,8,7,5,3,2,1),~F_2=(7,4,9,14,13,12,9,6,4,2),\\
F_3&=(10,6,12,19,17,16,12,8,5,2),~F_4=(6,3,7,11,10,9,7,5,3,1),\\
F_5&=(6,4,8,12,11,10,7,5,3,1),~F_6=(6,4,8,12,11,10,8,5,3,1),\\
F_7&=(8,5,10,15,14,13,10,7,4,2),~F_8=(3,2,4,6,5,5,4,3,2,1),\\
F_9&=(5,3,6,10,9,8,6,4,2,1).
\end{align*}
The matrix equality $FM_{E_{10}}F^\intercal=\mathbf{1}_9-\mathbf{I}_9$ holds true. Below we write the intersection matrix of the smooth rational curves in \eqref{eq:smooth-rat-curves-158}:
\setcounter{MaxMatrixCols}{16}
{\scriptsize
\[
\begin{pmatrix}
-2&0&2&2&0&0&0&0&0&2&0&0&2&2&2&2\\
0&-2&0&0&0&0&0&1&0&0&0&2&0&0&2&0\\
2&0&-2&0&2&0&0&0&2&0&0&2&2&0&2&2\\
2&0&0&-2&2&0&2&0&0&0&0&2&0&0&2&0\\
0&0&2&2&-2&2&0&0&0&0&0&2&2&2&0&2\\
0&0&0&0&2&-2&0&0&0&0&1&0&0&0&2&0\\
0&0&0&2&0&0&-2&1&0&0&1&0&2&0&0&2\\
0&1&0&0&0&0&1&-2&1&0&0&0&0&1&0&0\\
0&0&2&0&0&0&0&1&-2&0&1&0&0&0&0&0\\
2&0&0&0&0&0&0&0&0&-2&1&2&0&0&0&0\\
0&0&0&0&0&1&1&0&1&1&-2&0&0&0&0&0\\
0&2&2&2&2&0&0&0&0&2&0&-2&2&0&0&2\\
2&0&2&0&2&0&2&0&0&0&0&2&-2&0&2&2\\
2&0&0&0&2&0&0&1&0&0&0&0&0&-2&0&0\\
2&2&2&2&0&2&0&0&0&0&0&0&2&0&-2&2\\
2&0&2&0&2&0&2&0&0&0&0&2&2&0&2&-2
\end{pmatrix}
\]
}
From this matrix we can extract that the extended Dynkin diagrams corresponding to the elliptic configuration associated to $F_1,\ldots,F_9$, which are respectively:
\[
6\times\widetilde{A}_{1}^{\mathrm{F}},~\widetilde{A}_{1}^{\mathrm{HF}},~2\times\widetilde{D}_{4}^{\mathrm{F}}.
\]
The coefficients of the smooth rational curves in \eqref{eq:smooth-rat-curves-158} indeed match the multiplicities from Kodaira's classification. This completes the check for $i=158$, confirming that $\mathrm{nd}(Y_{158})\geq9$.
\end{example}

To assist with the above checking for the $152$ cases in Theorem~\ref{thm:nds-of-tautaubar-generic-Enriques-surfaces}~(2) and (3), we provide appropriate scripts in \cite{MRS23Code}.

\subsection{A conjecture about \texorpdfstring{$\mathrm{nd}(Y_{158})$}{Lg}}

By Theorem~\ref{thm:nds-of-tautaubar-generic-Enriques-surfaces}, we have that
\[
\mathrm{nd}(Y_{84}),\mathrm{nd}(Y_{121}),\mathrm{nd}(Y_{158})\geq9.
\]
As there are no known examples of Enriques surfaces with non-degeneracy invariant $9$, it is worthwhile to explore whether at least one of these three lower bounds is attained. As $\mathrm{aut}(Y_{158})$ has the smallest number of generators, $Y_{158}$ is a good candidate to study more in detail. We run the code \cite{MRS22Code} with large (finite) samples of smooth rational curves $\mathcal{R}$ in $\mathcal{R}(Y_{158})=\mathcal{R}_{\mathrm{temp}}\cdot\mathrm{aut}(Y_{158})$ always obtaining that $\mathrm{cnd}(Y_{158},\mathcal{R})=9$. This computational evidence motivates the following conjecture.

\begin{conjecture}
The realizable $(\tau,\overline{\tau})$-generic Enriques surfaces $158$ satisfies $\mathrm{nd}(Y_{158})=9$.
\end{conjecture}

\subsection{A question about \texorpdfstring{$\mathrm{cnd}(Y_{1})$}{Lg}}
\label{sec:cndY1}

For the $(\tau,\overline{\tau})$-generic Enriques surface $Y_1$, we found that $\mathrm{cnd}(Y_1)\geq\mathrm{cnd}(Y_1,\mathcal{R})=7$ for several choices of $\mathcal{R}$ among the finite sets of smooth rational curves containing $\mathcal{R}_{\mathrm{temp}}$. Additionally, by Proposition~\ref{prop:example-cnd<nd} we have that $\mathrm{cnd}(Y_1)<10$, so the possibilities are $\mathrm{cnd}(Y_1)\in\{7,8,9\}$. It would be interesting to know the value of $\mathrm{cnd}(Y_1)$.

\section{The non-degeneracy invariant of the Enriques surface \texorpdfstring{$145$}{Lg}}
\label{sec:145}

In this section we fix $Y:=Y_{145}$ to be the realizable $(\tau,\overline{\tau})$-generic Enriques surface $145$ in \cite[Table~1]{BS22}. In this case, $(\tau,\overline{\tau})=(E_8,E_8)$. Our goal is to prove that $\mathrm{nd}(Y)=4$ in Theorem~\ref{thm:nd-for-Enriques-145}. In what follows, we fix an $E_{10}$-basis of $S_Y$ compatible with the computational data in \cite{SB20data}.

\subsection{Automorphisms and smooth rational curves on \texorpdfstring{$Y$}{Lg}}
\label{sec:auts-and-rats-in-145}

Next, we describe $\mathrm{aut}(Y)$. A description of the full automorphism group $\mathrm{Aut}(Y)$ appears already in \cite[Theorem~4.12]{BP83}. 
At the same time, it is necessary for our purposes to describe explicit generators compatibly with \cite{BS22}.

Start by first considering the $10$ rational curves $R_0,\ldots,R_9$ in $\mathcal{R}_{\mathrm{temp}}$ on $Y$ which were introduced in Example~\ref{ex:how-to-extract-the-data}.  The dual graph of the union of $R_0,\ldots,R_9$ appears in Figure~\ref{fig:145TenRats}.

\begin{figure}[H]
\begin{tikzpicture}
\draw  (0,0)-- (1,0);
\draw  (4,0)-- (5,0);
\draw  (0,0)-- (1,1);
\draw  (1,1)-- (2.5,2);
\draw  (2.5,2)-- (4,1);
\draw  (4,1)-- (5,0);
\draw  (5,0)-- (4,-1);
\draw  (4,-1)-- (2.5,-2);
\draw  (2.5,-2)-- (1,-1);
\draw  (1,-1)-- (0,0);
\fill [color=black] (0,0) circle (2.4pt);
\draw[color=black] (-0.35,0) node {$R_9$};
\fill [color=black] (1,0) circle (2.4pt);
\draw[color=black] (1.02,0.36) node {$R_8$};
\fill [color=black] (4,0) circle (2.4pt);
\draw[color=black] (4,0.36) node {$R_1$};
\fill [color=black] (5,0) circle (2.4pt);
\draw[color=black] (5.35,0) node {$R_4$};
\fill [color=black] (2.5,2) circle (2.4pt);
\draw[color=black] (2.5,2.29) node {$R_6$};
\fill [color=black] (2.5,-2) circle (2.4pt);
\draw[color=black] (2.5,-2.3) node {$R_2$};
\fill [color=black] (1,1) circle (2.4pt);
\draw[color=black] (0.81,1.31) node {$R_7$};
\fill [color=black] (4,1) circle (2.4pt);
\draw[color=black] (4.26,1.29) node {$R_5$};
\fill [color=black] (4,-1) circle (2.4pt);
\draw[color=black] (4.3,-1.2) node {$R_3$};
\fill [color=black] (1,-1) circle (2.4pt);
\draw[color=black] (0.79,-1.2) node {$R_0$};
\end{tikzpicture}
\caption{Intersection graph of the smooth rational curves $R_0,\ldots,R_9$ on the Enriques surface $Y_{145}$.}
\label{fig:145TenRats}
\end{figure}
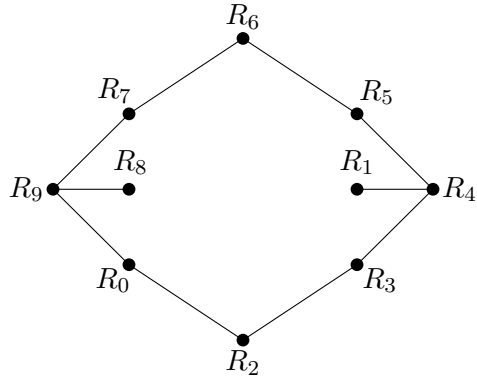

Recall that a pair $(2F_1,2F_2)$ of genus one pencils such that $F_1 \cdot F_2=1$ is called a \textit{$U$-pair}. To a $U$-pair one can associate the \emph{bielliptic map}, generically of degree $2$, which is induced by the linear series $|2F_1 + 2F_2|$. Therefore, a $U$-pair gives rise to an involution on $Y$. We refer to \cite[\S\,3.3]{CDL24} for more properties of the linear series associated to the $U$-pairs.

\begin{definition}\label{def:EpsilonAndDelta}
Consider the following fibers of elliptic fibrations:
\begin{align*}
G_1 &= 2(R_0+R_2+R_3+R_4+R_5+R_6+R_7+R_9),\\
G_2 &= R_0 + R_1 + R_3 +R_8 + 2(R_4+R_5+R_6+R_7+R_9),\\
G_3 &= 4R_9 + 3R_0+3R_7 + 2R_2+2R_6+2R_8 + R_3+R_5.
\end{align*}
Let $\bepsilon$ be the involution associated to the $U$-pair $(G_1,G_2)$. Another involution $\bdelta$ is associated to the $U$-pair $(G_1,G_3)$. Observe that $G_3 \cdot R_1=0$, so $R_1$ is a component of a singular fiber of $|G_3|$. By the classification in \cite{SB20ellipticfibrations}, such a fiber forms a $\widetilde{A}_1$ configuration: we denote by $R_1'$ its other component. Finally, define $\bgamma \coloneqq \bdelta \circ \bepsilon$.
\end{definition}

\begin{lemma}
\label{lem:145ActionOnCurves} The automorphisms $\bepsilon,\bdelta,\bgamma \in \mathrm{Aut}(Y)$ constructed above satisfy the following properties: 
\begin{enumerate}[(a)]

\item The involution $\bepsilon$ fixes $R_2$ and $R_6$ and it acts as a transposition on the pairs of curves
\[
(R_0,R_3),~(R_1,R_8),~(R_4,R_9),~(R_5,R_7).
\]

\item The involution $\bdelta$ fixes $R_4$, $R_8$, $R_9$, and it acts as a transposition on the pairs
\[
(R_0,R_7),~(R_2,R_6),~(R_3,R_5),~(R_1,R_1').
\]

\item The automorphism $\bgamma$ has infinite order and it acts as a transposition on the pairs
\[
(R_0,R_5),~(R_2,R_6),~(R_3,R_7),~(R_4,R_9).
\]
Moreover, let us define $R_{8,k}:=\bgamma^k (R_8)$. Then the following hold:
\begin{enumerate}[(c.1)]

\item $R_{8,-1}=R_1$;

\item $R_{8,m}\neq R_{8,n}$ for all $m,n\in\mathbb{Z}$, $m\neq n$;

\item $R_{8,k}\neq R_0,\ldots,R_9$ for all $k\in\mathbb{Z}\setminus\{-1,0\}$.

\end{enumerate}

\end{enumerate}
\end{lemma}

\begin{proof}
By \cite[Lemma~8.7.5]{DK24}, $\bepsilon$ preserves the elliptic fibrations $|G_1|$ and $|G_2|$. The automorphism $\bepsilon$ is numerically non-trivial: this follows from \cite[Lemma~8.2.5]{DK24} and the fact that the singular fibers in $|G_1|$ and $|G_2|$ have only $7$ common components. The common components $R_3, R_4, R_5, R_6, R_7, R_9, R_0$ form an $A_7$-configuration. Then, we claim that $\bepsilon$  must act on the $A_7$-configuration as the non-trivial symmetry of the $A_7$ diagram. Indeed, otherwise it would fix the classes of all curves $R_i$, $i=0,\ldots,9$. Since these classes are independent in $S_Y$, $\bepsilon$ would act trivially on a basis of $S_{Y}\otimes\mathbb{Q}$, and hence numerically trivially. This is a contradiction. The claimed action of $\bepsilon$ on $R_1,R_2,R_8$ follows since $\bepsilon$ preserves the intersection products. This shows part~(a).

Consider now $\bdelta$. Observe first that the $\widetilde{E}_7$-fiber $G_3$ must be preserved, and hence that $R_8$ must be fixed. Likewise, $\bdelta$ maps $R_4$ to itself. 
Next, we claim that $\bdelta(R_3)=R_3$ and $\bdelta(R_5)=R_5$ if and only if $\bdelta(R_1)=R_1$ and $\bdelta(R_1')=R_1'$. 
In fact, the curve $R_4$ is a bisection of $|G_3|$ meeting $G_3$ in $R_3$ and $R_5$, and meeting the singular fiber $R_1+R_1'$ once in each component. 
The restriction $\bdelta|_{R_4}$ has at least two fixed points given by the intersection with the two half-fibers of $|G_3|$. Since an automorphism of $R_4\cong\mathbb{P}^1$ with three fixed points is the identity, and since $\bdelta$ preserves the fibration $|G_3|$, the map $\bdelta|_{R_4}$ fixes $R_4 \cap R_3$ and $R_4 \cap R_5$ only if it is the identity, which establishes the claim. 
Arguing as for $\bepsilon$, the automorphism $\bdelta$ acts numerically non-trivially. Then, we claim that $\bdelta$ must act on the $A_7$-configuration of common curves among $G_1$ and $G_3$ as a reflection. Suppose otherwise for the sake of contradiction. In particular, the classes $R_3$ and $R_5$ are preserved by $\bdelta$. As explained above, this implies that $\bdelta$ also preserves $R_1$, and hence that it is numerically trivial, arguing as for $\bepsilon$. This is a contradiction, hence $\bdelta$ acts as a reflection on the $A_7$-configuration and transposes $R_1$ and $R_1'$.

For part~(c), the action of $\bgamma$ on $R_i$ for $i\neq8$ follows from parts~(a) and (b) above. We now prove that $\bgamma$ has infinite order. Denote by $\bgamma_0$ the isometry of $S_Y$ which is induced by $\bgamma$. First, observe that $\bgamma_0^2$ is the identity on the $\widetilde{A}_7$-configuration. Hence, by the proof of \cite[Lemma~4.10]{BP83}, we have that $\bgamma_0^2$ is an element of a subgroup of $O(S_Y)$ isomorphic to $\mathbb{Z}\rtimes(\mathbb{Z}/2\mathbb{Z})$. In particular, as $\bgamma_0^2$ is not the identity ($\bgamma_0^2(R_1)=R_1'$), we have that $\bgamma_0^2$ has order $2$ or infinite. So, if we show that $\bgamma_0^4$ is a nontrivial isometry, then we can conclude that $\bgamma_0^2$, and hence $\bgamma,\bgamma_0$, has infinite order. Consider the $\widetilde{E}_7$ elliptic configurations given by
\begin{align*}
M&=4R_4 + 3R_3+3R_5 + 2R_2+2R_6+2R_1 + R_0+R_7,\\
M'&=4R_4 + 3R_3+3R_5 + 2R_2+2R_6+2R_1' + R_0+R_7.
\end{align*}
By \cite{SB20ellipticfibrations}, the elliptic fibrations $|M|,|M'|$ have a $\widetilde{A}_1$ fiber each, which both contain the curve $R_8$. Explicitly, we can write these two $\widetilde{A}_1$ fibers as $R_8+R$ and $R_8+S$ for some smooth rational curves $R$ and $S$. As $\bgamma^2(M)=M'$, we have that $\bgamma^2(R_8+R)=R_8+S$. So, $\bgamma^2(R_8)=R_8$ or $S$. We show that the former is impossible by contradiction. If $\bgamma^2(R_8)=R_8$, then $\bgamma(R_1')=R_8$. By applying $\bdelta$ to both sides, we obtain $\bepsilon(R_1')=\bdelta(R_8)=R_8$, hence $\bepsilon(R_8)=R_1'$, which cannot be as $\bepsilon(R_8)=R_1$ and $R_1\neq R_1'$. We argued that $\bgamma^2(R_8)=S$, which forces $\bgamma^2(R)=R_8$, whence $\bgamma^4(R)=\bgamma^2(R_8)=S$. As it also holds $\bgamma_0^4(R)=S$, this proves that $\bgamma_0^4$ is not the trivial isometry.

We now can prove part~(c.2). If false, then let $k$ be a nonzero integer such that $\bgamma^k(R_8)=R_8$. Therefore, also $\bgamma^{2k}(R_8)=R_8$. Additionally, $\bgamma^{2k}$ fixes the $\widetilde{A}_7$-configuration. We have that
\begin{align*}
\bgamma^{2k}(R_8)=R_8&\implies\bgamma^{-2k}(R_8)=R_8\implies\bgamma^{-2k}(\bepsilon(R_1))=R_8\\
&\implies\bepsilon(\bgamma^{2k}(R_1))=R_8\implies\bgamma^{2k}(R_1)=\bepsilon(R_8)=R_1.
\end{align*}
As $\bgamma^{2k}$ fixes each curve $R_0,\ldots,R_9$, we have that $\bgamma_0^{2k}$ is the trivial isometry, contradicting the fact proved above that $\bgamma_0$ has infinite order. Finally, let us show (c.3). By the above argument, we already know that $R_{8,k}\neq R_1,R_8$. If $R_{8,k}=R_j$ for $j\neq 1,8$, then $R_{8,k+2}=R_j$ (recall $\bgamma^2$ fixes the $\widetilde{A}_7$-configuration), which contradicts (c.2).
\end{proof}

The remaining part of this section relies on the computational data \cite{SB20data}. To help the reader reproduce the computational arguments, we introduce some notations and conventions, and reconcile our considerations so far with \cite{BS22}. In particular, recall that \cite{BS22} fixes a $E_{10}$-basis of $S_Y$. The coordinates of the curves $R_0,\ldots,R_9$ with respect to this basis are written as row vectors in Example~\ref{ex:how-to-extract-the-data}. By letting $\bepsilon$ and $\bgamma$ act on $S_Y$, the associated matrices in this $E_{10}$-basis are written below. To denote them, we use the corresponding Greek letters, but not in boldface.

{\scriptsize
\[
\epsilon=\begin{pmatrix}-3 & -2 & -4 & -6 & -5 & -4 & -3 & -2 & 0 & 0 \\
0 & -1 & 0 & 0 & 0 & 0 & 0 & 0 & 0 & 0 \\
2 & 2 & 3 & 4 & 3 & 2 & 1 & 0 & 0 & 0 \\
0 & 0 & 0 & 0 & 0 & 0 & 0 & 1 & 0 & 0 \\
0 & 0 & 0 & 0 & 0 & 0 & 1 & 0 & 0 & 0 \\
0 & 0 & 0 & 0 & 0 & 1 & 0 & 0 & 0 & 0 \\
0 & 0 & 0 & 0 & 1 & 0 & 0 & 0 & 0 & 0 \\
0 & 0 & 0 & 1 & 0 & 0 & 0 & 0 & 0 & 0 \\
4 & 2 & 4 & 6 & 5 & 4 & 3 & 2 & 1 & 0 \\
0 & 0 & 0 & 0 & 0 & 0 & 0 & 0 & 0 & 1\end{pmatrix},
\]
\[
\gamma=\begin{pmatrix}-9 & -4 & -10 & -16 & -13 & -10 & -7 & -6 & -4 & -2\\-8 & -3 & -8 & -14 & -12 & -10 & -8 & -6 & -4 & -2\\10 & 4 & 11 & 18 & 15 & 12 & 9 & 6 & 4 & 2\\0 & 0 & 0 & 0 & 0 & 0 & 0 & 1 & 0 & 0\\0 & 0 & 0 & 0 & 0 & 0 & 0 & 0 & 1 & 0\\0 & 0 & 0 & 0 & 0 & 0 & 0 & 0 & 0 & 1\\4 & 2 & 4 & 6 & 5 & 4 & 3 & 2 & 1 & 0\\0 & 0 & 0 & 1 & 0 & 0 & 0 & 0 & 0 & 0\\0 & 0 & 0 & 0 & 1 & 0 & 0 & 0 & 0 & 0\\0 & 0 & 0 & 0 & 0 & 1 & 0 & 0 & 0 & 0\end{pmatrix}.
\]
}

From now on, we see $\mathrm{aut}(Y)$ as a group of matrices.

\begin{lemma}
\label{lem:145Group}
The group $\mathrm{aut}(Y)$ is generated by $\epsilon$ and $\gamma$, which satisfy the relations
\[
\epsilon^2 = \mathbf{I}_{10},\qquad \epsilon\gamma = \gamma^{-1}\epsilon.
\]
In particular, $\mathrm{aut}(Y)\cong \mathbb{Z}\rtimes(\mathbb{Z}/2\mathbb{Z}) $.
\end{lemma}

\begin{proof}
The claimed relations can be checked directly (this can also be checked at level of the automorphisms $\bepsilon,\bgamma$). The matrix $\gamma$ matches an element of $\mathcal{H}$, extracted from \cite[Enrs.txt]{SB20data} as explained in Example~\ref{ex:how-to-extract-the-data}. The other matrix in the record is $\gamma^{-1}$, and these are the only two elements in $\mathcal{H}$. The matrix $\epsilon$ appears in the record \texttt{autcham} inside \texttt{V0}. In \texttt{autcham} there are two possible \texttt{gY}, and $\epsilon$ equals the record \texttt{gY} different from $\mathbf{I}_{10}$. So, by the discussion in \cite[\S\,6.1]{BS22}, we have that $\epsilon$ and $\gamma$ generate $\mathrm{aut}(Y)$. Finally, the isomorphism of groups $\mathbb{Z}\rtimes(\mathbb{Z}/2\mathbb{Z})\cong\mathrm{aut}(Y)$ is explicitly given by $(n,i)\mapsto\gamma^n\epsilon^i$.
\end{proof}

Recall the convention in \cite{BS22} that the automorphisms act on the smooth rational curves by matrix multiplication on the right. We follow this convention, for consistency with the work of Brandhorst--Shimada. So, for instance, the action of $\bgamma$ on $R_8$ is given by the matrix multiplication
\[
R_8\cdot\gamma=(10,4,11,18,15,12,9,6,4,2),
\]
meaning that the coordinate vector of $\bgamma(R_8)$ is $R_8 \cdot \gamma$, where in the latter we identify as usual the smooth rational curve $R_8$ with its numerical class in $S_Y$ and coordinate vector.  
Although both the intersection product among curves and the matrix action of $\mathrm{aut}(Y)$ are denoted by ``$\cdot$'', the difference between the two will always be clear from the context.

As an illustration, we explicitly compute the coordinate vectors of $R_8\cdot\gamma^k$ as the transpose of
{\small\begin{equation} \label{eqn:parametriccurve145}
\begin{pmatrix} 
4k^2 + 4k + 1-(-1)^k\\
2k^2 + k + \frac{1}{2}-\frac{1}{2}(-1)^k\\
4k^2 + 4k + 2-(-1)^k\\
7k^2 + 7k + 2 - 2(-1)^k\\
6k^2 + 6k + \frac{3}{2}- \frac{3}{2}(-1)^k\\
5k^2 + 5k + 1- (-1)^k\\
4k^2  + 4k + \frac{1}{2}- \frac{1}{2}(-1)^k\\
3(k^2 + k)\\
2(k^2 + k)\\
k^2 + k
\end{pmatrix}.
\end{equation}}
This can be verified by using induction twice, one for the positive and one for the negative integers. Alternatively, the above claim can be checked directly with \texttt{SageMath}, as it can compute abstract $k$-th powers of matrices. We point out that \eqref{eqn:parametriccurve145} can be used to produce an alternative argument for Lemma~5.2~(c).

\begin{corollary}
The only smooth rational curves on $Y$ are $R_0,R_2,R_3,R_4,R_5,R_6,R_7,R_9$ and the infinitely many $R_{8,k}$ for $k\in\mathbb{Z}$, where in particular $R_8=R_{8,0}$ and $R_1=R_{8,-1}$.
\end{corollary}

\begin{proof}
By the discussion in \cite[\S\,6.2]{BS22}, we have that the set $\mathcal{R}(Y)$ of smooth rational curves on $Y$ can be obtained by acting with $\mathrm{aut}(Y)$ on $\mathcal{R}_{\mathrm{temp}}=\{R_0,\ldots,R_9\}$. At the same time, by Lemma~\ref{lem:145Group}, every element of $\mathrm{aut}(Y)$ can be written as $\bepsilon^i \circ \bgamma^n$ for $n\in\mathbb{Z}$, $i\in\{0,1\}$. So, we have that
\[
\mathcal{R}(Y)=\{R_j\cdot(\gamma^n\epsilon^i)\mid n\in\mathbb{Z},~i\in\{0,1\},~j\in\{0,\ldots,9\}\}.
\]
Then, combining the results in Lemma~\ref{lem:145ActionOnCurves} with the equality (given by the right action)
\[
R_j\cdot(\gamma^n\epsilon)=R_j\cdot(\epsilon\gamma^{-n})=(R_j\cdot\epsilon)\cdot\gamma^{-n},
\]
we obtain that
\[
\mathcal{R}(Y)=\mathcal{R}_{\mathrm{temp}}\sqcup\{R_{8,k}\mid k\in\mathbb{Z}\setminus\{0,-1\}\},
\]
which gives the description of $\mathcal{R}(Y)$ in the statement.
\end{proof}

\begin{remark}[The Barth--Peters example]

It can be proved that the $(E_8,E_8)$-generic Enriques surface $Y$ under analysis already appeared in \cite[\S\,4]{BP83}. We briefly recall the construction of Barth and Peters. Consider $Q\coloneqq \mathbb{P}^1 \times \mathbb{P}^1$ together with the involution given in an affine patch by $\mathbf{s}\colon (x,y)\mapsto (-x,-y)$. For certain $(4,4)$-curves $B \subseteq Q$ (see \cite[\S\,4.1,~Case~1]{BP83}), the double cover $\overline{X}\to Q$ branched along $B$ has two $A_7$ singularities. The minimal resolution $X$ of $\overline{X}$ is a K3 surface, and $\mathbf{s}$ lifts to a fixed point free involution $\bsigma$ on $X$. Then, the very general Enriques surface $Y=X/\langle \bsigma \rangle$ is $(E_8,E_8)$-generic. We omit the proof of this as it is not needed in the computations that follow.

Finally, we connect the geometry of $Y$ studied in \S\,\ref{sec:auts-and-rats-in-145} with the results and notations of Barth--Peters. Following \cite[Proposition~3.3.15]{CDL24}, the surface $W\coloneqq Q/\langle \mathbf{s} \rangle$ is a $4$-nodal quartic symmetroid del Pezzo surface in $\mathbb{P}^4$. The bielliptic map $Y\rightarrow W$ associated with the $U$-pair $(G_1,G_3)$ in Definition~\ref{def:EpsilonAndDelta} makes the following diagram commute:
\begin{equation*}
\begin{tikzcd}
    X \rar \dar & Y \dar \\
    Q \rar & W.
\end{tikzcd}
\end{equation*}
In particular, the covering involution of $Y\rightarrow W$, which we denoted by $\bdelta$, coincides with the involution induced on $Y$ by $\bsigma_3$ on $X$ (see \cite[\S\,4.2]{BP83}). While \cite{BP83} does not describe an explicit infinite order generator of $\mathrm{aut}(Y)$, there is a precise description of the subgroup $G$ of $\mathrm{O}(S_Y)$ which fixes $\{R_0,\ldots,R_9\}\setminus \{R_1,R_8\}$: it is an infinite dihedral group, which is generated by an involution $\alpha_0$ and an infinite order isometry $\alpha_1$ \cite[Lemma~4.10~(a)]{BP83}. One can verify that $\bgamma^2$ induces on $S_Y$ the isometry $\alpha_4\coloneqq\alpha_1^4$. This implies, together with the description of $\mathrm{aut}(Y)$ from \cite{BS22}, that $\mathrm{aut}(Y)\cap G$ is generated by $\alpha_4$.
\end{remark}

\subsection{Elliptic fibrations on \texorpdfstring{$Y$}{Lg}}

The elliptic fibrations on $Y$ up to $\mathrm{aut}(Y)$ are classified in \cite[Theorem~1.21]{BS22}. In the next lemma we give explicit representatives in term of the given smooth rational curves $\mathcal{R}(Y)$. First, it will be convenient to recall the following terminology and notation from \cite{MRS22Paper}.

\begin{definition}
\label{def:type}
The \emph{type of an elliptic configuration} supported on $\mathbb{R}(Y)$ is the associated extended Dynkin diagram, together with the information of being a fiber or a half-fiber. The \emph{type of an elliptic fibration} is the formal sum of the types of its singular fibers supported on $\mathcal{R}(Y)$. For instance, $(2\widetilde{A}_1^{\mathrm{HF}}+\widetilde{D}_6^{\mathrm{F}})$ refers to the fibrations whose singular fibers are three elliptic configurations, two of type $\widetilde{A}_1^{\mathrm{HF}}$ and one of type $\widetilde{D}_6^{\mathrm{F}}$.
\end{definition}

\begin{lemma} \label{lem:145TypesEF}
Up to the action of $\mathrm{aut}(Y)$ we have the following types of elliptic fibrations:
\begin{table}[H]
\centering\renewcommand\cellalign{lc}
\setcellgapes{3pt}\makegapedcells
\begin{tabular}{|c|c|}
\hline
Type & Representatives of elliptic fibrations up to $\mathrm{aut}(Y)$
\\
\hline
$\widetilde{A}_7^{\mathrm{HF}}$
&
$R_0+R_2+R_3+R_4+R_5+R_6+R_7+R_9$
\\
\hline
$\widetilde{A}_1^{\mathrm{F}} + \widetilde{E}_7^{\mathrm{F}}$
&
$\left(A:=\frac{1}{2}(R_8+R_{8,-2}),~\frac{1}{2}(3R_0+2R_2+R_3+3R_5+2R_6+R_7+4R_4+2R_1)\right)$
\\
\hline
$\widetilde{D}_8^{\mathrm{F}}$
&
\makecell{
$D_a:=\frac{1}{2}(R_0+R_1+R_3+2R_4+2R_5+2R_6+2R_7+2R_9+R_8)$
\\[1ex]
$D_b:=\frac{1}{2}(2R_0+R_1+2R_2+2R_3+2R_4+R_5+R_7+2R_9+R_8)$
}
\\
\hline
$\widetilde{E}_8^{\mathrm{F}}$
&
\makecell{
$E_a:=\frac{1}{2}(R_1+2R_4+3R_5+4R_6+5R_7+6R_9+4R_0+2R_2+3R_8)$
\\[1ex]
$E_b:=\frac{1}{2}(2R_6+4R_7+6R_9+5R_0+4R_2+3R_3+2R_4+R_1+3R_8)$
}
\\
\hline
\end{tabular}
\end{table}
\end{lemma}
\begin{proof}
Let us introduce the notation $A,D_a,D_b,E_a,E_b$ as in the table above. Using the graph in Figure~\ref{fig:145TenRats}, one can check directly that the given representatives have the corresponding type. What is left to prove is that the two representatives for $\widetilde{D}_8^{\mathrm{F}}$ and $\widetilde{E}_8^{\mathrm{F}}$ are not in the same $\mathrm{aut}(Y)$-orbit.

We start by discussing the case of $\widetilde{D}_8^{\mathrm{F}}$. 
We need to show that $D_a \neq D_b \cdot\sigma$ for all $\sigma\in\mathrm{aut}(Y)$. As $D_b \cdot\epsilon=D_b$, the possible $D_b \cdot\sigma$ are given by $D_{b,k}:=D_b\cdot\gamma^k$ for $k\in\mathbb{Z}$. Therefore, what we need to show is that $D_a$ is different from $D_{b,k}$ for all $k\in\mathbb{Z}$. For this purpose, it is enough to compute the intersection between $D_a$ and $D_{b,k}$ for every $k$, and show that this is never equal to $0$. It can be checked inductively that
\[
D_a \cdot D_{b,k} = k^2 + \frac{1}{2}(-1)^k + \frac{1}{2},
\]
which is never equal to zero for $k\in\mathbb{Z}$.

We use similar ideas for $\widetilde{E}_8^{\mathrm{F}}$. 
In this case, $E_b\cdot\epsilon\neq E_b$, so we consider
\[
E_{b,\epsilon}:=E_{b}\cdot\epsilon,~E_{b,k}:=E_{b}\cdot\gamma^k,~E_{b,\epsilon,k}:=E_{b,\epsilon}\cdot\gamma^k.
\]
Notice that by Lemma \ref{lem:145Group} $E_{b}\cdot\gamma^k \cdot \epsilon = E_{b}\cdot\epsilon \cdot \gamma^{-k}$.
Therefore, to prove that $E_{a}$ is not in the same orbit as $E_{b}$, it is enough to prove that the intersections $E_{a}\cdot E_{b,k}$ and $E_{a}\cdot E_{b,\epsilon,k}$ are nonzero for all $k\in\mathbb{Z}$. We have that
\begin{align*}
E_a\cdot E_{b,k}&=4k^2+3,\\
E_a\cdot E_{b,\epsilon,k}&=4k^2+2(-1)^k-4k+3.
\end{align*}
These are nonzero for all $k\in\mathbb{Z}$, proving what we needed.
\end{proof}

\subsection{Computing \texorpdfstring{$\mathrm{nd}(Y)$}{Lg}}

We start by computing the maximum number of elliptic fibration of the same type which can appear in the same isotropic sequence (this idea was also used in \cite{MRS22Paper}).

\begin{lemma}
\label{lem:145MaximumNumber}
On the Enriques surface $Y$, the maximum number of elliptic fibrations of the same type (see Lemma~\ref{lem:145TypesEF}) that can appear in the same isotropic sequence is given as follows:
\begin{table}[H]
\centering\renewcommand\cellalign{lc}
\setcellgapes{3pt}\makegapedcells
\begin{tabular}{|c|c|}
\hline
Type & Maximum
\\
\hline
$\widetilde{A}_7^{\mathrm{HF}}$
&
$1$
\\
\hline
$\widetilde{A}_1^{\mathrm{F}}+\widetilde{E}_7^{\mathrm{F}}$
&
$1$
\\
\hline
$\widetilde{D}_8^{\mathrm{F}}$
&
$2$
\\
\hline
$\widetilde{E}_8^{\mathrm{F}}$
&
$2$
\\
\hline
\end{tabular}
\end{table}
\end{lemma}

\begin{proof}
We analyze each case separately.

$\mathbf{(\widetilde{A}_7^{\mathbf{HF}})}$ The generators $\epsilon$ and $\gamma$ of $\mathrm{aut}(Y)$ map $R_0+R_2+R_3+R_4+R_5+R_6+R_7+R_9$ to itself, so there is exactly one elliptic configuration on $Y$ of type $\widetilde{A}_1^{\mathrm{HF}}$. Therefore, the maximum we are looking for is $1$.

$\mathbf{(\widetilde{A}_1^{\mathbf{F}}+\widetilde{E}_7^{\mathbf{F}})}$ Consider the representative $A$ of type $\widetilde{A}_1$. First, we observe that
\[
A\cdot(\epsilon\gamma^{-1})=A.
\]
This implies that, for all $\sigma\in\mathrm{aut}(Y)$, $A\cdot\sigma=A\cdot\gamma^k$ for some $k\in\mathbb{Z}$. This is because $\sigma=\gamma^n\epsilon^i$ for some $n\in\mathbb{Z}$ and $i\in\{0,1\}$, and either $i=0$, or $i=1$ and hence
\[
A\cdot(\gamma^n\epsilon)=A\cdot(\epsilon\gamma^{-n})=A\cdot(\epsilon\gamma^{-1})\gamma^{-n+1}=A\cdot\gamma^{-n+1}.
\]
Therefore, we define
\[
A_{k}:=A\cdot\gamma^k,
\]
which give all the other fibrations of the same type $\widetilde{A}_1^{\mathrm{F}}$. Now, assume that two diagrams of type $\widetilde{A}_1^{\mathrm{F}}$ are in the same isotropic sequence. Then, up to the action of $\mathrm{aut}(Y)$, we can assume that one of them to be $A_{0} = A$ while the other is $A_{k}$ for some $k\in\mathbb{Z}\setminus\{0\}$. We compute that
\[
A_{0}\cdot A_{k}=k^2 - \frac{1}{2}(-1)^k + \frac{1}{2},
\]
which is always even, and hence it never equals $1$. In conclusion, there can be only one diagram of type $\widetilde{A}_1^{\mathrm{F}}$ in the same non-degenerate isotropic sequence.

$\mathbf{(\widetilde{D}_8^{\mathbf{F}})}$ Recall the curves $D_{a},D_{b},D_{b,k}$ from the proof of Lemma~\ref{lem:145TypesEF}. Also, note that $D_{a}\cdot\epsilon=D_{a}$ and introduce $D_{a,k}:=D_{a}\cdot\gamma^k$. These give all the fibers of type $\widetilde{D}_8$ on $Y$. Now, assume two diagrams of type $\widetilde{D}_8$ are in the same isotropic sequence. Up to $\mathrm{aut}(Y)$ we have the following possibilities:
\[
(D_{a},D_{a,k}),~(D_{b},D_{b,k}),~(D_{a},D_{b,k}).
\]
We can compute the pairwise intersections as
\begin{align*}
D_{a} \cdot D_{a,k} &= k^2 - \frac{1}{2}(-1)^k + \frac{1}{2},\\
D_{b} \cdot D_{b,k} &= k^2 - \frac{1}{2}(-1)^k + \frac{1}{2},\\
D_{a} \cdot D_{b,k} &= k^2 + \frac{1}{2}(-1)^k + \frac{1}{2}.
\end{align*}
The first two are different from $1$ for every $k$, while the third is equal to $1$ if and only if $k=0,\pm1$. Finally, we observe that $(D_{a},D_{b,k})$ cannot be further extended, as adding $D_{a,h}$ or $D_{b,h}$ for some $h\in\mathbb{Z}$ will result in two elements of the sequence having intersection not equal to $1$.

$\mathbf{(\widetilde{E}_8^{\mathbf{F}})}$ Recall the curves $E_{a},E_{b},E_{b,\epsilon},E_{b,k},E_{b,\epsilon,k}$ from the proof of Lemma~\ref{lem:145TypesEF}. Also, introduce 
$E_{a,\epsilon}:=E_{a}\cdot\epsilon$, 
$E_{a,k}:=E_{a}\cdot\gamma^k$, and 
$E_{a,\epsilon,k}:=E_{a,\epsilon}\cdot\gamma^k$. 
These give all the fibers of type $\widetilde{E}_8$ on $Y$. Now, assume that two diagrams of type $\widetilde{E}_8^{\mathrm{F}}$ are in the same isotropic sequence. Up to $\mathrm{aut}(Y)$ we have the following possibilities:
\[
(E_{a},E_{a,k}),~
(E_{a},E_{a,\epsilon,k}),~
(E_{b},E_{b,k}),~
(E_{b},E_{b,\epsilon,k}),~
(E_{a},E_{b,k}),~
(E_{a},E_{b,\epsilon,k}).
\]
We can compute the pairwise intersections as
\begin{enumerate}

\item $E_{a}\cdot E_{a,k}=4k^2 - 2(-1)^k + 2$;

\item $E_{a}\cdot E_{a,\epsilon,k}=4k^2 - 4k + 4$;

\item $E_{b}\cdot E_{b,k}=4k^2 - 2(-1)^k + 2$;

\item $E_{b}\cdot E_{b,\epsilon,k}=4k^2 - 4k + 4$;

\item $E_{a}\cdot E_{b,k}=4k^2 + 3$;

\item $E_{a}\cdot E_{b,\epsilon,k}=4k^2 + 2(-1)^k - 4k + 3$.

\end{enumerate}
We observe that in cases (1), (2), (3), (4), (5) the intersection product is not equal to $1$ for any $k$, while in case (6) $E_{a}\cdot E_{b,\epsilon,k}=1$ if and only if $k=1$. From the above calculations also follows that the isotropic sequence $(E_{a},E_{b,\epsilon,k})$ cannot be further extended because by adding a curve $E_{a,h}$, $E_{a,\epsilon,h}$, $E_{b,h}$, or $E_{b,\epsilon,h}$ for some $h\in\mathbb{Z}$ we will have a pairwise intersection not equal to $1$.
\end{proof}

\begin{proposition}
\label{prop:cnd145}
Let $Y$ be the realizable $(\tau,\overline{\tau})$-generic Enriques surface $145$ in \cite[Table~1]{BS22}. Then, $\mathrm{cnd}(S)=4$. 
\end{proposition}

\begin{proof}
An explicit example of non-degenerate isotropic sequence of length $4$, also described in \cite[Remark~4.3]{MMV22a}, is given by the numerical equivalence classes of
\begin{align*}
R_{0}+R_{2}+R_{3}+R_{4}+R_{5}+R_{6}+R_{7}+R_{9}\qquad&(\widetilde{A}_7^{\mathrm{HF}})\\
\frac{1}{2}(R_{0}+R_{8}+R_{1}+R_{3}+2R_{9}+2R_{7}+2R_{6}+2R_{5}+2R_{4})\qquad&(\widetilde{D}_8^{\mathrm{F}})\\
\frac{1}{2}(R_{1}+R_{5}+R_{7}+R_{8}+2R_{4}+2R_{3}+2R_{2}+2R_{0}+2R_{9})\qquad&(\widetilde{D}_8^{\mathrm{F}})\\
\frac{1}{2}(2R_{2}+3R_{0}+4R_{9}+3R_{7}+2R_{6}+R_{5}+2R_{8}+R_{3})\qquad&(\widetilde{E}_7^{\mathrm{F}}).
\end{align*}
This shows $\mathrm{cnd}(Y)\geq 4$. To prove that equality holds, we assume by contradiction that there exists an isotropic sequence of five elements in $\mathsf{HF}(Y,\mathcal{R}(Y))$ (see Definition~\ref{def:relative-cnd-invariant}). By Lemma~\ref{lem:145MaximumNumber}, the types of the corresponding five elliptic fibrations can only be one of the following combinations:
\begin{itemize}

\item $(\widetilde{E}_7^{\mathrm{F}} + \widetilde{A}_1^{\mathrm{F}}),~2 \times (\widetilde{D}_8^{\mathrm{F}}),~2 \times (\widetilde{E}_8^{\mathrm{F}})$;

\item $(\widetilde{A}_7^{\mathrm{HF}}),~2 \times (\widetilde{D}_8^{\mathrm{F}}),~2 \times (\widetilde{E}_8^{\mathrm{F}})$;

\item $(\widetilde{A}_7^{\mathrm{HF}}),~(\widetilde{E}_7^{\mathrm{F}} + \widetilde{A}_1^{\mathrm{F}}),~(\widetilde{D}_8^{\mathrm{F}}),~2 \times (\widetilde{E}_8^{\mathrm{F}})$;

\item $(\widetilde{A}_7^{\mathrm{HF}}),~(\widetilde{E}_7^{F} + \widetilde{A}_1^{\mathrm{F}}),~2 \times (\widetilde{D}_8^{\mathrm{F}}),~(\widetilde{E}_8^{\mathrm{F}})$.

\end{itemize}
Recall from the beginning of the proof of Lemma~\ref{lem:145MaximumNumber} that the half-fiber of type $\widetilde{A}_7$ is invariant under $\mathrm{aut}(Y)$. Moreover, for all $k\in\mathbb{Z}$, we have that
\[
\widetilde{A}_7\cdot E_{a,k}=\widetilde{A}_7\cdot E_{a,\epsilon,k}=\widetilde{A}_7\cdot E_{b,k}=\widetilde{A}_7\cdot E_{b,\epsilon,k}=2.
\]
This implies that no isotropic sequence on $Y$ contains an elliptic fibration of type $\widetilde{A}_7^{\mathrm{HF}}$ and another of type $\widetilde{E}_8^{\mathrm{F}}$. So, the only combination of types allowed in the isotropic sequence of length $5$ is
\begin{equation}
\label{eq:sequence-of-types-allowed}
(\widetilde{A}_1^{\mathrm{F}} + \widetilde{E}_7^{\mathrm{F}}),~2 \times (\widetilde{D}_8^{\mathrm{F}}),~2 \times (\widetilde{E}_8^{\mathrm{F}}).
\end{equation}
Now consider one of the two half-fibers of type $\widetilde{E}_8$. Up to $\mathrm{aut}(Y)$, this is either $E_{a}$ or $E_{b}$. Assume we have $E_{a}$ (an analogous argument holds for $E_{b}$). We now show that there exists exactly one half-fiber of type $\widetilde{D}_8$ which has intersection $1$ with $E_{a}$, giving a contradiction as we have two of such half-fibers from \eqref{eq:sequence-of-types-allowed}. We consider then the following intersection numbers:
\begin{itemize}

\item $E_{a}\cdot D_{a,k}=2k^2 - \frac{1}{2}(-1)^k - k + \frac{3}{2}=1$ if and only if $k=0$;

\item $E_{a}\cdot D_{b,k}=2k^2 + \frac{1}{2}(-1)^k - k + \frac{3}{2}\neq1$ for all $k\in\mathbb{Z}$.

\end{itemize}
So, $D_{a,0}=D_{a}$ is the only half-fiber of type $\widetilde{D}_8$ such that $E_{a}\cdot D_{a}=1$. In conclusion, we cannot have on $Y$ an isotropic sequence of elements in $\mathsf{HF}(Y,\mathcal{R}(Y))$ of length $5$, hence $\mathrm{cnd}(S)=4$.
\end{proof}

\begin{theorem}
\label{thm:nd-for-Enriques-145}
Let $Y$ be the $(\tau,\overline{\tau})$-generic Enriques surface $145$ in \cite[Table~1]{BS22}. Then, we have that $\mathrm{nd}(Y)=4$. 
\end{theorem}

\begin{proof}
By the classification of elliptic fibrations modulo automorphisms \cite[Theorem~1.21]{BS22} and the complete table of these, which can be found in \cite{SB20ellipticfibrations}, we obtain that every elliptic fibration on $Y$ admits a fiber or a half-fiber supported on the union of smooth rational curves. 
So, we can apply Proposition~\ref{prop:CndEqualNd} to argue that $\mathrm{nd}(Y)=\mathrm{cnd}(Y)$. Finally, $\mathrm{cnd}(Y)$ was shown to equal $4$ in Proposition~\ref{prop:cnd145}.
\end{proof}

\section{List of isotropic sequences}
\label{App:TableNonDeg}

Below we provide explicit non-degenerate isotropic sequences realizing the lower bounds in Theorem~\ref{thm:nds-of-tautaubar-generic-Enriques-surfaces}. Recall that for each case, the curves $R_i$ and the automorphisms $H_i$ are taken from the data \texttt{Rats-Ratstemp} and \texttt{Autrec-HHH} of \cite{SB20data} respectively, and they are numbered sequentially starting from zero. We only list the realizable $(\tau,\overline{\tau})$-generic Enriques surfaces, so we skip $Y_{26},Y_{48},Y_{49},\ldots$ (we refer to \cite[Table~1]{BS22} for the complete list of non-existent cases). Recall that the non-degeneracy invariants for $Y_1,Y_{172},Y_{184}$ are known: see Remark~\ref{rmk:nd-for1-172-184} and the discussion in \S\,\ref{sec:cndY1} about $\mathrm{cnd}(Y_1)$. A digital version of these data is available in \cite{MRS23Code}.

\

{
\scriptsize
\noindent \textbf{(1)} $\mathrm{cnd}(Y_{1}) \geq 7$. Sequence: $\frac{1}{2}(R_{0}+R_{1})$, $\frac{1}{2}(R_{0}+R_{2})$, $\frac{1}{2}(R_{0}+R_{3})$, $\frac{1}{2}(R_{0}+R_{10})$, $\frac{1}{2}(R_{0}+R_{11})$, $\frac{1}{2}(R_{0}+R_{12})$, $\frac{1}{2}(R_{0}+R_{13})$.\\ 
\textbf{(2)} $\mathrm{nd}(Y_{2}) = 10$. Sequence: $\frac{1}{2}(R_{0}+R_{4})$, $\frac{1}{2}(R_{0}+R_{6})$, $\frac{1}{2}(R_{0}+R_{12})$, $\frac{1}{2}(R_{0}+R_{15})$, $\frac{1}{2}(R_{0}+R_{29})$, $\frac{1}{2}(R_{0}+R_{31})$, $\frac{1}{2}(R_{1}+R_{2})$, $\frac{1}{2}(R_{1}+R_{3})$, $\frac{1}{2}(R_{1}+R_{22})$, $(R_{0}+R_{1})$.\\ 
\textbf{(3)} $\mathrm{nd}(Y_{3}) = 10$. Sequence: $\frac{1}{2}(R_{0}+R_{4})$, $\frac{1}{2}(R_{0}+R_{9})$, $\frac{1}{2}(R_{0}+R_{10})$, $\frac{1}{2}(R_{1}+R_{3})$, $\frac{1}{2}(R_{1}+R_{6})$, $\frac{1}{2}(R_{1}+R_{11})$, $\frac{1}{2}(R_{2}+R_{5})$, $\frac{1}{2}(R_{2}+R_{7})$, $\frac{1}{2}(R_{2}+R_{8})$, $\frac{1}{2}(R_{0}+R_{1}+R_{2})$.\\ 
\textbf{(4)} $\mathrm{nd}(Y_{4}) = 10$. Sequence: $\frac{1}{2}(R_{0}+R_{6})$, $\frac{1}{2}(R_{0}+R_{9})$, $\frac{1}{2}(R_{1}+R_{10})$, $\frac{1}{2}(R_{0}+R_{8})$, $\frac{1}{2}(R_{0}+R_{46})$, $\frac{1}{2}(R_{0}+R_{41})$, $\frac{1}{2}(R_{1}+R_{7})$, $\frac{1}{2}(R_{1}+R_{55})$, $\frac{1}{2}(R_{2}+R_{11})$, $\frac{1}{2}(R_{2}+R_{37})$.\\ 
\textbf{(5)} $\mathrm{nd}(Y_{5}) = 10$. Sequence: $\frac{1}{2}(R_{0}+R_{10})$, $\frac{1}{2}(R_{0}+R_{12})$, $\frac{1}{2}(R_{1}+R_{7})$, $\frac{1}{2}(R_{1}+R_{13})$, $\frac{1}{2}(R_{2}+R_{16})$, $\frac{1}{2}(R_{4}+R_{5})$, $\frac{1}{2}(R_{4}+R_{6})$, $\frac{1}{2}(R_{8}+R_{16})$, $\frac{1}{2}(R_{15}+R_{16})$, $\frac{1}{2}(R_{3}+R_{16})$.\\ 
\textbf{(6)} $\mathrm{nd}(Y_{6}) = 10$. Sequence: $\frac{1}{2}(R_{79}+R_{124})$, $\frac{1}{2}(R_{79}+R_{19})$, $\frac{1}{2}(R_{120}+R_{8})$, $\frac{1}{2}(R_{9}+R_{11})$, $\frac{1}{2}(R_{74}+R_{6})$, $\frac{1}{2}(R_{79}+R_{0}+R_{18})$, $\frac{1}{2}(R_{120}+R_{6}+R_{18})$, $\frac{1}{2}(R_{8}+R_{6}+R_{18})$, $\frac{1}{2}(R_{9}+R_{6}+R_{0})$, $\frac{1}{2}(R_{11}+R_{6}+R_{0})$.\\ 
\textbf{(7)} $\mathrm{nd}(Y_{7}) = 10$. Sequence: $\frac{1}{2}(R_{0}+R_{7})$, $\frac{1}{2}(R_{1}+R_{24})$, $\frac{1}{2}(R_{2}+R_{8})$, $\frac{1}{2}(R_{2}+R_{13})$, $\frac{1}{2}(R_{3}+R_{4})$, $\frac{1}{2}(R_{3}+R_{15})$, $\frac{1}{2}(R_{7}+R_{12})$, $\frac{1}{2}(R_{9}+R_{24})$, $\frac{1}{2}(R_{10}+R_{24})$, $\frac{1}{2}(R_{11}+R_{24})$.\\ 
\textbf{(8)} $\mathrm{nd}(Y_{8}) = 10$. Sequence: $\frac{1}{2}(R_{0}+R_{9})$, $\frac{1}{2}(R_{0}+R_{51})$, $\frac{1}{2}(R_{0}+R_{11})$, $\frac{1}{2}(R_{0}+R_{24})$, $\frac{1}{2}(R_{1}+R_{10})$, $\frac{1}{2}(R_{2}+R_{72})$, $\frac{1}{2}(R_{2}+R_{40})$, $\frac{1}{2}(R_{2}+R_{56})$, $\frac{1}{2}(R_{3}+R_{64})$, $\frac{1}{2}(R_{3}+R_{92})$.\\ 
\textbf{(9)} $\mathrm{nd}(Y_{9}) = 10$. Sequence: $\frac{1}{2}(R_{0}+R_{9})$, $\frac{1}{2}(R_{0}+R_{15})$, $\frac{1}{2}(R_{1}+R_{7})$, $\frac{1}{2}(R_{1}+R_{12})$, $\frac{1}{2}(R_{2}+R_{13})$, $\frac{1}{2}(R_{2}+R_{14})$, $\frac{1}{2}(R_{3}+R_{20})$, $\frac{1}{2}(R_{8}+R_{20})$, $\frac{1}{2}(R_{10}+R_{20})$, $\frac{1}{2}(R_{16}+R_{20})$.\\ 
\textbf{(10)} $\mathrm{nd}(Y_{10}) = 10$. Sequence: $\frac{1}{2}(R_{0}+R_{23})$, $\frac{1}{2}(R_{5}+R_{98})$, $\frac{1}{2}(R_{9}+R_{101})$, $\frac{1}{2}(R_{10}+R_{101})$, $\frac{1}{2}(R_{91}+R_{98})$, $\frac{1}{2}(R_{1}+R_{21})$, $\frac{1}{2}(R_{1}+R_{95})$, $\frac{1}{2}(R_{1}+R_{96})$, $\frac{1}{2}(R_{5}+R_{91})$, $\frac{1}{2}(R_{1}+R_{23}+R_{92})$.\\ 
\textbf{(11)} $\mathrm{nd}(Y_{11}) = 10$. Sequence: $\frac{1}{2}(R_{0}+R_{15})$, $\frac{1}{2}(R_{0}+R_{16})$, $\frac{1}{2}(R_{0}+R_{19})$, $\frac{1}{2}(R_{2}+R_{11})$, $\frac{1}{2}(R_{3}+R_{9})$, $\frac{1}{2}(R_{7}+R_{22})$, $\frac{1}{2}(R_{11}+R_{12})$, $\frac{1}{2}(R_{11}+R_{21})$, $\frac{1}{2}(R_{1}+R_{17})$, $\frac{1}{2}(R_{0}+R_{7}+R_{11})$.\\ 
\textbf{(12)} $\mathrm{nd}(Y_{12}) = 10$. Sequence: $\frac{1}{2}(R_{1}+R_{21})$, $\frac{1}{2}(R_{2}+R_{28})$, $\frac{1}{2}(R_{4}+R_{29})$, $\frac{1}{2}(R_{5}+R_{30})$, $\frac{1}{2}(R_{9}+R_{15})$, $\frac{1}{2}(R_{0}+R_{5}+R_{29})$, $\frac{1}{2}(R_{3}+R_{5}+R_{7})$, $\frac{1}{2}(R_{0}+R_{7}+R_{1}+R_{9})$, $\frac{1}{2}(R_{1}+R_{2}+R_{5}+R_{7})$, $\frac{1}{2}(R_{1}+R_{7}+R_{5}+R_{28})$.\\ 
\textbf{(13)} $\mathrm{nd}(Y_{13}) = 10$. Sequence: $\frac{1}{2}(R_{0}+R_{19})$, $\frac{1}{2}(R_{3}+R_{9})$, $\frac{1}{2}(R_{0}+R_{1}+R_{2}+R_{8})$, $\frac{1}{2}(R_{0}+R_{1}+R_{17}+R_{3})$, $\frac{1}{2}(R_{0}+R_{1}+R_{17}+R_{9})$, $\frac{1}{2}(R_{0}+R_{1}+R_{18}+R_{8})$, $\frac{1}{2}(R_{0}+R_{6}+R_{17}+R_{8})$, $\frac{1}{2}(R_{0}+R_{7}+R_{17}+R_{8})$, $\frac{1}{2}(R_{1}+R_{10}+R_{8}+R_{17})$, $\frac{1}{2}(R_{1}+R_{16}+R_{8}+R_{17})$.\\ 
\textbf{(14)} $\mathrm{nd}(Y_{14}) = 10$. Sequence: $\frac{1}{2}(R_{0}+R_{9})$, $\frac{1}{2}(R_{0}+R_{30})$, $\frac{1}{2}(R_{1}+R_{8})$, $\frac{1}{2}(R_{2}+R_{25})$, $\frac{1}{2}(R_{3}+R_{10})$, $\frac{1}{2}(R_{10}+R_{24})$, $\frac{1}{2}(R_{11}+R_{25})$, $\frac{1}{2}(R_{12}+R_{25})$, $\frac{1}{2}(R_{13}+R_{25})$, $\frac{1}{2}(R_{25}+R_{31})$.\\ 
\textbf{(15)} $\mathrm{nd}(Y_{15}) = 10$. Sequence: $\frac{1}{2}(R_{0}+R_{12})$, $\frac{1}{2}(R_{0}+R_{16})$, $\frac{1}{2}(R_{0}+R_{65})$, $\frac{1}{2}(R_{1}+R_{15})$, $\frac{1}{2}(R_{2}+R_{86})$, $\frac{1}{2}(R_{8}+R_{47})$, $\frac{1}{2}(R_{0}+R_{66})$, $\frac{1}{2}(R_{2}+R_{19})$, $\frac{1}{2}(R_{2}+R_{87})$, $(R_{0}+R_{15})$.\\ 
\textbf{(16)} $\mathrm{nd}(Y_{16}) = 10$. Sequence: $\frac{1}{2}(R_{0}+R_{12})$, $\frac{1}{2}(R_{10}+R_{21})$, $\frac{1}{2}(R_{12}+R_{18})$, $\frac{1}{2}(R_{1}+R_{20})$, $\frac{1}{2}(R_{2}+R_{10})$, $\frac{1}{2}(R_{2}+R_{21})$, $\frac{1}{2}(R_{3}+R_{15})$, $\frac{1}{2}(R_{11}+R_{20})$, $\frac{1}{2}(R_{12}+R_{50})$, $\frac{1}{2}(R_{13}+R_{15})$.\\ 
\textbf{(17)} $\mathrm{nd}(Y_{17}) = 10$. Sequence: $\frac{1}{2}(R_{0}+R_{18})$, $\frac{1}{2}(R_{0}+R_{24})$, $\frac{1}{2}(R_{0}+R_{25})$, $\frac{1}{2}(R_{1}+R_{17})$, $\frac{1}{2}(R_{3}+R_{19})$, $\frac{1}{2}(R_{17}+R_{20})$, $\frac{1}{2}(R_{17}+R_{27})$, $\frac{1}{2}(R_{19}+R_{21})$, $\frac{1}{2}(R_{19}+R_{23})$, $\frac{1}{2}(R_{0}+R_{17}+R_{19})$.\\ 
\textbf{(18)} $\mathrm{nd}(Y_{18}) = 10$. Sequence: $\frac{1}{2}(R_{0}+R_{17})$, $\frac{1}{2}(R_{2}+R_{18})$, $\frac{1}{2}(R_{17}+R_{25})$, $\frac{1}{2}(R_{4}+R_{11})$, $\frac{1}{2}(R_{8}+R_{10})$, $\frac{1}{2}(R_{8}+R_{12})$, $\frac{1}{2}(R_{1}+R_{3}+R_{17})$, $\frac{1}{2}(R_{1}+R_{5}+R_{8})$, $\frac{1}{2}(R_{3}+R_{4}+R_{5})$, $\frac{1}{2}(R_{4}+R_{8}+R_{17})$.\\ 
\textbf{(19)} $\mathrm{nd}(Y_{19}) = 10$. Sequence: $\frac{1}{2}(R_{0}+R_{17})$, $\frac{1}{2}(R_{0}+R_{20})$, $\frac{1}{2}(R_{0}+R_{22})$, $\frac{1}{2}(R_{1}+R_{18})$, $\frac{1}{2}(R_{3}+R_{6})$, $\frac{1}{2}(R_{5}+R_{15})$, $\frac{1}{2}(R_{1}+R_{19})$, $\frac{1}{2}(R_{3}+R_{25})$, $\frac{1}{2}(R_{7}+R_{27})$, $\frac{1}{2}(R_{1}+R_{3}+R_{7})$.\\ 
\textbf{(20)} $\mathrm{nd}(Y_{20}) = 10$. Sequence: $\frac{1}{2}(R_{2}+R_{27})$, $\frac{1}{2}(R_{3}+R_{19})$, $\frac{1}{2}(R_{6}+R_{20})$, $\frac{1}{2}(R_{10}+R_{14})$, $\frac{1}{2}(R_{14}+R_{21})$, $\frac{1}{2}(R_{19}+R_{22})$, $\frac{1}{2}(R_{4}+R_{23})$, $\frac{1}{2}(R_{7}+R_{11})$, $\frac{1}{2}(R_{2}+R_{11}+R_{14})$, $\frac{1}{2}(R_{1}+R_{5}+R_{14}+R_{9})$.\\ 
\textbf{(21)} $\mathrm{nd}(Y_{21}) = 10$. Sequence: $\frac{1}{2}(R_{0}+R_{17})$, $\frac{1}{2}(R_{0}+R_{26})$, $\frac{1}{2}(R_{4}+R_{27})$, $\frac{1}{2}(R_{6}+R_{19})$, $\frac{1}{2}(R_{0}+R_{9}+R_{4}+R_{12})$, $\frac{1}{2}(R_{0}+R_{10}+R_{1}+R_{12})$, $\frac{1}{2}(R_{0}+R_{11}+R_{2}+R_{12})$, $\frac{1}{2}(R_{1}+R_{9}+R_{2}+R_{12})$, $\frac{1}{2}(R_{1}+R_{11}+R_{4}+R_{12})$, $\frac{1}{2}(R_{2}+R_{10}+R_{4}+R_{12})$.\\ 
\textbf{(22)} $\mathrm{nd}(Y_{22}) = 10$. Sequence: $\frac{1}{2}(R_{0}+R_{20})$, $\frac{1}{2}(R_{0}+R_{31})$, $\frac{1}{2}(R_{1}+R_{18})$, $\frac{1}{2}(R_{1}+R_{27})$, $\frac{1}{2}(R_{2}+R_{26})$, $\frac{1}{2}(R_{20}+R_{31})$, $\frac{1}{2}(R_{21}+R_{24})$, $\frac{1}{2}(R_{3}+R_{32})$, $\frac{1}{2}(R_{4}+R_{24})$, $\frac{1}{2}(R_{1}+R_{3}+R_{24})$.\\ 
\textbf{(23)} $\mathrm{nd}(Y_{23}) = 10$. Sequence: $\frac{1}{2}(R_{12}+R_{15})$, $\frac{1}{2}(R_{0}+R_{3}+R_{6})$, $\frac{1}{2}(R_{0}+R_{4}+R_{13})$, $\frac{1}{2}(R_{0}+R_{5}+R_{14})$, $\frac{1}{2}(R_{0}+R_{9}+R_{11})$, $\frac{1}{2}(R_{1}+R_{5}+R_{9})$, $\frac{1}{2}(R_{2}+R_{4}+R_{14})$, $\frac{1}{2}(R_{2}+R_{5}+R_{13})$, $\frac{1}{2}(R_{0}+R_{3}+R_{12}+R_{9})$, $\frac{1}{2}(R_{1}+R_{2}+R_{7}+R_{12})$.\\ 
\textbf{(24)} $\mathrm{nd}(Y_{24}) = 10$. Sequence: $\frac{1}{2}(R_{2}+R_{8})$, $\frac{1}{2}(R_{5}+R_{14})$, $\frac{1}{2}(R_{11}+R_{12})$, $\frac{1}{2}(R_{0}+R_{3}+R_{2}+R_{7})$, $\frac{1}{2}(R_{0}+R_{4}+R_{2}+R_{14})$, $\frac{1}{2}(R_{1}+R_{10}+R_{6}+R_{12})$, $\frac{1}{2}(R_{0}+R_{3}+R_{12}+R_{1}+R_{4})$, $\frac{1}{2}(R_{0}+R_{4}+R_{1}+R_{10}+R_{7})$, $\frac{1}{2}(R_{1}+R_{4}+R_{2}+R_{3}+R_{10})$, $\frac{1}{2}(R_{1}+R_{4}+R_{2}+R_{7}+R_{12})$.\\ 
\textbf{(25)} $\mathrm{nd}(Y_{25}) = 10$. Sequence: $\frac{1}{2}(R_{0}+R_{20})$, $\frac{1}{2}(R_{0}+R_{40})$, $\frac{1}{2}(R_{1}+R_{14})$, $\frac{1}{2}(R_{1}+R_{15})$, $\frac{1}{2}(R_{1}+R_{31})$, $\frac{1}{2}(R_{1}+R_{63})$, $\frac{1}{2}(R_{1}+R_{77})$, $\frac{1}{2}(R_{4}+R_{72})$, $\frac{1}{2}(R_{6}+R_{42})$, $(R_{0}+R_{7})$.\\ 
\textbf{(27)} $\mathrm{nd}(Y_{27}) = 10$. Sequence: $\frac{1}{2}(R_{0}+R_{14})$, $\frac{1}{2}(R_{0}+R_{22})$, $\frac{1}{2}(R_{0}+R_{23})$, $\frac{1}{2}(R_{1}+R_{21})$, $\frac{1}{2}(R_{1}+R_{24})$, $\frac{1}{2}(R_{2}+R_{16})$, $\frac{1}{2}(R_{2}+R_{28})$, $\frac{1}{2}(R_{3}+R_{19})$, $\frac{1}{2}(R_{4}+R_{8})$, $\frac{1}{2}(R_{10}+R_{12})$.\\ 
\textbf{(28)} $\mathrm{nd}(Y_{28}) = 10$. Sequence: $\frac{1}{2}(R_{0}+R_{15})$, $\frac{1}{2}(R_{0}+R_{20})$, $\frac{1}{2}(R_{0}+R_{23})$, $\frac{1}{2}(R_{1}+R_{17})$, $\frac{1}{2}(R_{1}+R_{18})$, $\frac{1}{2}(R_{2}+R_{16})$, $\frac{1}{2}(R_{2}+R_{19})$, $\frac{1}{2}(R_{2}+R_{28})$, $\frac{1}{2}(R_{4}+R_{9})$, $\frac{1}{2}(R_{7}+R_{8})$.\\ 
\textbf{(29)} $\mathrm{nd}(Y_{29}) = 10$. Sequence: $\frac{1}{2}(R_{0}+R_{21})$, $\frac{1}{2}(R_{0}+R_{45})$, $\frac{1}{2}(R_{0}+R_{51})$, $\frac{1}{2}(R_{1}+R_{46})$, $\frac{1}{2}(R_{1}+R_{50})$, $\frac{1}{2}(R_{2}+R_{27})$, $\frac{1}{2}(R_{3}+R_{30})$, $\frac{1}{2}(R_{3}+R_{47})$, $\frac{1}{2}(R_{5}+R_{48})$, $\frac{1}{2}(R_{4}+R_{36})$.\\ 
\textbf{(30)} $\mathrm{nd}(Y_{30}) = 10$. Sequence: $\frac{1}{2}(R_{1}+R_{42})$, $\frac{1}{2}(R_{2}+R_{17})$, $\frac{1}{2}(R_{5}+R_{54})$, $\frac{1}{2}(R_{8}+R_{38})$, $\frac{1}{2}(R_{8}+R_{55})$, $\frac{1}{2}(R_{13}+R_{52})$, $\frac{1}{2}(R_{26}+R_{41})$, $\frac{1}{2}(R_{37}+R_{52})$, $\frac{1}{2}(R_{42}+R_{51})$, $\frac{1}{2}(R_{8}+R_{42}+R_{52})$.\\ 
\textbf{(31)} $\mathrm{nd}(Y_{31}) = 10$. Sequence: $\frac{1}{2}(R_{0}+R_{20})$, $\frac{1}{2}(R_{0}+R_{21})$, $\frac{1}{2}(R_{0}+R_{27})$, $\frac{1}{2}(R_{1}+R_{23})$, $\frac{1}{2}(R_{1}+R_{25})$, $\frac{1}{2}(R_{2}+R_{22})$, $\frac{1}{2}(R_{2}+R_{26})$, $\frac{1}{2}(R_{4}+R_{31})$, $\frac{1}{2}(R_{9}+R_{31})$, $\frac{1}{2}(R_{2}+R_{6}+R_{12})$.\\ 
\textbf{(32)} $\mathrm{nd}(Y_{32}) = 10$. Sequence: $\frac{1}{2}(R_{1}+R_{2})$, $\frac{1}{2}(R_{1}+R_{17})$, $\frac{1}{2}(R_{1}+R_{29})$, $\frac{1}{2}(R_{1}+R_{5})$, $\frac{1}{2}(R_{1}+R_{6})$, $\frac{1}{2}(R_{1}+R_{7})$, $\frac{1}{2}(R_{11}+R_{33})$, $\frac{1}{2}(R_{13}+R_{21})$, $\frac{1}{2}(R_{14}+R_{22})$, $\frac{1}{2}(R_{1}+R_{18})$.\\ 
\textbf{(33)} $\mathrm{nd}(Y_{33}) = 10$. Sequence: $\frac{1}{2}(R_{0}+R_{29})$, $\frac{1}{2}(R_{0}+R_{41})$, $\frac{1}{2}(R_{17}+R_{24})$, $\frac{1}{2}(R_{1}+R_{6})$, $\frac{1}{2}(R_{1}+R_{7})$, $\frac{1}{2}(R_{1}+R_{11})$, $\frac{1}{2}(R_{1}+R_{13})$, $\frac{1}{2}(R_{1}+R_{25})$, $\frac{1}{2}(R_{1}+R_{27})$, $\frac{1}{2}(R_{1}+R_{8})$.\\ 
\textbf{(34)} $\mathrm{nd}(Y_{34}) = 10$. Sequence: $\frac{1}{2}(R_{0}+R_{38})$, $\frac{1}{2}(R_{1}+R_{21})$, $\frac{1}{2}(R_{5}+R_{11})$, $\frac{1}{2}(R_{6}+R_{21})$, $\frac{1}{2}(R_{11}+R_{22})$, $\frac{1}{2}(R_{14}+R_{21})$, $\frac{1}{2}(R_{21}+R_{33})$, $\frac{1}{2}(R_{21}+R_{39})$, $\frac{1}{2}(R_{7}+R_{21})$, $\frac{1}{2}(R_{23}+R_{30})$.\\ 
\textbf{(35)} $\mathrm{nd}(Y_{35}) = 10$. Sequence: $\frac{1}{2}(R_{0}+R_{25})$, $\frac{1}{2}(R_{0}+R_{34})$, $\frac{1}{2}(R_{2}+R_{27})$, $\frac{1}{2}(R_{2}+R_{29})$, $\frac{1}{2}(R_{3}+R_{22})$, $\frac{1}{2}(R_{3}+R_{33})$, $\frac{1}{2}(R_{0}+R_{26})$, $\frac{1}{2}(R_{1}+R_{30})$, $\frac{1}{2}(R_{1}+R_{32})$, $\frac{1}{2}(R_{0}+R_{2}+R_{3})$.\\ 
\textbf{(36)} $\mathrm{nd}(Y_{36}) = 10$. Sequence: $(R_{2}+R_{4})$, $\frac{1}{2}(R_{0}+R_{8}+R_{9})$, $\frac{1}{2}(R_{0}+R_{10}+R_{11})$, $\frac{1}{2}(R_{0}+R_{12}+R_{13})$, $\frac{1}{2}(R_{0}+R_{14}+R_{15})$, $\frac{1}{2}(R_{1}+R_{13}+R_{14})$, $\frac{1}{2}(R_{2}+R_{10}+R_{12})$, $\frac{1}{2}(R_{2}+R_{11}+R_{13})$, $\frac{1}{2}(R_{0}+R_{10}+R_{2}+R_{13})$, $\frac{1}{2}(R_{1}+R_{2}+R_{5}+R_{3})$.\\ 
\textbf{(37)} $\mathrm{nd}(Y_{37}) = 10$. Sequence: $\frac{1}{2}(R_{0}+R_{1})$, $\frac{1}{2}(R_{0}+R_{8})$, $\frac{1}{2}(R_{0}+R_{9})$, $\frac{1}{2}(R_{0}+R_{13})$, $\frac{1}{2}(R_{0}+R_{21})$, $\frac{1}{2}(R_{0}+R_{24})$, $(R_{0}+R_{16})$, $\frac{1}{2}(R_{2}+R_{6}+R_{22})$, $\frac{1}{2}(R_{2}+R_{14}+R_{17})$, $\frac{1}{2}(R_{6}+R_{12}+R_{17})$.\\ 
\textbf{(38)} $\mathrm{nd}(Y_{38}) = 10$. Sequence: $\frac{1}{2}(R_{0}+R_{1})$, $\frac{1}{2}(R_{0}+R_{2})$, $\frac{1}{2}(R_{0}+R_{12})$, $\frac{1}{2}(R_{0}+R_{13})$, $\frac{1}{2}(R_{0}+R_{9})$, $\frac{1}{2}(R_{0}+R_{10})$, $\frac{1}{2}(R_{17}+R_{21})$, $\frac{1}{2}(R_{0}+R_{15})$, $\frac{1}{2}(R_{3}+R_{5}+R_{4}+R_{11})$, $\frac{1}{2}(R_{6}+R_{8}+R_{7}+R_{14})$.\\ 
\textbf{(39)} $\mathrm{nd}(Y_{39}) = 10$. Sequence: $\frac{1}{2}(R_{0}+R_{20})$, $\frac{1}{2}(R_{0}+R_{27})$, $\frac{1}{2}(R_{0}+R_{29})$, $\frac{1}{2}(R_{1}+R_{18})$, $\frac{1}{2}(R_{1}+R_{23})$, $\frac{1}{2}(R_{1}+R_{31})$, $\frac{1}{2}(R_{2}+R_{25})$, $\frac{1}{2}(R_{3}+R_{22})$, $\frac{1}{2}(R_{3}+R_{26})$, $\frac{1}{2}(R_{0}+R_{1}+R_{3})$.\\ 
\textbf{(40)} $\mathrm{nd}(Y_{40}) = 10$. Sequence: $\frac{1}{2}(R_{0}+R_{19})$, $\frac{1}{2}(R_{0}+R_{20})$, $\frac{1}{2}(R_{0}+R_{33})$, $\frac{1}{2}(R_{1}+R_{29})$, $\frac{1}{2}(R_{1}+R_{30})$, $\frac{1}{2}(R_{2}+R_{21})$, $\frac{1}{2}(R_{2}+R_{28})$, $\frac{1}{2}(R_{3}+R_{32})$, $\frac{1}{2}(R_{4}+R_{25})$, $\frac{1}{2}(R_{0}+R_{2}+R_{6})$.\\ 
\textbf{(41)} $\mathrm{nd}(Y_{41}) = 10$. Sequence: $(R_{1}+R_{5})$, $(R_{3}+R_{7})$, $\frac{1}{2}(R_{0}+R_{7}+R_{8})$, $\frac{1}{2}(R_{1}+R_{2}+R_{6})$, $\frac{1}{2}(R_{1}+R_{9}+R_{12})$, $\frac{1}{2}(R_{1}+R_{10}+R_{11})$, $\frac{1}{2}(R_{2}+R_{7}+R_{12})$, $\frac{1}{2}(R_{6}+R_{7}+R_{9})$, $\frac{1}{2}(R_{0}+R_{11}+R_{12})$, $\frac{1}{2}(R_{1}+R_{6}+R_{7}+R_{12})$.\\ 
\textbf{(42)} $\mathrm{nd}(Y_{42}) = 10$. Sequence: $\frac{1}{2}(R_{0}+R_{19})$, $\frac{1}{2}(R_{1}+R_{13})$, $(R_{0}+R_{6})$, $\frac{1}{2}(R_{0}+R_{1}+R_{2}+R_{5})$, $\frac{1}{2}(R_{0}+R_{1}+R_{2}+R_{10})$, $\frac{1}{2}(R_{0}+R_{1}+R_{4}+R_{3})$, $\frac{1}{2}(R_{0}+R_{1}+R_{16}+R_{7})$, $\frac{1}{2}(R_{0}+R_{3}+R_{2}+R_{7})$, $\frac{1}{2}(R_{1}+R_{2}+R_{3}+R_{16})$, $\frac{1}{2}(R_{1}+R_{2}+R_{7}+R_{4})$.\\ 
\textbf{(43)} $\mathrm{nd}(Y_{43}) = 10$. Sequence: $\frac{1}{2}(R_{0}+R_{9})$, $\frac{1}{2}(R_{2}+R_{15})$, $\frac{1}{2}(R_{2}+R_{14} \cdot H_{0})$, $\frac{1}{2}(R_{1} \cdot H_{4}+R_{7} \cdot H_{4})$, $\frac{1}{2}(R_{14}+R_{15} \cdot H_{0})$, $\frac{1}{2}(R_{0}+R_{10}+R_{10} \cdot H_{0})$, $\frac{1}{2}(R_{0}+R_{10}+R_{13} \cdot H_{0})$, $\frac{1}{2}(R_{0}+R_{12}+R_{11} \cdot H_{0})$, $\frac{1}{2}(R_{0}+R_{12}+R_{12} \cdot H_{0})$, $\frac{1}{2}(R_{2}+R_{7}+R_{7} \cdot H_{4})$.\\ 
\textbf{(44)} $\mathrm{nd}(Y_{44}) = 10$. Sequence: $\frac{1}{2}(R_{0}+R_{29})$, $\frac{1}{2}(R_{4}+R_{34})$, $\frac{1}{2}(R_{6}+R_{47})$, $\frac{1}{2}(R_{10}+R_{46})$, $\frac{1}{2}(R_{14}+R_{23})$, $\frac{1}{2}(R_{34}+R_{35})$, $\frac{1}{2}(R_{5}+R_{29}+R_{36})$, $\frac{1}{2}(R_{6}+R_{22}+R_{36})$, $\frac{1}{2}(R_{10}+R_{14}+R_{34})$, $\frac{1}{2}(R_{14}+R_{34}+R_{46})$.\\ 
\textbf{(45)} $\mathrm{nd}(Y_{45}) = 10$. Sequence: $\frac{1}{2}(R_{0}+R_{11}+R_{14})$, $\frac{1}{2}(R_{1}+R_{6}+R_{13})$, $\frac{1}{2}(R_{2}+R_{6}+R_{12})$, $\frac{1}{2}(R_{3}+R_{5}+R_{12})$, $\frac{1}{2}(R_{4}+R_{5}+R_{13})$, $\frac{1}{2}(R_{5}+R_{10}+R_{14})$, $\frac{1}{2}(R_{6}+R_{9}+R_{14})$, $\frac{1}{2}(R_{7}+R_{11}+R_{13})$, $\frac{1}{2}(R_{0}+R_{7}+R_{13}+R_{14})$, $\frac{1}{2}(R_{2}+R_{3}+R_{5}+R_{6})$.\\ 
\textbf{(46)} $\mathrm{nd}(Y_{46}) = 10$. Sequence: $\frac{1}{2}(R_{0}+R_{7}+R_{6}+R_{11})$, $\frac{1}{2}(R_{0}+R_{8}+R_{6}+R_{9})$, $\frac{1}{2}(R_{3}+R_{8}+R_{10}+R_{9})$, $\frac{1}{2}(R_{0}+R_{7}+R_{4}+R_{1}+R_{3}+R_{9})$, $\frac{1}{2}(R_{0}+R_{7}+R_{4}+R_{2}+R_{3}+R_{8})$, $\frac{1}{2}(R_{0}+R_{8}+R_{3}+R_{1}+R_{4}+R_{11})$, $\frac{1}{2}(R_{0}+R_{9}+R_{3}+R_{2}+R_{4}+R_{11})$, $\frac{1}{2}(R_{1}+R_{2}+R_{7}+R_{11}+2R_{4})$, $\frac{1}{2}(R_{1}+R_{2}+R_{8}+R_{9}+2R_{3})$, $\frac{1}{2}(R_{7}+R_{8}+R_{9}+R_{11}+2R_{0})$.\\ 
\textbf{(47)} $\mathrm{nd}(Y_{47}) = 10$. Sequence: $\frac{1}{2}(R_{0}+R_{4}+R_{1}+R_{5}+R_{2}+R_{7})$, $\frac{1}{2}(R_{0}+R_{4}+R_{1}+R_{9}+R_{6}+R_{8})$, $\frac{1}{2}(R_{0}+R_{4}+R_{3}+R_{2}+R_{5}+R_{8})$, $\frac{1}{2}(R_{0}+R_{4}+R_{3}+R_{6}+R_{9}+R_{7})$, $\frac{1}{2}(R_{0}+R_{7}+R_{2}+R_{3}+R_{6}+R_{8})$, $\frac{1}{2}(R_{0}+R_{7}+R_{9}+R_{1}+R_{5}+R_{8})$, $\frac{1}{2}(R_{1}+R_{4}+R_{3}+R_{2}+R_{7}+R_{9})$, $\frac{1}{2}(R_{1}+R_{4}+R_{3}+R_{6}+R_{8}+R_{5})$, $\frac{1}{2}(R_{1}+R_{5}+R_{2}+R_{3}+R_{6}+R_{9})$, $\frac{1}{2}(R_{2}+R_{5}+R_{8}+R_{6}+R_{9}+R_{7})$.\\ 
\textbf{(50)} $\mathrm{nd}(Y_{50}) = 10$. Sequence: $\frac{1}{2}(R_{0}+R_{20})$, $\frac{1}{2}(R_{0}+R_{21})$, $\frac{1}{2}(R_{0}+R_{27})$, $\frac{1}{2}(R_{1}+R_{18})$, $\frac{1}{2}(R_{3}+R_{39})$, $\frac{1}{2}(R_{6}+R_{33})$, $\frac{1}{2}(R_{14}+R_{33})$, $\frac{1}{2}(R_{17}+R_{39})$, $\frac{1}{2}(R_{18}+R_{25})$, $\frac{1}{2}(R_{2}+R_{19})$.\\ 
\textbf{(51)} $\mathrm{nd}(Y_{51}) = 10$. Sequence: $\frac{1}{2}(R_{0}+R_{20})$, $\frac{1}{2}(R_{0}+R_{31})$, $\frac{1}{2}(R_{0}+R_{34})$, $\frac{1}{2}(R_{1}+R_{19})$, $\frac{1}{2}(R_{1}+R_{32})$, $\frac{1}{2}(R_{2}+R_{24})$, $\frac{1}{2}(R_{3}+R_{35})$, $\frac{1}{2}(R_{4}+R_{29})$, $\frac{1}{2}(R_{5}+R_{22})$, $\frac{1}{2}(R_{2}+R_{40})$.\\ 
\textbf{(52)} $\mathrm{nd}(Y_{52}) = 10$. Sequence: $\frac{1}{2}(R_{0}+R_{8})$, $\frac{1}{2}(R_{0}+R_{14})$, $\frac{1}{2}(R_{0}+R_{47})$, $\frac{1}{2}(R_{0}+R_{63})$, $\frac{1}{2}(R_{1}+R_{7})$, $\frac{1}{2}(R_{2}+R_{7})$, $\frac{1}{2}(R_{4}+R_{13})$, $\frac{1}{2}(R_{7}+R_{44})$, $\frac{1}{2}(R_{7}+R_{62})$, $\frac{1}{2}(R_{45}+R_{65})$.\\ 
\textbf{(54)} $\mathrm{nd}(Y_{54}) = 10$. Sequence: $\frac{1}{2}(R_{0}+R_{23})$, $\frac{1}{2}(R_{0}+R_{27})$, $\frac{1}{2}(R_{0}+R_{28})$, $\frac{1}{2}(R_{1}+R_{25})$, $\frac{1}{2}(R_{1}+R_{26})$, $\frac{1}{2}(R_{2}+R_{22})$, $\frac{1}{2}(R_{3}+R_{30})$, $\frac{1}{2}(R_{3}+R_{16})$, $\frac{1}{2}(R_{8}+R_{33})$, $\frac{1}{2}(R_{1}+R_{3}+R_{10})$.\\ 
\textbf{(55)} $\mathrm{nd}(Y_{55}) = 10$. Sequence: $\frac{1}{2}(R_{0}+R_{1})$, $\frac{1}{2}(R_{0}+R_{7})$, $\frac{1}{2}(R_{0}+R_{12})$, $\frac{1}{2}(R_{0}+R_{8})$, $\frac{1}{2}(R_{0}+R_{11})$, $\frac{1}{2}(R_{0}+R_{16})$, $\frac{1}{2}(R_{0}+R_{24})$, $\frac{1}{2}(R_{14}+R_{22})$, $\frac{1}{2}(R_{2}+R_{4}+R_{9})$, $\frac{1}{2}(R_{4}+R_{5}+R_{10})$.\\ 
\textbf{(56)} $\mathrm{nd}(Y_{56}) = 10$. Sequence: $\frac{1}{2}(R_{0}+R_{4})$, $\frac{1}{2}(R_{0}+R_{13})$, $\frac{1}{2}(R_{0}+R_{15})$, $\frac{1}{2}(R_{0}+R_{16})$, $\frac{1}{2}(R_{0}+R_{25})$, $\frac{1}{2}(R_{0}+R_{26})$, $\frac{1}{2}(R_{0}+R_{6})$, $\frac{1}{2}(R_{1}+R_{32})$, $\frac{1}{2}(R_{2}+R_{21})$, $\frac{1}{2}(R_{1}+R_{3}+R_{8}+R_{7})$.\\ 
\textbf{(58)} $\mathrm{nd}(Y_{58}) = 10$. Sequence: $\frac{1}{2}(R_{0}+R_{25})$, $\frac{1}{2}(R_{0}+R_{37})$, $\frac{1}{2}(R_{3}+R_{38})$, $\frac{1}{2}(R_{4}+R_{31})$, $\frac{1}{2}(R_{0}+R_{33})$, $\frac{1}{2}(R_{1}+R_{23})$, $\frac{1}{2}(R_{1}+R_{36})$, $\frac{1}{2}(R_{2}+R_{28})$, $\frac{1}{2}(R_{1}+R_{34})$, $\frac{1}{2}(R_{0}+R_{1}+R_{5})$.\\ 
\textbf{(59)} $\mathrm{nd}(Y_{59}) = 10$. Sequence: $\frac{1}{2}(R_{0}+R_{25})$, $\frac{1}{2}(R_{0}+R_{31})$, $\frac{1}{2}(R_{2}+R_{38})$, $\frac{1}{2}(R_{3}+R_{32})$, $\frac{1}{2}(R_{4}+R_{39})$, $\frac{1}{2}(R_{5}+R_{29})$, $\frac{1}{2}(R_{0}+R_{37})$, $\frac{1}{2}(R_{1}+R_{23})$, $\frac{1}{2}(R_{1}+R_{33})$, $\frac{1}{2}(R_{0}+R_{3}+R_{6})$.\\ 
\textbf{(60)} $\mathrm{nd}(Y_{60}) = 10$. Sequence: $(R_{0}+R_{12})$, $(R_{2}+R_{11})$, $\frac{1}{2}(R_{1}+R_{4}+R_{7})$, $\frac{1}{2}(R_{1}+R_{5}+R_{15})$, $\frac{1}{2}(R_{1}+R_{6}+R_{16})$, $\frac{1}{2}(R_{1}+R_{10}+R_{13})$, $\frac{1}{2}(R_{2}+R_{6}+R_{10})$, $\frac{1}{2}(R_{3}+R_{5}+R_{16})$, $\frac{1}{2}(R_{5}+R_{9}+R_{13})$, $\frac{1}{2}(R_{1}+R_{5}+R_{3}+R_{6})$.\\ 
\textbf{(61)} $\mathrm{nd}(Y_{61}) = 10$. Sequence: $\frac{1}{2}(R_{0}+R_{1})$, $\frac{1}{2}(R_{0}+R_{9})$, $\frac{1}{2}(R_{0}+R_{10})$, $\frac{1}{2}(R_{0}+R_{16})$, $\frac{1}{2}(R_{0}+R_{21})$, $\frac{1}{2}(R_{0}+R_{22})$, $(R_{0}+R_{13})$, $\frac{1}{2}(R_{2}+R_{5}+R_{11})$, $\frac{1}{2}(R_{2}+R_{8}+R_{24})$, $\frac{1}{2}(R_{7}+R_{13}+R_{18})$.\\ 
\textbf{(62)} $\mathrm{nd}(Y_{62}) = 10$. Sequence: $\frac{1}{2}(R_{0}+R_{14})$, $\frac{1}{2}(R_{2}+R_{17})$, $\frac{1}{2}(R_{2}+R_{3})$, $\frac{1}{2}(R_{2}+R_{4})$, $\frac{1}{2}(R_{2}+R_{15})$, $\frac{1}{2}(R_{2}+R_{16})$, $\frac{1}{2}(R_{2}+R_{11})$, $\frac{1}{2}(R_{2}+R_{12})$, $\frac{1}{2}(R_{5}+R_{7}+R_{6}+R_{13})$, $\frac{1}{2}(R_{7}+R_{9}+R_{13}+R_{10})$.\\ 
\textbf{(63)} $\mathrm{nd}(Y_{63}) = 10$. Sequence: $\frac{1}{2}(R_{0}+R_{11})$, $\frac{1}{2}(R_{0}+R_{15})$, $\frac{1}{2}(R_{0}+R_{20})$, $\frac{1}{2}(R_{7}+R_{19})$, $\frac{1}{2}(R_{6}+R_{13})$, $\frac{1}{2}(R_{9}+R_{18})$, $(R_{0}+R_{17})$, $\frac{1}{2}(R_{1}+R_{9}+R_{5}+R_{3}+R_{16})$, $\frac{1}{2}(R_{1}+R_{9}+R_{5}+R_{4}+R_{17})$, $\frac{1}{2}(R_{3}+R_{4}+R_{8}+R_{9}+2R_{5})$.\\ 
\textbf{(64)} $\mathrm{nd}(Y_{64}) = 10$. Sequence: $\frac{1}{2}(R_{0}+R_{25})$, $\frac{1}{2}(R_{0}+R_{33})$, $\frac{1}{2}(R_{2}+R_{29})$, $\frac{1}{2}(R_{2}+R_{32})$, $\frac{1}{2}(R_{4}+R_{21})$, $\frac{1}{2}(R_{4}+R_{36})$, $\frac{1}{2}(R_{0}+R_{26})$, $\frac{1}{2}(R_{1}+R_{30})$, $\frac{1}{2}(R_{7}+R_{39})$, $\frac{1}{2}(R_{2}+R_{7}+R_{8})$.\\ 
\textbf{(65)} $\mathrm{nd}(Y_{65}) = 10$. Sequence: $\frac{1}{2}(R_{0}+R_{24})$, $\frac{1}{2}(R_{0}+R_{30})$, $\frac{1}{2}(R_{0}+R_{31})$, $\frac{1}{2}(R_{1}+R_{29})$, $\frac{1}{2}(R_{1}+R_{32})$, $\frac{1}{2}(R_{2}+R_{28})$, $\frac{1}{2}(R_{3}+R_{33})$, $\frac{1}{2}(R_{5}+R_{39})$, $\frac{1}{2}(R_{13}+R_{39})$, $\frac{1}{2}(R_{1}+R_{10}+R_{21})$.\\ 
\textbf{(66)} $\mathrm{nd}(Y_{66}) = 10$. Sequence: $\frac{1}{2}(R_{1}+R_{2})$, $\frac{1}{2}(R_{2}+R_{19})$, $\frac{1}{2}(R_{2}+R_{21})$, $\frac{1}{2}(R_{2}+R_{7})$, $\frac{1}{2}(R_{2}+R_{10})$, $\frac{1}{2}(R_{2}+R_{22})$, $\frac{1}{2}(R_{6}+R_{12}+R_{13})$, $\frac{1}{2}(R_{6}+R_{15}+R_{18})$, $\frac{1}{2}(R_{12}+R_{15}+R_{16})$, $(R_{6}+R_{12}+R_{15})$.\\ 
\textbf{(67)} $\mathrm{nd}(Y_{67}) = 10$. Sequence: $\frac{1}{2}(R_{0}+R_{30})$, $\frac{1}{2}(R_{2}+R_{33})$, $\frac{1}{2}(R_{1}+R_{9})$, $\frac{1}{2}(R_{1}+R_{10})$, $\frac{1}{2}(R_{1}+R_{15})$, $\frac{1}{2}(R_{1}+R_{17})$, $\frac{1}{2}(R_{1}+R_{18})$, $\frac{1}{2}(R_{1}+R_{24})$, $\frac{1}{2}(R_{1}+R_{31})$, $\frac{1}{2}(R_{2}+R_{5}+R_{16}+R_{7})$.\\ 
\textbf{(68)} $\mathrm{nd}(Y_{68}) = 10$. Sequence: $\frac{1}{2}(R_{0}+R_{2})$, $\frac{1}{2}(R_{0}+R_{8})$, $\frac{1}{2}(R_{0}+R_{27})$, $\frac{1}{2}(R_{0}+R_{9})$, $\frac{1}{2}(R_{0}+R_{15})$, $\frac{1}{2}(R_{0}+R_{17})$, $\frac{1}{2}(R_{0}+R_{24})$, $\frac{1}{2}(R_{1}+R_{16})$, $\frac{1}{2}(R_{3}+R_{26})$, $\frac{1}{2}(R_{3}+R_{4}+R_{5}+R_{12})$.\\ 
\textbf{(69)} $\mathrm{nd}(Y_{69}) = 10$. Sequence: $\frac{1}{2}(R_{0}+R_{33})$, $\frac{1}{2}(R_{0}+R_{34})$, $\frac{1}{2}(R_{0}+R_{44})$, $\frac{1}{2}(R_{1}+R_{38})$, $\frac{1}{2}(R_{1}+R_{42})$, $\frac{1}{2}(R_{2}+R_{39})$, $\frac{1}{2}(R_{2}+R_{41})$, $\frac{1}{2}(R_{3}+R_{36})$, $\frac{1}{2}(R_{4}+R_{43})$, $\frac{1}{2}(R_{0}+R_{3}+R_{4})$.\\ 
\textbf{(70)} $\mathrm{nd}(Y_{70}) = 10$. Sequence: $\frac{1}{2}(R_{0}+R_{31})$, $\frac{1}{2}(R_{0}+R_{32})$, $\frac{1}{2}(R_{0}+R_{45})$, $\frac{1}{2}(R_{1}+R_{41})$, $\frac{1}{2}(R_{1}+R_{42})$, $\frac{1}{2}(R_{2}+R_{33})$, $\frac{1}{2}(R_{2}+R_{40})$, $\frac{1}{2}(R_{3}+R_{44})$, $\frac{1}{2}(R_{4}+R_{37})$, $\frac{1}{2}(R_{0}+R_{2}+R_{6})$.\\ 
\textbf{(71)} $\mathrm{nd}(Y_{71}) = 10$. Sequence: $\frac{1}{2}(R_{2}+R_{3})$, $\frac{1}{2}(R_{2}+R_{9})$, $\frac{1}{2}(R_{2}+R_{15})$, $(R_{2}+R_{4})$, $\frac{1}{2}(R_{0}+R_{11}+R_{4}+R_{12})$, $\frac{1}{2}(R_{0}+R_{13}+R_{1}+R_{14})$, $\frac{1}{2}(R_{1}+R_{6}+R_{4}+R_{7})$, $\frac{1}{2}(R_{1}+R_{8}+R_{10}+R_{14})$, $\frac{1}{2}(R_{4}+R_{5}+R_{8}+R_{7})$, $\frac{1}{2}(R_{4}+R_{6}+R_{14}+R_{11})$.\\ 
\textbf{(72)} $\mathrm{nd}(Y_{72}) = 10$. Sequence: $\frac{1}{2}(R_{10}+R_{11})$, $(R_{0}+R_{11})$, $\frac{1}{2}(R_{0}+R_{7}+R_{6}+R_{13})$, $\frac{1}{2}(R_{0}+R_{8}+R_{6}+R_{9})$, $\frac{1}{2}(R_{0}+R_{7}+R_{4}+R_{1}+R_{3}+R_{9})$, $\frac{1}{2}(R_{0}+R_{7}+R_{5}+R_{1}+R_{12}+R_{9})$, $\frac{1}{2}(R_{1}+R_{3}+R_{9}+R_{6}+R_{7}+R_{5})$, $\frac{1}{2}(R_{1}+R_{4}+R_{7}+R_{6}+R_{9}+R_{12})$, $\frac{1}{2}(R_{0}+R_{3}+R_{6}+R_{12}+2R_{9})$, $\frac{1}{2}(R_{0}+R_{4}+R_{5}+R_{6}+2R_{7})$.\\ 
\textbf{(73)} $\mathrm{nd}(Y_{73}) = 10$. Sequence: $\frac{1}{2}(R_{1}+R_{2})$, $\frac{1}{2}(R_{1}+R_{3})$, $\frac{1}{2}(R_{1}+R_{16})$, $\frac{1}{2}(R_{1}+R_{17})$, $\frac{1}{2}(R_{1}+R_{13})$, $\frac{1}{2}(R_{1}+R_{14})$, $\frac{1}{2}(R_{1}+R_{19})$, $\frac{1}{2}(R_{0}+R_{10}+R_{9}+R_{18})$, $\frac{1}{2}(R_{0}+R_{11}+R_{9}+R_{12})$, $\frac{1}{2}(R_{6}+R_{11}+R_{15}+R_{12})$.\\ 
\textbf{(74)} $\mathrm{nd}(Y_{74}) = 10$. Sequence: $(R_{3}+R_{6})$, $\frac{1}{2}(R_{0}+R_{4}+R_{1}+R_{5}+R_{2}+R_{8})$, $\frac{1}{2}(R_{0}+R_{4}+R_{1}+R_{10}+R_{7}+R_{9})$, $\frac{1}{2}(R_{0}+R_{4}+R_{3}+R_{2}+R_{5}+R_{9})$, $\frac{1}{2}(R_{0}+R_{4}+R_{3}+R_{7}+R_{10}+R_{8})$, $\frac{1}{2}(R_{0}+R_{8}+R_{2}+R_{3}+R_{7}+R_{9})$, $\frac{1}{2}(R_{1}+R_{4}+R_{3}+R_{2}+R_{8}+R_{10})$, $\frac{1}{2}(R_{1}+R_{4}+R_{3}+R_{7}+R_{9}+R_{5})$, $\frac{1}{2}(R_{1}+R_{5}+R_{2}+R_{3}+R_{7}+R_{10})$, $\frac{1}{2}(R_{2}+R_{5}+R_{9}+R_{7}+R_{10}+R_{8})$.\\ 
\textbf{(75)} $\mathrm{nd}(Y_{75}) = 10$. Sequence: $\frac{1}{2}(R_{0}+R_{28})$, $\frac{1}{2}(R_{0}+R_{29})$, $\frac{1}{2}(R_{0}+R_{42})$, $\frac{1}{2}(R_{1}+R_{38})$, $\frac{1}{2}(R_{1}+R_{39})$, $\frac{1}{2}(R_{2}+R_{30})$, $\frac{1}{2}(R_{2}+R_{37})$, $\frac{1}{2}(R_{3}+R_{41})$, $\frac{1}{2}(R_{4}+R_{34})$, $\frac{1}{2}(R_{0}+R_{2}+R_{6})$.\\ 
\textbf{(76)} $\mathrm{nd}(Y_{76}) = 10$. Sequence: $(R_{0}+R_{12})$, $(R_{1}+R_{14})$, $(R_{5}+R_{10})$, $\frac{1}{2}(R_{0}+R_{5}+R_{14})$, $\frac{1}{2}(R_{1}+R_{4}+R_{8})$, $\frac{1}{2}(R_{1}+R_{6}+R_{16})$, $\frac{1}{2}(R_{1}+R_{7}+R_{17})$, $\frac{1}{2}(R_{1}+R_{11}+R_{13})$, $\frac{1}{2}(R_{4}+R_{12}+R_{17})$, $\frac{1}{2}(R_{6}+R_{10}+R_{13})$.\\ 
\textbf{(77)} $\mathrm{nd}(Y_{77}) = 10$. Sequence: $(R_{0}+R_{13})$, $(R_{5}+R_{11})$, $(R_{6}+R_{12})$, $\frac{1}{2}(R_{0}+R_{5}+R_{12})$, $\frac{1}{2}(R_{1}+R_{6}+R_{16})$, $\frac{1}{2}(R_{1}+R_{11}+R_{14})$, $\frac{1}{2}(R_{2}+R_{7}+R_{11})$, $\frac{1}{2}(R_{3}+R_{6}+R_{17})$, $\frac{1}{2}(R_{4}+R_{13}+R_{17})$, $\frac{1}{2}(R_{6}+R_{10}+R_{14})$.\\ 
\textbf{(78)} $\mathrm{nd}(Y_{78}) = 10$. Sequence: $\frac{1}{2}(R_{0}+R_{12})$, $\frac{1}{2}(R_{8}+R_{9} \cdot H_{0})$, $\frac{1}{2}(R_{8}+R_{10} \cdot H_{0})$, $\frac{1}{2}(R_{9}+R_{11} \cdot H_{0})$, $\frac{1}{2}(R_{10}+R_{11} \cdot H_{0})$, $\frac{1}{2}(R_{2} \cdot H_{6}+R_{3} \cdot H_{6} \cdot H_{0})$, $(R_{1}+R_{11} \cdot H_{0})$, $(R_{4}+R_{8})$, $\frac{1}{2}(R_{0}+R_{8}+R_{11} \cdot H_{0})$, $\frac{1}{2}(R_{1}+R_{2}+R_{3}+R_{4}+2R_{5})$.\\ 
\textbf{(79)} $\mathrm{nd}(Y_{79}) = 10$. Sequence: $(R_{0}+R_{14})$, $(R_{1}+R_{13})$, $\frac{1}{2}(R_{0}+R_{10}+R_{11})$, $\frac{1}{2}(R_{0}+R_{12}+R_{13})$, $\frac{1}{2}(R_{1}+R_{4}+R_{15})$, $\frac{1}{2}(R_{2}+R_{9}+R_{14})$, $\frac{1}{2}(R_{1}+R_{8}+R_{9})$, $\frac{1}{2}(R_{5}+R_{8}+R_{14})$, $\frac{1}{2}(R_{6}+R_{7}+R_{14})$, $\frac{1}{2}(R_{1}+R_{5}+R_{14}+R_{9})$.\\ 
\textbf{(80)} $\mathrm{nd}(Y_{80}) = 10$. Sequence: $\frac{1}{2}(R_{1}+R_{16})$, $\frac{1}{2}(R_{6}+R_{13})$, $\frac{1}{2}(R_{9}+R_{17})$, $(R_{3}+R_{18})$, $(R_{6}+R_{12})$, $\frac{1}{2}(R_{7}+R_{9}+R_{18})$, $\frac{1}{2}(R_{7}+R_{17}+R_{18})$, $\frac{1}{2}(R_{1}+R_{3}+R_{4}+R_{12})$, $\frac{1}{2}(R_{1}+R_{3}+R_{5}+R_{8})$, $\frac{1}{2}(R_{2}+R_{3}+R_{5}+R_{12})$.\\ 
\textbf{(81)} $\mathrm{nd}(Y_{81}) = 10$. Sequence: $\frac{1}{2}(R_{0}+R_{22})$, $\frac{1}{2}(R_{1}+R_{23})$, $(R_{0}+R_{21})$, $\frac{1}{2}(R_{0}+R_{13}+R_{2}+R_{15})$, $\frac{1}{2}(R_{0}+R_{13}+R_{2}+R_{18})$, $\frac{1}{2}(R_{0}+R_{13}+R_{3}+R_{14})$, $\frac{1}{2}(R_{0}+R_{13}+R_{5}+R_{16})$, $\frac{1}{2}(R_{0}+R_{14}+R_{2}+R_{16})$, $\frac{1}{2}(R_{2}+R_{13}+R_{3}+R_{16})$, $\frac{1}{2}(R_{2}+R_{13}+R_{5}+R_{14})$.\\ 
\textbf{(82)} $\mathrm{nd}(Y_{82}) = 10$. Sequence: $(R_{0}+R_{3}+R_{9})$, $(R_{2}+R_{6}+R_{7})$, $\frac{1}{2}(R_{0}+R_{9}+R_{1}+R_{10})$, $\frac{1}{2}(R_{2}+R_{7}+R_{10}+R_{8})$, $\frac{1}{2}(R_{0}+R_{3}+R_{6}+R_{2}+R_{8})$, $\frac{1}{2}(R_{0}+R_{3}+R_{6}+R_{7}+R_{10})$, $\frac{1}{2}(R_{0}+R_{8}+R_{5}+R_{4}+R_{9})$, $\frac{1}{2}(R_{1}+R_{4}+R_{5}+R_{2}+R_{7})$, $\frac{1}{2}(R_{1}+R_{7}+R_{6}+R_{3}+R_{9})$, $\frac{1}{2}(R_{2}+R_{5}+R_{4}+R_{9}+R_{3}+R_{6})$.\\ 
\textbf{(83)} $\mathrm{nd}(Y_{83}) = 10$. Sequence: $(R_{0}+R_{10}+R_{11})$, $\frac{1}{2}(R_{0}+R_{10}+R_{2}+R_{12})$, $\frac{1}{2}(R_{0}+R_{10}+R_{5}+R_{13})$, $\frac{1}{2}(R_{0}+R_{10}+R_{8}+R_{14})$, $\frac{1}{2}(R_{0}+R_{11}+R_{3}+R_{12})$, $\frac{1}{2}(R_{0}+R_{11}+R_{4}+R_{13})$, $\frac{1}{2}(R_{0}+R_{11}+R_{7}+R_{14})$, $\frac{1}{2}(R_{1}+R_{6}+R_{13}+R_{14})$, $\frac{1}{2}(R_{1}+R_{9}+R_{12}+R_{14})$, $\frac{1}{2}(R_{6}+R_{9}+R_{12}+R_{13})$.\\ 
\textbf{(84)} $\mathrm{nd}(Y_{84}) \geq 9$. Sequence: $\frac{1}{2}(R_{0}+R_{0} \cdot H_{5})$, $\frac{1}{2}(R_{1}+R_{5} \cdot H_{2}+R_{6}+R_{1} \cdot H_{4})$, $(R_{1}+R_{5} \cdot H_{2}+R_{5}+R_{1} \cdot H_{4})$, $(R_{2}+R_{3}+R_{4}+R_{7})$, $\frac{1}{2}(R_{0}+R_{3}+R_{2}+R_{7}+R_{1}+R_{1} \cdot H_{4}+R_{5})$, $\frac{1}{2}(R_{0}+R_{3}+R_{4}+R_{7}+R_{1}+R_{5} \cdot H_{2}+R_{5})$, $\frac{1}{2}(R_{1}+R_{7}+R_{2}+R_{3}+R_{8}+R_{5}+R_{5} \cdot H_{2})$, $\frac{1}{2}(R_{1}+R_{7}+R_{4}+R_{3}+R_{8}+R_{5}+R_{1} \cdot H_{4})$, $\frac{1}{2}(R_{0}+R_{8}+R_{5} \cdot H_{2}+R_{1} \cdot H_{4}+2R_{5})$.\\ 
\textbf{(85)} $\mathrm{nd}(Y_{85}) \geq 7$. Sequence: $(R_{0}+R_{3}+R_{4}+R_{1}+R_{8}+R_{5})$, $\frac{1}{2}(R_{0}+R_{3}+R_{4}+R_{1}+R_{8}+R_{6}+R_{2}+R_{9})$, $\frac{1}{2}(R_{0}+R_{3}+R_{4}+R_{7}+R_{2}+R_{6}+R_{8}+R_{5})$, $\frac{1}{2}(R_{0}+R_{5}+R_{8}+R_{1}+R_{4}+R_{7}+R_{2}+R_{9})$, $\frac{1}{2}(R_{1}+R_{6}+R_{3}+R_{9}+2R_{8}+2R_{5}+2R_{0})$, $\frac{1}{2}(R_{1}+R_{7}+R_{5}+R_{9}+2R_{4}+2R_{3}+2R_{0})$, $\frac{1}{2}(R_{3}+R_{7}+R_{5}+R_{6}+2R_{4}+2R_{1}+2R_{8})$.\\ 
\textbf{(91)} $\mathrm{nd}(Y_{91}) = 10$. Sequence: $\frac{1}{2}(R_{7}+R_{12})$, $\frac{1}{2}(R_{0}+R_{7})$, $\frac{1}{2}(R_{1}+R_{7})$, $\frac{1}{2}(R_{7}+R_{8})$, $\frac{1}{2}(R_{7}+R_{22})$, $\frac{1}{2}(R_{7}+R_{30})$, $\frac{1}{2}(R_{7}+R_{10})$, $\frac{1}{2}(R_{15}+R_{23})$, $\frac{1}{2}(R_{16}+R_{24})$, $\frac{1}{2}(R_{6}+R_{9}+R_{13})$.\\ 
\textbf{(95)} $\mathrm{nd}(Y_{95}) = 10$. Sequence: $\frac{1}{2}(R_{0}+R_{26})$, $\frac{1}{2}(R_{1}+R_{24})$, $\frac{1}{2}(R_{2}+R_{30})$, $\frac{1}{2}(R_{3}+R_{32})$, $\frac{1}{2}(R_{0}+R_{37})$, $\frac{1}{2}(R_{4}+R_{35})$, $\frac{1}{2}(R_{0}+R_{39})$, $\frac{1}{2}(R_{1}+R_{41})$, $\frac{1}{2}(R_{2}+R_{43})$, $\frac{1}{2}(R_{0}+R_{3}+R_{4})$.\\ 
\textbf{(96)} $\mathrm{nd}(Y_{96}) = 10$. Sequence: $\frac{1}{2}(R_{1}+R_{2})$, $\frac{1}{2}(R_{2}+R_{10})$, $\frac{1}{2}(R_{2}+R_{16})$, $\frac{1}{2}(R_{2}+R_{20})$, $\frac{1}{2}(R_{2}+R_{22})$, $\frac{1}{2}(R_{2}+R_{23})$, $(R_{2}+R_{19})$, $\frac{1}{2}(R_{3}+R_{6}+R_{11})$, $\frac{1}{2}(R_{3}+R_{8}+R_{25})$, $\frac{1}{2}(R_{5}+R_{8}+R_{24})$.\\ 
\textbf{(97)} $\mathrm{nd}(Y_{97}) = 10$. Sequence: $\frac{1}{2}(R_{1}+R_{4})$, $\frac{1}{2}(R_{1}+R_{5})$, $\frac{1}{2}(R_{1}+R_{19})$, $\frac{1}{2}(R_{2}+R_{12})$, $\frac{1}{2}(R_{2}+R_{16})$, $\frac{1}{2}(R_{1}+R_{24})$, $\frac{1}{2}(R_{14}+R_{15})$, $\frac{1}{2}(R_{0}+R_{9}+R_{8}+R_{23})$, $\frac{1}{2}(R_{0}+R_{10}+R_{8}+R_{11})$, $\frac{1}{2}(R_{7}+R_{27}+R_{18}+R_{28})$.\\ 
\textbf{(98)} $\mathrm{nd}(Y_{98}) = 10$. Sequence: $\frac{1}{2}(R_{0}+R_{29})$, $\frac{1}{2}(R_{0}+R_{34})$, $\frac{1}{2}(R_{1}+R_{31})$, $\frac{1}{2}(R_{3}+R_{25})$, $\frac{1}{2}(R_{4}+R_{16})$, $\frac{1}{2}(R_{5}+R_{15})$, $\frac{1}{2}(R_{0}+R_{32})$, $\frac{1}{2}(R_{4}+R_{39})$, $\frac{1}{2}(R_{8}+R_{43})$, $\frac{1}{2}(R_{0}+R_{5}+R_{10})$.\\ 
\textbf{(99)} $\mathrm{nd}(Y_{99}) = 10$. Sequence: $\frac{1}{2}(R_{1}+R_{2})$, $\frac{1}{2}(R_{1}+R_{8})$, $\frac{1}{2}(R_{1}+R_{13})$, $\frac{1}{2}(R_{1}+R_{9})$, $\frac{1}{2}(R_{1}+R_{19})$, $\frac{1}{2}(R_{1}+R_{20})$, $\frac{1}{2}(R_{3}+R_{5}+R_{10})$, $\frac{1}{2}(R_{3}+R_{7}+R_{22})$, $\frac{1}{2}(R_{7}+R_{10}+R_{14})$, $(R_{3}+R_{7}+R_{10})$.\\ 
\textbf{(100)} $\mathrm{nd}(Y_{100}) = 10$. Sequence: $\frac{1}{2}(R_{0}+R_{42})$, $\frac{1}{2}(R_{1}+R_{10})$, $\frac{1}{2}(R_{1}+R_{11})$, $\frac{1}{2}(R_{1}+R_{15})$, $\frac{1}{2}(R_{1}+R_{37})$, $\frac{1}{2}(R_{2}+R_{8})$, $\frac{1}{2}(R_{5}+R_{8})$, $\frac{1}{2}(R_{8}+R_{9})$, $\frac{1}{2}(R_{8}+R_{34})$, $\frac{1}{2}(R_{17}+R_{35})$.\\ 
\textbf{(101)} $\mathrm{nd}(Y_{101}) = 10$. Sequence: $\frac{1}{2}(R_{0}+R_{35})$, $\frac{1}{2}(R_{0}+R_{36})$, $\frac{1}{2}(R_{0}+R_{48})$, $\frac{1}{2}(R_{1}+R_{44})$, $\frac{1}{2}(R_{1}+R_{47})$, $\frac{1}{2}(R_{2}+R_{37})$, $\frac{1}{2}(R_{2}+R_{43})$, $\frac{1}{2}(R_{3}+R_{49})$, $\frac{1}{2}(R_{5}+R_{33})$, $\frac{1}{2}(R_{0}+R_{2}+R_{6})$.\\ 
\textbf{(102)} $\mathrm{nd}(Y_{102}) = 10$. Sequence: $\frac{1}{2}(R_{0}+R_{10})$, $\frac{1}{2}(R_{0}+R_{18})$, $\frac{1}{2}(R_{0}+R_{19})$, $\frac{1}{2}(R_{0}+R_{32})$, $\frac{1}{2}(R_{1}+R_{14})$, $\frac{1}{2}(R_{5}+R_{14})$, $\frac{1}{2}(R_{6}+R_{14})$, $\frac{1}{2}(R_{14}+R_{30})$, $\frac{1}{2}(R_{2}+R_{22})$, $\frac{1}{2}(R_{4}+R_{38})$.\\ 
\textbf{(105)} $\mathrm{nd}(Y_{105}) = 10$. Sequence: $\frac{1}{2}(R_{2}+R_{5})$, $\frac{1}{2}(R_{2}+R_{19})$, $\frac{1}{2}(R_{3}+R_{6})$, $\frac{1}{2}(R_{3}+R_{18})$, $\frac{1}{2}(R_{2}+R_{12})$, $(R_{2}+R_{13})$, $\frac{1}{2}(R_{0}+R_{13}+R_{7}+R_{15})$, $\frac{1}{2}(R_{1}+R_{8}+R_{15}+R_{16})$, $\frac{1}{2}(R_{7}+R_{8}+R_{17}+R_{13})$, $\frac{1}{2}(R_{8}+R_{10}+R_{14}+R_{15})$.\\ 
\textbf{(106)} $\mathrm{nd}(Y_{106}) = 10$. Sequence: $\frac{1}{2}(R_{0}+R_{17})$, $\frac{1}{2}(R_{2}+R_{3})$, $\frac{1}{2}(R_{2}+R_{14})$, $\frac{1}{2}(R_{2}+R_{15})$, $\frac{1}{2}(R_{2}+R_{21})$, $\frac{1}{2}(R_{2}+R_{4})$, $\frac{1}{2}(R_{2}+R_{18})$, $\frac{1}{2}(R_{2}+R_{19})$, $\frac{1}{2}(R_{1}+R_{11}+R_{10}+R_{20})$, $\frac{1}{2}(R_{7}+R_{12}+R_{16}+R_{13})$.\\ 
\textbf{(107)} $\mathrm{nd}(Y_{107}) = 10$. Sequence: $\frac{1}{2}(R_{10}+R_{13})$, $(R_{4}+R_{12})$, $\frac{1}{2}(R_{0}+R_{8}+R_{6}+R_{9})$, $\frac{1}{2}(R_{3}+R_{8}+R_{14}+R_{9})$, $\frac{1}{2}(R_{0}+R_{7}+R_{4}+R_{1}+R_{14}+R_{8})$, $\frac{1}{2}(R_{0}+R_{7}+R_{5}+R_{1}+R_{3}+R_{8})$, $\frac{1}{2}(R_{1}+R_{3}+R_{8}+R_{6}+R_{7}+R_{4})$, $\frac{1}{2}(R_{1}+R_{5}+R_{7}+R_{6}+R_{8}+R_{14})$, $\frac{1}{2}(R_{0}+R_{4}+R_{5}+R_{6}+2R_{7})$, $\frac{1}{2}(R_{3}+R_{4}+R_{5}+R_{14}+2R_{1})$.\\ 
\textbf{(109)} $\mathrm{nd}(Y_{109}) = 10$. Sequence: $(R_{2}+R_{6})$, $(R_{5}+R_{8})$, $\frac{1}{2}(R_{0}+R_{4}+R_{1}+R_{5}+R_{2}+R_{9})$, $\frac{1}{2}(R_{0}+R_{4}+R_{3}+R_{2}+R_{5}+R_{10})$, $\frac{1}{2}(R_{0}+R_{9}+R_{2}+R_{3}+R_{7}+R_{10})$, $\frac{1}{2}(R_{0}+R_{9}+R_{11}+R_{1}+R_{5}+R_{10})$, $\frac{1}{2}(R_{1}+R_{4}+R_{3}+R_{2}+R_{9}+R_{11})$, $\frac{1}{2}(R_{1}+R_{4}+R_{3}+R_{7}+R_{10}+R_{5})$, $\frac{1}{2}(R_{1}+R_{5}+R_{2}+R_{3}+R_{7}+R_{11})$, $\frac{1}{2}(R_{2}+R_{5}+R_{10}+R_{7}+R_{11}+R_{9})$.\\ 
\textbf{(110)} $\mathrm{nd}(Y_{110}) = 10$. Sequence: $\frac{1}{2}(R_{0}+R_{31})$, $\frac{1}{2}(R_{0}+R_{40})$, $\frac{1}{2}(R_{0}+R_{42})$, $\frac{1}{2}(R_{1}+R_{29})$, $\frac{1}{2}(R_{1}+R_{35})$, $\frac{1}{2}(R_{1}+R_{43})$, $\frac{1}{2}(R_{2}+R_{39})$, $\frac{1}{2}(R_{3}+R_{34})$, $\frac{1}{2}(R_{3}+R_{36})$, $\frac{1}{2}(R_{0}+R_{1}+R_{3})$.\\ 
\textbf{(111)} $\mathrm{nd}(Y_{111}) = 10$. Sequence: $(R_{2}+R_{8})$, $(R_{5}+R_{16})$, $(R_{7}+R_{18})$, $\frac{1}{2}(R_{0}+R_{15}+R_{16})$, $\frac{1}{2}(R_{0}+R_{17}+R_{18})$, $\frac{1}{2}(R_{1}+R_{15}+R_{18})$, $\frac{1}{2}(R_{1}+R_{16}+R_{17})$, $\frac{1}{2}(R_{3}+R_{14}+R_{16})$, $\frac{1}{2}(R_{9}+R_{13}+R_{18})$, $\frac{1}{2}(R_{0}+R_{13}+R_{14})$.\\ 
\textbf{(112)} $\mathrm{nd}(Y_{112}) = 10$. Sequence: $(R_{0}+R_{16})$, $(R_{5}+R_{15})$, $(R_{6}+R_{14})$, $\frac{1}{2}(R_{0}+R_{14}+R_{15})$, $\frac{1}{2}(R_{1}+R_{5}+R_{17})$, $\frac{1}{2}(R_{1}+R_{6}+R_{18})$, $\frac{1}{2}(R_{2}+R_{6}+R_{10})$, $\frac{1}{2}(R_{3}+R_{5}+R_{18})$, $\frac{1}{2}(R_{3}+R_{6}+R_{17})$, $\frac{1}{2}(R_{5}+R_{9}+R_{13})$.\\ 
\textbf{(113)} $\mathrm{nd}(Y_{113}) = 10$. Sequence: $\frac{1}{2}(R_{0}+R_{23})$, $\frac{1}{2}(R_{1}+R_{23})$, $\frac{1}{2}(R_{9}+R_{23})$, $\frac{1}{2}(R_{10}+R_{23})$, $\frac{1}{2}(R_{21}+R_{23})$, $\frac{1}{2}(R_{23}+R_{24})$, $(R_{2}+R_{23})$, $\frac{1}{2}(R_{2}+R_{5}+R_{11})$, $\frac{1}{2}(R_{2}+R_{7}+R_{25})$, $\frac{1}{2}(R_{2}+R_{8}+R_{26})$.\\ 
\textbf{(114)} $\mathrm{nd}(Y_{114}) = 10$. Sequence: $\frac{1}{2}(R_{2}+R_{3})$, $\frac{1}{2}(R_{2}+R_{6})$, $\frac{1}{2}(R_{2}+R_{7})$, $\frac{1}{2}(R_{2}+R_{23})$, $\frac{1}{2}(R_{2}+R_{24})$, $\frac{1}{2}(R_{2}+R_{27})$, $(R_{2}+R_{17})$, $\frac{1}{2}(R_{0}+R_{17}+R_{18})$, $\frac{1}{2}(R_{1}+R_{20}+R_{21})$, $\frac{1}{2}(R_{4}+R_{17}+R_{19})$.\\ 
\textbf{(115)} $\mathrm{nd}(Y_{115}) = 10$. Sequence: $\frac{1}{2}(R_{1}+R_{7})$, $\frac{1}{2}(R_{12}+R_{15})$, $\frac{1}{2}(R_{13}+R_{14})$, $\frac{1}{2}(R_{2}+R_{3})$, $\frac{1}{2}(R_{2}+R_{11})$, $\frac{1}{2}(R_{2}+R_{21})$, $\frac{1}{2}(R_{2}+R_{4})$, $\frac{1}{2}(R_{2}+R_{18})$, $\frac{1}{2}(R_{2}+R_{16})$, $\frac{1}{2}(R_{2}+R_{19})$.\\ 
\textbf{(116)} $\mathrm{nd}(Y_{116}) = 10$. Sequence: $(R_{0}+R_{16})$, $(R_{2}+R_{11})$, $(R_{3}+R_{10})$, $\frac{1}{2}(R_{0}+R_{10}+R_{11})$, $\frac{1}{2}(R_{1}+R_{13}+R_{14})$, $\frac{1}{2}(R_{5}+R_{10}+R_{12})$, $\frac{1}{2}(R_{0}+R_{12}+R_{13})$, $\frac{1}{2}(R_{1}+R_{5}+R_{8})$, $\frac{1}{2}(R_{1}+R_{12}+R_{15})$, $\frac{1}{2}(R_{4}+R_{9}+R_{16})$.\\ 
\textbf{(117)} $\mathrm{nd}(Y_{117}) = 10$. Sequence: $\frac{1}{2}(R_{0}+R_{35})$, $\frac{1}{2}(R_{4}+R_{34})$, $\frac{1}{2}(R_{1}+R_{8})$, $\frac{1}{2}(R_{1}+R_{9})$, $\frac{1}{2}(R_{1}+R_{14})$, $\frac{1}{2}(R_{1}+R_{16})$, $\frac{1}{2}(R_{1}+R_{17})$, $\frac{1}{2}(R_{1}+R_{30})$, $\frac{1}{2}(R_{1}+R_{36})$, $\frac{1}{2}(R_{0}+R_{21}+R_{12}+R_{24})$.\\ 
\textbf{(118)} $\mathrm{nd}(Y_{118}) = 10$. Sequence: $\frac{1}{2}(R_{2}+R_{8})$, $\frac{1}{2}(R_{2}+R_{10})$, $\frac{1}{2}(R_{2}+R_{11})$, $\frac{1}{2}(R_{2}+R_{12})$, $\frac{1}{2}(R_{0}+R_{5}+R_{7}+R_{16})$, $\frac{1}{2}(R_{3}+R_{4}+R_{16}+R_{13})$, $\frac{1}{2}(R_{0}+R_{4}+R_{3}+R_{1}+R_{14})$, $\frac{1}{2}(R_{0}+R_{4}+R_{9}+R_{6}+R_{5})$, $\frac{1}{2}(R_{1}+R_{13}+R_{7}+R_{5}+R_{14})$, $\frac{1}{2}(R_{3}+R_{9}+R_{6}+R_{7}+R_{13})$.\\ 
\textbf{(119)} $\mathrm{nd}(Y_{119}) = 10$. Sequence: $\frac{1}{2}(R_{4}+R_{7})$, $\frac{1}{2}(R_{4}+R_{12})$, $\frac{1}{2}(R_{4}+R_{19})$, $(R_{4}+R_{6})$, $\frac{1}{2}(R_{0}+R_{13}+R_{6}+R_{16})$, $\frac{1}{2}(R_{0}+R_{14}+R_{3}+R_{15})$, $\frac{1}{2}(R_{0}+R_{14}+R_{9}+R_{17})$, $\frac{1}{2}(R_{0}+R_{17}+R_{1}+R_{18})$, $\frac{1}{2}(R_{1}+R_{11}+R_{15}+R_{17})$, $\frac{1}{2}(R_{3}+R_{11}+R_{18}+R_{14})$.\\ 
\textbf{(120)} $\mathrm{nd}(Y_{120}) = 10$. Sequence: $\frac{1}{2}(R_{3}+R_{5})$, $\frac{1}{2}(R_{3}+R_{12})$, $\frac{1}{2}(R_{3}+R_{19})$, $(R_{3}+R_{7})$, $\frac{1}{2}(R_{0}+R_{13}+R_{7}+R_{16})$, $\frac{1}{2}(R_{0}+R_{14}+R_{4}+R_{15})$, $\frac{1}{2}(R_{0}+R_{15}+R_{11}+R_{18})$, $\frac{1}{2}(R_{0}+R_{17}+R_{1}+R_{18})$, $\frac{1}{2}(R_{1}+R_{9}+R_{14}+R_{18})$, $\frac{1}{2}(R_{4}+R_{9}+R_{17}+R_{15})$.\\ 
\textbf{(121)} $\mathrm{nd}(Y_{121}) \geq 9$. Sequence: $\frac{1}{2}(R_{1}+R_{9})$, $\frac{1}{2}(R_{9}+R_{10})$, $\frac{1}{2}(R_{3}+R_{9})$, $\frac{1}{2}(R_{9}+R_{3} \cdot H_{2})$, $\frac{1}{2}(R_{11}+R_{11} \cdot H_{0})$, $\frac{1}{2}(R_{0}+R_{2}+R_{4}+R_{5}+R_{7}+R_{8}+R_{8} \cdot H_{2})$, $\frac{1}{2}(R_{0}+R_{5}+R_{6}+R_{7}+R_{8}+R_{12}+R_{8} \cdot H_{2})$, $\frac{1}{2}(R_{0}+R_{2}+2R_{7}+R_{11}+R_{12})$, $\frac{1}{2}(R_{0}+R_{4}+2R_{5}+R_{6}+R_{11})$.\\ 
\textbf{(122)} $\mathrm{nd}(Y_{122}) \geq 7$. Sequence: $\frac{1}{2}(R_{2}+R_{9})$, $(R_{2}+R_{7})$, $(R_{0}+R_{4}+R_{5}+R_{1}+R_{10}+R_{6})$, $\frac{1}{2}(R_{0}+R_{4}+R_{5}+R_{1}+R_{10}+R_{7}+R_{3}+R_{11})$, $\frac{1}{2}(R_{0}+R_{4}+R_{5}+R_{8}+R_{3}+R_{7}+R_{10}+R_{6})$, $\frac{1}{2}(R_{1}+R_{7}+R_{4}+R_{11}+2R_{10}+2R_{6}+2R_{0})$, $\frac{1}{2}(R_{4}+R_{8}+R_{6}+R_{7}+2R_{5}+2R_{1}+2R_{10})$.\\ 
\textbf{(123)} $\mathrm{nd}(Y_{123}) \geq 7$. Sequence: $\frac{1}{2}(R_{2}+R_{3})$, $\frac{1}{2}(R_{2}+R_{10})$, $\frac{1}{2}(R_{2}+R_{11})$, $(R_{2}+R_{12})$, $\frac{1}{2}(R_{0}+R_{5}+R_{6}+R_{1}+R_{12}+R_{8}+R_{4}+R_{13})$, $\frac{1}{2}(R_{0}+R_{5}+R_{6}+R_{9}+R_{4}+R_{8}+R_{12}+R_{7})$, $\frac{1}{2}(R_{0}+R_{7}+R_{12}+R_{1}+R_{6}+R_{9}+R_{4}+R_{13})$.\\ 
\textbf{(124)} $\mathrm{nd}(Y_{124}) = 10$. Sequence: $\frac{1}{2}(R_{0}+R_{28})$, $\frac{1}{2}(R_{0}+R_{29})$, $\frac{1}{2}(R_{0}+R_{43})$, $\frac{1}{2}(R_{1}+R_{37})$, $\frac{1}{2}(R_{1}+R_{41})$, $\frac{1}{2}(R_{2}+R_{38})$, $\frac{1}{2}(R_{2}+R_{40})$, $\frac{1}{2}(R_{3}+R_{31})$, $\frac{1}{2}(R_{6}+R_{42})$, $\frac{1}{2}(R_{0}+R_{3}+R_{6})$.\\ 
\textbf{(125)} $\mathrm{nd}(Y_{125}) = 10$. Sequence: $\frac{1}{2}(R_{0}+R_{30})$, $\frac{1}{2}(R_{0}+R_{31})$, $\frac{1}{2}(R_{0}+R_{41})$, $\frac{1}{2}(R_{1}+R_{37})$, $\frac{1}{2}(R_{1}+R_{40})$, $\frac{1}{2}(R_{2}+R_{32})$, $\frac{1}{2}(R_{2}+R_{36})$, $\frac{1}{2}(R_{4}+R_{25})$, $\frac{1}{2}(R_{4}+R_{39})$, $\frac{1}{2}(R_{0}+R_{2}+R_{4})$.\\ 
\textbf{(126)} $\mathrm{nd}(Y_{126}) = 10$. Sequence: $\frac{1}{2}(R_{13}+R_{20})$, $(R_{0}+R_{14})$, $(R_{1}+R_{7})$, $(R_{2}+R_{12})$, $(R_{5}+R_{10})$, $\frac{1}{2}(R_{0}+R_{12}+R_{13})$, $\frac{1}{2}(R_{1}+R_{10}+R_{13})$, $\frac{1}{2}(R_{2}+R_{5}+R_{15})$, $\frac{1}{2}(R_{0}+R_{9}+R_{4}+R_{10})$, $\frac{1}{2}(R_{1}+R_{4}+R_{9}+R_{12})$.\\ 
\textbf{(127)} $\mathrm{nd}(Y_{127}) = 10$. Sequence: $\frac{1}{2}(R_{0}+R_{38})$, $\frac{1}{2}(R_{0}+R_{39})$, $\frac{1}{2}(R_{0}+R_{50})$, $\frac{1}{2}(R_{1}+R_{46})$, $\frac{1}{2}(R_{1}+R_{47})$, $\frac{1}{2}(R_{2}+R_{40})$, $\frac{1}{2}(R_{2}+R_{45})$, $\frac{1}{2}(R_{3}+R_{51})$, $\frac{1}{2}(R_{5}+R_{43})$, $\frac{1}{2}(R_{0}+R_{2}+R_{6})$.\\ 
\textbf{(128)} $\mathrm{nd}(Y_{128}) = 10$. Sequence: $\frac{1}{2}(R_{2}+R_{7})$, $\frac{1}{2}(R_{2}+R_{14})$, $\frac{1}{2}(R_{2}+R_{22})$, $\frac{1}{2}(R_{4}+R_{11})$, $\frac{1}{2}(R_{4}+R_{21})$, $\frac{1}{2}(R_{5}+R_{6})$, $\frac{1}{2}(R_{5}+R_{23})$, $\frac{1}{2}(R_{6}+R_{23})$, $\frac{1}{2}(R_{11}+R_{21})$, $(R_{2}+R_{16})$.\\ 
\textbf{(129)} $\mathrm{nd}(Y_{129}) = 10$. Sequence: $\frac{1}{2}(R_{1}+R_{8})$, $\frac{1}{2}(R_{12}+R_{15})$, $\frac{1}{2}(R_{13}+R_{14})$, $(R_{0}+R_{15})$, $(R_{1}+R_{4})$, $(R_{5}+R_{14})$, $\frac{1}{2}(R_{1}+R_{14}+R_{15})$, $\frac{1}{2}(R_{0}+R_{9}+R_{5}+R_{3}+R_{16}+R_{11})$, $\frac{1}{2}(R_{0}+R_{9}+R_{6}+R_{3}+R_{4}+R_{11})$, $\frac{1}{2}(R_{3}+R_{4}+R_{11}+R_{7}+R_{9}+R_{5})$.\\ 
\textbf{(130)} $\mathrm{nd}(Y_{130}) = 10$. Sequence: $(R_{1}+R_{12})$, $(R_{2}+R_{10})$, $(R_{3}+R_{7})$, $\frac{1}{2}(R_{2}+R_{7}+R_{12})$, $\frac{1}{2}(R_{0}+R_{5}+R_{1}+R_{6}+R_{3}+R_{9})$, $\frac{1}{2}(R_{0}+R_{5}+R_{4}+R_{3}+R_{6}+R_{10})$, $\frac{1}{2}(R_{0}+R_{9}+R_{11}+R_{1}+R_{6}+R_{10})$, $\frac{1}{2}(R_{1}+R_{5}+R_{4}+R_{8}+R_{10}+R_{6})$, $\frac{1}{2}(R_{1}+R_{6}+R_{3}+R_{4}+R_{8}+R_{11})$, $\frac{1}{2}(R_{3}+R_{6}+R_{10}+R_{8}+R_{11}+R_{9})$.\\ 
\textbf{(131)} $\mathrm{nd}(Y_{131}) = 10$. Sequence: $(R_{0}+R_{12})$, $(R_{3}+R_{11})$, $(R_{5}+R_{8})$, $\frac{1}{2}(R_{0}+R_{3}+R_{8})$, $\frac{1}{2}(R_{1}+R_{6}+R_{2}+R_{12}+R_{9}+R_{11})$, $\frac{1}{2}(R_{1}+R_{6}+R_{5}+R_{9}+R_{12}+R_{10})$, $\frac{1}{2}(R_{1}+R_{10}+R_{4}+R_{5}+R_{9}+R_{11})$, $\frac{1}{2}(R_{2}+R_{6}+R_{5}+R_{9}+R_{11}+R_{7})$, $\frac{1}{2}(R_{2}+R_{7}+R_{4}+R_{5}+R_{9}+R_{12})$, $\frac{1}{2}(R_{4}+R_{7}+R_{11}+R_{9}+R_{12}+R_{10})$.\\ 
\textbf{(132)} $\mathrm{nd}(Y_{132}) = 10$. Sequence: $(R_{0}+R_{20})$, $(R_{1}+R_{10})$, $(R_{3}+R_{11})$, $(R_{7}+R_{15})$, $\frac{1}{2}(R_{0}+R_{12}+R_{13})$, $\frac{1}{2}(R_{2}+R_{10}+R_{20})$, $\frac{1}{2}(R_{0}+R_{14}+R_{15})$, $\frac{1}{2}(R_{0}+R_{12}+R_{11}+R_{17})$, $\frac{1}{2}(R_{1}+R_{9}+R_{15}+R_{17})$, $\frac{1}{2}(R_{0}+R_{12}+R_{11}+R_{1}+R_{9}+R_{15})$.\\ 
\textbf{(133)} $\mathrm{nd}(Y_{133}) = 10$. Sequence: $(R_{0}+R_{12})$, $(R_{2}+R_{9})$, $(R_{3}+R_{11})$, $\frac{1}{2}(R_{1}+R_{6}+R_{10}+R_{12})$, $\frac{1}{2}(R_{2}+R_{4}+R_{8}+R_{7})$, $\frac{1}{2}(R_{1}+R_{3}+R_{4}+R_{2}+R_{6})$, $\frac{1}{2}(R_{1}+R_{3}+R_{4}+R_{8}+R_{12})$, $\frac{1}{2}(R_{1}+R_{3}+R_{7}+R_{2}+R_{12})$, $\frac{1}{2}(R_{1}+R_{5}+R_{4}+R_{2}+R_{12})$, $\frac{1}{2}(R_{2}+R_{4}+R_{3}+R_{10}+R_{12})$.\\ 
\textbf{(134)} $\mathrm{nd}(Y_{134}) = 10$. Sequence: $\frac{1}{2}(R_{0}+R_{18})$, $\frac{1}{2}(R_{10}+R_{17})$, $\frac{1}{2}(R_{12}+R_{15})$, $(R_{0}+R_{19})$, $(R_{3}+R_{11})$, $\frac{1}{2}(R_{0}+R_{10}+R_{13})$, $\frac{1}{2}(R_{0}+R_{11}+R_{12})$, $\frac{1}{2}(R_{8}+R_{10}+R_{12})$, $\frac{1}{2}(R_{8}+R_{11}+R_{13})$, $\frac{1}{2}(R_{3}+R_{4}+R_{9}+R_{19}+2R_{5})$.\\ 
\textbf{(135)} $\mathrm{nd}(Y_{135}) = 10$. Sequence: $\frac{1}{2}(R_{0}+R_{18})$, $\frac{1}{2}(R_{9}+R_{16})$, $\frac{1}{2}(R_{11}+R_{14})$, $(R_{0}+R_{17})$, $(R_{1}+R_{15})$, $\frac{1}{2}(R_{0}+R_{9}+R_{12})$, $\frac{1}{2}(R_{0}+R_{14}+R_{15})$, $\frac{1}{2}(R_{2}+R_{9}+R_{14})$, $\frac{1}{2}(R_{2}+R_{12}+R_{15})$, $\frac{1}{2}(R_{1}+R_{3}+R_{4}+R_{17}+2R_{8})$.\\ 
\textbf{(136)} $\mathrm{nd}(Y_{136}) = 10$. Sequence: $(R_{0}+R_{18})$, $(R_{3}+R_{15})$, $\frac{1}{2}(R_{0}+R_{12}+R_{13})$, $\frac{1}{2}(R_{0}+R_{14}+R_{15})$, $\frac{1}{2}(R_{1}+R_{5}+R_{10})$, $\frac{1}{2}(R_{2}+R_{6}+R_{18})$, $\frac{1}{2}(R_{3}+R_{11}+R_{19})$, $\frac{1}{2}(R_{4}+R_{7}+R_{18})$, $\frac{1}{2}(R_{8}+R_{9}+R_{18})$, $\frac{1}{2}(R_{3}+R_{4}+R_{18}+R_{9})$.\\ 
\textbf{(137)} $\mathrm{nd}(Y_{137}) = 10$. Sequence: $(R_{0}+R_{18})$, $(R_{1}+R_{10})$, $\frac{1}{2}(R_{0}+R_{12}+R_{13})$, $\frac{1}{2}(R_{0}+R_{14}+R_{15})$, $\frac{1}{2}(R_{0}+R_{16}+R_{17})$, $\frac{1}{2}(R_{1}+R_{5}+R_{9})$, $\frac{1}{2}(R_{1}+R_{14}+R_{17})$, $\frac{1}{2}(R_{1}+R_{15}+R_{16})$, $\frac{1}{2}(R_{3}+R_{10}+R_{18})$, $\frac{1}{2}(R_{0}+R_{14}+R_{1}+R_{16})$.\\ 
\textbf{(138)} $\mathrm{nd}(Y_{138}) = 10$. Sequence: $\frac{1}{2}(R_{7}+R_{11})$, $(R_{0}+R_{10})$, $(R_{1}+R_{6})$, $(R_{2}+R_{9})$, $\frac{1}{2}(R_{0}+R_{8}+R_{5}+R_{12})$, $\frac{1}{2}(R_{4}+R_{8}+R_{5}+R_{13})$, $\frac{1}{2}(R_{4}+R_{8}+R_{6}+R_{12})$, $\frac{1}{2}(R_{4}+R_{9}+R_{5}+R_{12})$, $\frac{1}{2}(R_{5}+R_{8}+R_{6}+R_{9})$, $\frac{1}{2}(R_{1}+R_{3}+R_{2}+R_{7}+R_{10})$.\\ 
\textbf{(139)} $\mathrm{nd}(Y_{139}) = 10$. Sequence: $(R_{0}+R_{2}+R_{10})$, $(R_{1}+R_{7}+R_{11})$, $(R_{4}+R_{6}+R_{9})$, $\frac{1}{2}(R_{0}+R_{2}+R_{9}+R_{6}+R_{5})$, $\frac{1}{2}(R_{0}+R_{5}+R_{7}+R_{1}+R_{10})$, $\frac{1}{2}(R_{1}+R_{7}+R_{6}+R_{4}+R_{8})$, $\frac{1}{2}(R_{1}+R_{10}+R_{2}+R_{3}+R_{11})$, $\frac{1}{2}(R_{2}+R_{9}+R_{4}+R_{8}+R_{10})$, $\frac{1}{2}(R_{3}+R_{9}+R_{6}+R_{7}+R_{11})$, $\frac{1}{2}(R_{1}+R_{7}+R_{6}+R_{9}+R_{2}+R_{10})$.\\ 
\textbf{(140)} $\mathrm{nd}(Y_{140}) = 10$. Sequence: $(R_{0}+R_{3}+R_{10})$, $(R_{1}+R_{6}+R_{11})$, $(R_{5}+R_{7}+R_{9})$, $\frac{1}{2}(R_{0}+R_{3}+R_{2}+R_{1}+R_{11})$, $\frac{1}{2}(R_{0}+R_{3}+R_{9}+R_{5}+R_{4})$, $\frac{1}{2}(R_{0}+R_{10}+R_{8}+R_{6}+R_{11})$, $\frac{1}{2}(R_{1}+R_{2}+R_{9}+R_{7}+R_{6})$, $\frac{1}{2}(R_{3}+R_{9}+R_{7}+R_{8}+R_{10})$, $\frac{1}{2}(R_{4}+R_{5}+R_{7}+R_{6}+R_{11})$, $\frac{1}{2}(R_{0}+R_{3}+R_{9}+R_{7}+R_{6}+R_{11})$.\\ 
\textbf{(141)} $\mathrm{nd}(Y_{141}) = 10$. Sequence: $\frac{1}{2}(R_{6}+R_{17})$, $\frac{1}{2}(R_{11}+R_{14})$, $(R_{1}+R_{8})$, $(R_{2}+R_{13})$, $\frac{1}{2}(R_{0}+R_{7}+R_{9}+R_{12})$, $\frac{1}{2}(R_{0}+R_{7}+R_{10}+R_{13})$, $\frac{1}{2}(R_{0}+R_{7}+R_{17}+R_{8})$, $\frac{1}{2}(R_{0}+R_{8}+R_{9}+R_{13})$, $\frac{1}{2}(R_{0}+R_{8}+R_{10}+R_{12})$, $\frac{1}{2}(R_{7}+R_{9}+R_{8}+R_{10})$.\\ 
\textbf{(143)} $\mathrm{nd}(Y_{143}) \geq 8$. Sequence: $(R_{0}+R_{5}+R_{10}+R_{8})$, $(R_{1}+R_{6}+R_{4}+R_{9})$, $\frac{1}{2}(R_{0}+R_{1}+R_{6}+R_{4}+R_{3}+R_{2}+R_{10}+R_{8})$, $\frac{1}{2}(R_{0}+R_{1}+R_{6}+R_{4}+R_{3}+R_{7}+R_{10}+R_{5})$, $\frac{1}{2}(R_{0}+R_{1}+R_{9}+R_{4}+R_{3}+R_{2}+R_{10}+R_{5})$, $\frac{1}{2}(R_{0}+R_{1}+R_{9}+R_{4}+R_{3}+R_{7}+R_{10}+R_{8})$, $\frac{1}{2}(R_{2}+R_{7}+R_{6}+R_{9}+2R_{3}+2R_{4})$, $\frac{1}{2}(R_{5}+R_{8}+R_{6}+R_{9}+2R_{0}+2R_{1})$.\\ 
\textbf{(144)} $\mathrm{nd}(Y_{144}) \geq 8$. Sequence: $(R_{0}+R_{4}+R_{10}+R_{6})$, $\frac{1}{2}(R_{0}+R_{4}+R_{3}+R_{9}+R_{2}+R_{8}+R_{1}+R_{6})$, $\frac{1}{2}(R_{0}+R_{4}+R_{3}+R_{9}+R_{7}+R_{8}+R_{11}+R_{6})$, $\frac{1}{2}(R_{1}+R_{6}+R_{10}+R_{4}+R_{3}+R_{9}+R_{7}+R_{8})$, $\frac{1}{2}(R_{2}+R_{8}+R_{11}+R_{6}+R_{10}+R_{4}+R_{3}+R_{9})$, $\frac{1}{2}(R_{0}+R_{1}+R_{10}+R_{11}+2R_{6})$, $\frac{1}{2}(R_{0}+R_{3}+R_{5}+R_{10}+2R_{4})$, $\frac{1}{2}(R_{0}+R_{10}+R_{2}+R_{7}+2R_{4}+2R_{3}+2R_{9})$.\\ 
\textbf{(145)} $\mathrm{nd}(Y_{145}) = 4$. Sequence: $(R_{0}+R_{2}+R_{3}+R_{4}+R_{5}+R_{6}+R_{7}+R_{9})$, $\frac{1}{2}(R_{0}+R_{8}+R_{1}+R_{3}+2R_{9}+2R_{7}+2R_{6}+2R_{5}+2R_{4})$, $\frac{1}{2}(R_{1}+R_{5}+R_{7}+R_{8}+2R_{4}+2R_{3}+2R_{2}+2R_{0}+2R_{9})$, $\frac{1}{2}(2R_{2}+3R_{0}+4R_{9}+3R_{7}+2R_{6}+R_{5}+2R_{8}+R_{3})$.\\ 
\textbf{(154)} $\mathrm{nd}(Y_{154}) = 10$. Sequence: $\frac{1}{2}(R_{1}+R_{2})$, $\frac{1}{2}(R_{1}+R_{11})$, $\frac{1}{2}(R_{1}+R_{25})$, $\frac{1}{2}(R_{1}+R_{10})$, $\frac{1}{2}(R_{1}+R_{15})$, $\frac{1}{2}(R_{1}+R_{28})$, $(R_{1}+R_{18})$, $\frac{1}{2}(R_{0}+R_{20}+R_{21})$, $\frac{1}{2}(R_{3}+R_{8}+R_{26})$, $\frac{1}{2}(R_{3}+R_{16}+R_{19})$.\\ 
\textbf{(155)} $\mathrm{nd}(Y_{155}) = 10$. Sequence: $\frac{1}{2}(R_{0}+R_{1})$, $\frac{1}{2}(R_{0}+R_{9})$, $\frac{1}{2}(R_{0}+R_{11})$, $\frac{1}{2}(R_{0}+R_{8})$, $\frac{1}{2}(R_{0}+R_{23})$, $\frac{1}{2}(R_{0}+R_{26})$, $\frac{1}{2}(R_{2}+R_{6}+R_{24})$, $\frac{1}{2}(R_{2}+R_{12}+R_{16})$, $\frac{1}{2}(R_{6}+R_{10}+R_{16})$, $(R_{2}+R_{6}+R_{16})$.\\ 
\textbf{(157)} $\mathrm{nd}(Y_{157}) = 10$. Sequence: $\frac{1}{2}(R_{4}+R_{8})$, $\frac{1}{2}(R_{4}+R_{15})$, $\frac{1}{2}(R_{5}+R_{9})$, $\frac{1}{2}(R_{5}+R_{14})$, $\frac{1}{2}(R_{4}+R_{23})$, $(R_{4}+R_{7})$, $\frac{1}{2}(R_{0}+R_{17}+R_{3}+R_{18})$, $\frac{1}{2}(R_{0}+R_{20}+R_{1}+R_{21})$, $\frac{1}{2}(R_{1}+R_{13}+R_{18}+R_{20})$, $\frac{1}{2}(R_{3}+R_{13}+R_{21}+R_{17})$.\\ 
\textbf{(158)} $\mathrm{nd}(Y_{158}) \geq 9$. Sequence: $\frac{1}{2}(R_{0}+R_{2})$, $\frac{1}{2}(R_{2}+R_{16})$, $\frac{1}{2}(R_{2}+R_{2} \cdot H_{2})$, $\frac{1}{2}(R_{3}+R_{4})$, $\frac{1}{2}(R_{3}+R_{12})$, $\frac{1}{2}(R_{2}+R_{14})$, $(R_{2}+R_{8})$, $\frac{1}{2}(R_{1}+R_{6}+R_{8}+R_{15}+2R_{7})$, $\frac{1}{2}(R_{5}+R_{6}+R_{8}+R_{9}+2R_{11})$.\\ 
\textbf{(159)} $\mathrm{nd}(Y_{159}) \geq 7$. Sequence: $\frac{1}{2}(R_{0}+R_{13})$, $\frac{1}{2}(R_{3}+R_{4})$, $\frac{1}{2}(R_{3}+R_{11})$, $\frac{1}{2}(R_{3}+R_{12})$, $(R_{3}+R_{14})$, $\frac{1}{2}(R_{1}+R_{6}+R_{7}+R_{2}+R_{14}+R_{9}+R_{5}+R_{15})$, $\frac{1}{2}(R_{1}+R_{8}+R_{14}+R_{2}+R_{7}+R_{10}+R_{5}+R_{15})$.\\ 
\textbf{(161)} $\mathrm{nd}(Y_{161}) = 10$. Sequence: $\frac{1}{2}(R_{0}+R_{39})$, $\frac{1}{2}(R_{0}+R_{47})$, $\frac{1}{2}(R_{0}+R_{53})$, $\frac{1}{2}(R_{1}+R_{37})$, $\frac{1}{2}(R_{1}+R_{49})$, $\frac{1}{2}(R_{2}+R_{48})$, $\frac{1}{2}(R_{2}+R_{52})$, $\frac{1}{2}(R_{3}+R_{55})$, $\frac{1}{2}(R_{4}+R_{46})$, $\frac{1}{2}(R_{0}+R_{2}+R_{7})$.\\ 
\textbf{(162)} $\mathrm{nd}(Y_{162}) = 10$. Sequence: $\frac{1}{2}(R_{2}+R_{8})$, $\frac{1}{2}(R_{2}+R_{18})$, $\frac{1}{2}(R_{5}+R_{13})$, $\frac{1}{2}(R_{5}+R_{19})$, $\frac{1}{2}(R_{6}+R_{27})$, $\frac{1}{2}(R_{11}+R_{27})$, $\frac{1}{2}(R_{2}+R_{26})$, $\frac{1}{2}(R_{4}+R_{12})$, $\frac{1}{2}(R_{12}+R_{20})$, $(R_{2}+R_{17})$.\\ 
\textbf{(163)} $\mathrm{nd}(Y_{163}) = 10$. Sequence: $\frac{1}{2}(R_{1}+R_{11})$, $\frac{1}{2}(R_{16}+R_{19})$, $\frac{1}{2}(R_{17}+R_{18})$, $\frac{1}{2}(R_{2}+R_{3})$, $\frac{1}{2}(R_{2}+R_{4})$, $\frac{1}{2}(R_{2}+R_{23})$, $\frac{1}{2}(R_{2}+R_{15})$, $\frac{1}{2}(R_{2}+R_{20})$, $\frac{1}{2}(R_{2}+R_{25})$, $\frac{1}{2}(R_{2}+R_{22})$.\\ 
\textbf{(164)} $\mathrm{nd}(Y_{164}) = 10$. Sequence: $(R_{0}+R_{8})$, $(R_{1}+R_{13})$, $(R_{7}+R_{9})$, $(R_{3}+R_{11})$, $\frac{1}{2}(R_{1}+R_{6}+R_{2}+R_{12}+R_{8}+R_{11})$, $\frac{1}{2}(R_{1}+R_{6}+R_{5}+R_{4}+R_{7}+R_{11})$, $\frac{1}{2}(R_{1}+R_{10}+R_{4}+R_{5}+R_{8}+R_{11})$, $\frac{1}{2}(R_{1}+R_{10}+R_{12}+R_{2}+R_{7}+R_{11})$, $\frac{1}{2}(R_{2}+R_{6}+R_{5}+R_{8}+R_{11}+R_{7})$, $\frac{1}{2}(R_{4}+R_{7}+R_{11}+R_{8}+R_{12}+R_{10})$.\\ 
\textbf{(165)} $\mathrm{nd}(Y_{165}) = 10$. Sequence: $\frac{1}{2}(R_{1}+R_{2})$, $\frac{1}{2}(R_{1}+R_{9})$, $\frac{1}{2}(R_{1}+R_{24})$, $\frac{1}{2}(R_{1}+R_{26})$, $\frac{1}{2}(R_{1}+R_{10})$, $\frac{1}{2}(R_{1}+R_{30})$, $(R_{1}+R_{22})$, $\frac{1}{2}(R_{0}+R_{22}+R_{23})$, $\frac{1}{2}(R_{3}+R_{8}+R_{29})$, $\frac{1}{2}(R_{3}+R_{7}+R_{28})$.\\ 
\textbf{(166)} $\mathrm{nd}(Y_{166}) = 10$. Sequence: $\frac{1}{2}(R_{0}+R_{24})$, $\frac{1}{2}(R_{14}+R_{21})$, $\frac{1}{2}(R_{16}+R_{19})$, $\frac{1}{2}(R_{2}+R_{3})$, $\frac{1}{2}(R_{2}+R_{13})$, $\frac{1}{2}(R_{2}+R_{22})$, $\frac{1}{2}(R_{2}+R_{27})$, $\frac{1}{2}(R_{2}+R_{4})$, $\frac{1}{2}(R_{2}+R_{26})$, $\frac{1}{2}(R_{2}+R_{25})$.\\ 
\textbf{(167)} $\mathrm{nd}(Y_{167}) = 10$. Sequence: $\frac{1}{2}(R_{1}+R_{2})$, $\frac{1}{2}(R_{1}+R_{10})$, $\frac{1}{2}(R_{1}+R_{15})$, $\frac{1}{2}(R_{1}+R_{24})$, $\frac{1}{2}(R_{1}+R_{13})$, $\frac{1}{2}(R_{1}+R_{29})$, $(R_{1}+R_{14})$, $\frac{1}{2}(R_{0}+R_{7}+R_{19})$, $\frac{1}{2}(R_{3}+R_{14}+R_{18})$, $\frac{1}{2}(R_{5}+R_{8}+R_{28})$.\\ 
\textbf{(168)} $\mathrm{nd}(Y_{168}) = 10$. Sequence: $\frac{1}{2}(R_{2}+R_{10})$, $\frac{1}{2}(R_{2}+R_{14})$, $\frac{1}{2}(R_{2}+R_{12})$, $\frac{1}{2}(R_{2}+R_{13})$, $\frac{1}{2}(R_{0}+R_{4}+R_{3}+R_{1}+R_{16})$, $\frac{1}{2}(R_{0}+R_{4}+R_{11}+R_{6}+R_{5})$, $\frac{1}{2}(R_{1}+R_{3}+R_{11}+R_{8}+R_{7})$, $\frac{1}{2}(R_{1}+R_{15}+R_{9}+R_{5}+R_{16})$, $\frac{1}{2}(R_{3}+R_{11}+R_{6}+R_{9}+R_{15})$, $\frac{1}{2}(R_{5}+R_{6}+R_{8}+R_{7}+R_{16})$.\\ 
\textbf{(169)} $\mathrm{nd}(Y_{169}) = 10$. Sequence: $\frac{1}{2}(R_{1}+R_{10})$, $\frac{1}{2}(R_{1}+R_{12})$, $\frac{1}{2}(R_{1}+R_{13})$, $\frac{1}{2}(R_{1}+R_{14})$, $\frac{1}{2}(R_{0}+R_{4}+R_{11}+R_{6}+R_{5})$, $\frac{1}{2}(R_{0}+R_{5}+R_{9}+R_{2}+R_{16})$, $\frac{1}{2}(R_{2}+R_{8}+R_{7}+R_{6}+R_{9})$, $\frac{1}{2}(R_{2}+R_{15}+R_{3}+R_{4}+R_{16})$, $\frac{1}{2}(R_{3}+R_{11}+R_{6}+R_{9}+R_{15})$, $\frac{1}{2}(R_{4}+R_{11}+R_{7}+R_{8}+R_{16})$.\\ 
\textbf{(171)} $\mathrm{nd}(Y_{171}) \geq 8$. Sequence: $\frac{1}{2}(R_{1}+R_{3})$, $\frac{1}{2}(R_{1}+R_{12})$, $\frac{1}{2}(R_{1}+R_{13})$, $(R_{1}+R_{5})$, $\frac{1}{2}(R_{0}+R_{6}+R_{5}+R_{11}+R_{4}+R_{10}+R_{2}+R_{8})$, $\frac{1}{2}(R_{2}+R_{8}+R_{14}+R_{6}+R_{5}+R_{11}+R_{9}+R_{10})$, $\frac{1}{2}(R_{0}+R_{2}+R_{14}+R_{15}+2R_{8})$, $\frac{1}{2}(R_{0}+R_{5}+R_{7}+R_{14}+2R_{6})$.\\ 
\textbf{(172)} $\mathrm{nd}(Y_{172}) = 4$. Sequence: $\frac{1}{2}(R_{1}+R_{2})$, $(R_{1}+R_{10})$, $\frac{1}{2}(R_{0}+R_{10}+R_{3}+R_{5}+2R_{11}+2R_{9}+2R_{8}+2R_{7}+2R_{6})$, $\frac{1}{2}(R_{3}+R_{7}+R_{9}+R_{10}+2R_{6}+2R_{5}+2R_{4}+2R_{0}+2R_{11})$.\\ 
\textbf{(173)} $\mathrm{nd}(Y_{173}) = 10$. Sequence: $\frac{1}{2}(R_{0}+R_{37})$, $\frac{1}{2}(R_{0}+R_{38})$, $\frac{1}{2}(R_{0}+R_{52})$, $\frac{1}{2}(R_{1}+R_{43})$, $\frac{1}{2}(R_{1}+R_{49})$, $\frac{1}{2}(R_{2}+R_{45})$, $\frac{1}{2}(R_{3}+R_{50})$, $\frac{1}{2}(R_{8}+R_{57})$, $\frac{1}{2}(R_{13}+R_{58})$, $\frac{1}{2}(R_{0}+R_{10}+R_{13})$.\\ 
\textbf{(174)} $\mathrm{nd}(Y_{174}) = 10$. Sequence: $(R_{0}+R_{19})$, $(R_{2}+R_{14})$, $(R_{3}+R_{6})$, $(R_{5}+R_{15})$, $\frac{1}{2}(R_{0}+R_{13}+R_{14})$, $\frac{1}{2}(R_{3}+R_{13}+R_{15})$, $\frac{1}{2}(R_{1}+R_{16}+R_{17})$, $(R_{1}+R_{4}+R_{17})$, $\frac{1}{2}(R_{0}+R_{15}+R_{1}+R_{17})$, $\frac{1}{2}(R_{1}+R_{3}+R_{14}+R_{17})$.\\ 
\textbf{(175)} $\mathrm{nd}(Y_{175}) = 10$. Sequence: $\frac{1}{2}(R_{3}+R_{8})$, $\frac{1}{2}(R_{3}+R_{17})$, $\frac{1}{2}(R_{3}+R_{26})$, $\frac{1}{2}(R_{5}+R_{12})$, $\frac{1}{2}(R_{5}+R_{18})$, $\frac{1}{2}(R_{6}+R_{11})$, $\frac{1}{2}(R_{6}+R_{27})$, $\frac{1}{2}(R_{11}+R_{27})$, $\frac{1}{2}(R_{12}+R_{18})$, $(R_{3}+R_{16})$.\\ 
\textbf{(176)} $\mathrm{nd}(Y_{176}) \geq 7$. Sequence: $\frac{1}{2}(R_{0}+R_{13})$, $\frac{1}{2}(R_{3}+R_{12})$, $\frac{1}{2}(R_{7}+R_{8})$, $(R_{0}+R_{15})$, $(R_{3}+R_{10})$, $(R_{7}+R_{11})$, $\frac{1}{2}(R_{0}+R_{3}+R_{7})$.\\ 
\textbf{(178)} $\mathrm{nd}(Y_{178}) = 10$. Sequence: $\frac{1}{2}(R_{0}+R_{22})$, $\frac{1}{2}(R_{13}+R_{20})$, $\frac{1}{2}(R_{15}+R_{18})$, $(R_{0}+R_{21})$, $(R_{1}+R_{19})$, $(R_{6}+R_{17})$, $\frac{1}{2}(R_{0}+R_{13}+R_{17})$, $\frac{1}{2}(R_{0}+R_{15}+R_{19})$, $\frac{1}{2}(R_{2}+R_{13}+R_{15})$, $\frac{1}{2}(R_{2}+R_{17}+R_{19})$.\\ 
\textbf{(179)} $\mathrm{nd}(Y_{179}) = 10$. Sequence: $(R_{0}+R_{13})$, $(R_{1}+R_{15})$, $(R_{2}+R_{14})$, $(R_{3}+R_{12})$, $(R_{4}+R_{9})$, $(R_{5}+R_{8})$, $\frac{1}{2}(R_{0}+R_{3}+R_{8})$, $\frac{1}{2}(R_{0}+R_{9}+R_{15})$, $\frac{1}{2}(R_{3}+R_{9}+R_{14})$, $\frac{1}{2}(R_{8}+R_{14}+R_{15})$.\\ 
\textbf{(180)} $\mathrm{nd}(Y_{180}) = 10$. Sequence: $(R_{0}+R_{22})$, $(R_{1}+R_{14})$, $(R_{2}+R_{21})$, $(R_{4}+R_{16})$, $(R_{6}+R_{12})$, $(R_{7}+R_{19})$, $\frac{1}{2}(R_{0}+R_{18}+R_{19})$, $\frac{1}{2}(R_{1}+R_{18}+R_{21})$, $\frac{1}{2}(R_{2}+R_{7}+R_{23})$, $\frac{1}{2}(R_{2}+R_{7}+R_{4}+R_{14}+R_{22}+R_{12})$.\\ 
\textbf{(181)} $\mathrm{nd}(Y_{181}) = 10$. Sequence: $(R_{0}+R_{14})$, $(R_{2}+R_{7})$, $(R_{3}+R_{9})$, $(R_{4}+R_{13})$, $(R_{5}+R_{10})$, $\frac{1}{2}(R_{0}+R_{10}+R_{8}+R_{1}+R_{12})$, $\frac{1}{2}(R_{0}+R_{10}+R_{8}+R_{7}+R_{13})$, $\frac{1}{2}(R_{0}+R_{10}+R_{9}+R_{1}+R_{13})$, $\frac{1}{2}(R_{0}+R_{11}+R_{8}+R_{1}+R_{13})$, $\frac{1}{2}(R_{1}+R_{8}+R_{10}+R_{6}+R_{13})$.\\ 
\textbf{(182)} $\mathrm{nd}(Y_{182}) = 10$. Sequence: $(R_{0}+R_{4}+R_{13})$, $(R_{1}+R_{3}+R_{12})$, $(R_{5}+R_{6}+R_{10})$, $(R_{7}+R_{8}+R_{9})$, $\frac{1}{2}(R_{0}+R_{4}+R_{3}+R_{1}+R_{14})$, $\frac{1}{2}(R_{0}+R_{4}+R_{11}+R_{6}+R_{5})$, $\frac{1}{2}(R_{0}+R_{5}+R_{10}+R_{2}+R_{13})$, $\frac{1}{2}(R_{0}+R_{13}+R_{9}+R_{7}+R_{14})$, $\frac{1}{2}(R_{1}+R_{12}+R_{10}+R_{5}+R_{14})$, $\frac{1}{2}(R_{2}+R_{12}+R_{3}+R_{4}+R_{13})$.\\ 
\textbf{(183)} $\mathrm{nd}(Y_{183}) = 10$. Sequence: $(R_{0}+R_{15})$, $(R_{2}+R_{10})$, $(R_{3}+R_{14})$, $\frac{1}{2}(R_{0}+R_{11}+R_{5}+R_{12})$, $\frac{1}{2}(R_{1}+R_{13}+R_{8}+R_{14})$, $\frac{1}{2}(R_{0}+R_{11}+R_{9}+R_{1}+R_{13})$, $\frac{1}{2}(R_{0}+R_{11}+R_{9}+R_{8}+R_{14})$, $\frac{1}{2}(R_{0}+R_{11}+R_{10}+R_{1}+R_{14})$, $\frac{1}{2}(R_{0}+R_{12}+R_{9}+R_{1}+R_{14})$, $\frac{1}{2}(R_{1}+R_{9}+R_{11}+R_{5}+R_{14})$.\\ 
\textbf{(184)} $\mathrm{nd}(Y_{184}) = 7$. Sequence: $(R_{0}+R_{3}+R_{4}+R_{8})$, $(R_{1}+R_{6}+R_{10}+R_{9})$, $(R_{2}+R_{7}+R_{5}+R_{11})$, $\frac{1}{2}(R_{0}+R_{3}+R_{4}+R_{5}+R_{7}+R_{2}+R_{1}+R_{9}+R_{10})$, $\frac{1}{2}(R_{0}+R_{3}+R_{4}+R_{5}+R_{11}+R_{2}+R_{1}+R_{6}+R_{10})$, $\frac{1}{2}(R_{0}+R_{8}+R_{4}+R_{5}+R_{7}+R_{2}+R_{1}+R_{6}+R_{10})$, $\frac{1}{2}(R_{0}+R_{8}+R_{4}+R_{5}+R_{11}+R_{2}+R_{1}+R_{9}+R_{10})$.
}


\begin{thebibliography}{BHPV04}

\bibitem[BHPV04]{BHPV04}
Wolf~P. Barth, Klaus Hulek, Chris A.~M. Peters, and Antonius Van~de Ven.
\newblock {\em Compact complex surfaces}, volume~4 of {\em Ergebnisse der
  Mathematik und ihrer Grenzgebiete. 3. Folge. A Series of Modern Surveys in
  Mathematics}.
\newblock Springer-Verlag, Berlin, second edition, 2004.

\bibitem[BM22]{BM22}
Simon Brandhorst and Giacomo Mezzedimi.
\newblock Borcherds lattices and {K}3 surfaces of zero entropy.
\newblock {\em arXiv e-prints}, page arXiv:2211.09600, March 2022.
\newblock \url{https://arxiv.org/abs/2211.09600}.

\bibitem[BP83]{BP83}
Wolf~P. Barth and Chris A.~M. Peters.
\newblock Automorphisms of {E}nriques surfaces.
\newblock {\em Invent. Math.}, 73(3):383--411, 1983.

\bibitem[BS22]{BS22}
Simon Brandhorst and Ichiro Shimada.
\newblock Automorphism groups of certain {E}nriques surfaces.
\newblock {\em Found. Comput. Math.}, 22(5):1463--1512, 2022.

\bibitem[CD89]{CD89}
Fran\c{c}ois Cossec and Igor Dolgachev.
\newblock {\em Enriques surfaces. {I}}, volume~76 of {\em Progress in
  Mathematics}.
\newblock Birkh\"{a}user Boston, Inc., Boston, MA, 1989.

\bibitem[CDL24]{CDL24}
Fran\c{c}ois Cossec, Igor Dolgachev, and Christian Liedtke.
\newblock Enriques surfaces. {I}, 2024.
\newblock \url{http://www.math.lsa.umich.edu/~idolga/EnriquesOne.pdf} (version
  of April 19, 2024).

\bibitem[Cos83]{Cos83}
Fran\c{c}ois Cossec.
\newblock Reye congruences.
\newblock {\em Trans. Amer. Math. Soc.}, 280(2):737--751, 1983.

\bibitem[Cos85]{Cos85}
Fran\c{c}ois Cossec.
\newblock On the {P}icard group of {E}nriques surfaces.
\newblock {\em Math. Ann.}, 271(4):577--600, 1985.

\bibitem[DK24]{DK24}
Igor Dolgachev and Shigeyuki Kond\=o.
\newblock Enriques surfaces. {II}, 2024.
\newblock \url{http://www.math.lsa.umich.edu/~idolga/EnriquesTwo.pdf} (version
  of April 19, 2024).

\bibitem[DM19]{DM19}
Igor {Dolgachev} and Dimitri {Markushevich}.
\newblock {Lagrangian tens of planes, Enriques surfaces and holomorphic
  symplectic fourfolds}.
\newblock {\em arXiv e-prints}, page arXiv:1906.01445, June 2019.

\bibitem[Dol18]{Dol18}
Igor Dolgachev.
\newblock Salem numbers and {E}nriques surfaces.
\newblock {\em Exp. Math.}, 27(3):287--301, 2018.

\bibitem[Kon86]{Kon86}
Shigeyuki Kond\=o.
\newblock Enriques surfaces with finite automorphism groups.
\newblock {\em Japan. J. Math. (N.S.)}, 12(2):191--282, 1986.

\bibitem[Mil58]{Mil58}
John Milnor.
\newblock On simply connected {$4$}-manifolds.
\newblock In {\em Symposium internacional de topolog\'{\i}a algebraica
  {I}nternational symposium on algebraic topology}, pages 122--128. Universidad
  Nacional Aut\'{o}noma de M\'{e}xico and UNESCO, Mexico City, 1958.

\bibitem[MMV22]{MMV22a}
Gebhard {Martin}, Giacomo {Mezzedimi}, and Davide~Cesare {Veniani}.
\newblock {Enriques surfaces of non-degeneracy 3}.
\newblock {\em arXiv e-prints}, page arXiv:2203.08000, March 2022.
\newblock \url{https://arxiv.org/abs/2203.08000}.

\bibitem[MMV23]{MMV22b}
Gebhard Martin, Giacomo Mezzedimi, and Davide~Cesare Veniani.
\newblock On extra-special {E}nriques surfaces.
\newblock {\em Math. Ann.}, 387(1-2):133--143, 2023.

\bibitem[MRS22a]{MRS22Paper}
Riccardo Moschetti, Franco Rota, and Luca Schaffler.
\newblock A computational view on the non-degeneracy invariant for Enriques
  surfaces.
\newblock {\em Experimental Mathematics}, 0(0):1--22, 2022.

\bibitem[MRS22b]{MRS22Code}
Riccardo Moschetti, Franco Rota, and Luca Schaffler.
\newblock Sage{M}ath code {CND}{F}inder, 2022.
\newblock \url{https://github.com/rmoschetti/CNDFinder}.

\bibitem[MRS23]{MRS23Code}
Riccardo Moschetti, Franco Rota, and Luca Schaffler.
\newblock Check isotropic sequences, 2023.
\newblock \url{https://github.com/rmoschetti/CND-TauTauBar-Enriques}.

\bibitem[SB20a]{SB20data}
Ichiro Shimada and Simon Brandhorst.
\newblock Automorphism groups of certain {E}nriques surfaces, 2020.
\newblock Zenodo. \url{https://doi.org/10.5281/zenodo.4327019}. File:
  compdata.zip.

\bibitem[SB20b]{SB20explanation}
Ichiro Shimada and Simon Brandhorst.
\newblock Automorphism groups of certain {E}nriques surfaces, 2020.
\newblock Zenodo. \url{https://doi.org/10.5281/zenodo.4327019}. File:
  AutEnrVolCompData.pdf.

\bibitem[SB20c]{SB20ellipticfibrations}
Ichiro Shimada and Simon Brandhorst.
\newblock Automorphism groups of certain {E}nriques surfaces, 2020.
\newblock Zenodo. \url{https://doi.org/10.5281/zenodo.4327019}. File:
  EllFibs.pdf.

\bibitem[Ser73]{Ser73}
Jean-Pierre Serre.
\newblock {\em A course in arithmetic}.
\newblock Graduate Texts in Mathematics, No. 7. Springer-Verlag, New
  York-Heidelberg, 1973.
\newblock Translated from the French.

\bibitem[Shi21]{Shi21}
Ichiro Shimada.
\newblock Rational double points on {E}nriques surfaces.
\newblock {\em Sci. China Math.}, 64(4):665--690, 2021.

\end{thebibliography}
\end{document}